\documentclass[11pt,twoside,]{amsart}
\usepackage[dvips]{graphicx}
\usepackage{tikz,tkz-berge,tkz-graph}
\usepackage{longtable}
\usepackage{amsmath, amsthm, amscd, amsfonts, amssymb, color, comment,verbatim}
\usepackage[bookmarksnumbered, plainpages]{hyperref}
 \usepackage[usenames,dvipsnames]{pstricks}
 \usepackage{epsfig}
 \usepackage{pst-grad} % For gradients
 \usepackage{pst-plot} % For axes
\newtheorem{thm}{Theorem}[section]
\newtheorem{cor}[thm]{Corollary}
\newtheorem{lem}[thm]{Lemma}
\newtheorem{prop}[thm]{Proposition}
\newtheorem{exm}[thm]{Example}
\newtheorem{defn}[thm]{Definition}
\newtheorem{conj}[thm]{Conjecture}
\newtheorem{rem}[thm]{Remark}

\newtheorem{qu}[thm]{Question}
\numberwithin{equation}{section}
\begin{document}

\oddsidemargin 0mm
\evensidemargin 0mm

\thispagestyle{plain}

\vspace{5cc}
\begin{center}

% Title and Names ----------------------------------------------------------------------------------------------------------

{\large\bf Zero divisors and units with small supports in group algebras \\of torsion-free groups}
\rule{0mm}{6mm}\renewcommand{\thefootnote}{}
\footnotetext{
{\rule{2.4mm}{0mm}Keywords and Phrases.  Kaplansky's zero divisor conjecture, Kaplansky's unit conjecture, group ring, torsion-free group, zero divisor.}}

\vspace{1cc}
{\large\it Alireza Abdollahi and Zahra Taheri}

% Abstract --------------------------------------------------------------------------------------------------------------------

\vspace{1cc}
\parbox{27cc}{{\small

\textbf{Abstract.} Kaplansky's zero divisor conjecture (unit conjecture, respectively)  states that for a torsion-free group $G$ and a field $\mathbb{F}$, the group ring $\mathbb{F}[G]$ has no zero divisors (has no units with supports of size greater than $1$).  In this paper, we study  possible zero divisors and units in $\mathbb{F}[G]$ whose supports have size $3$.
For any field $\mathbb{F}$ and all torsion-free groups $G$, we prove that if $\alpha \beta=0$ for some non-zero $\alpha, \beta \in \mathbb{F}[G]$ such that $|supp(\alpha)|=3$, then $|supp(\beta)|\geq 10$. If $\mathbb{F}=\mathbb{F}_2$ is the field with 2 elements, the latter result can be improved so that $|supp(\beta)|\geq 20$. This improves a result in [J. Group Theory, 16 (2013), no. 5, 667-693]. 
Concerning the unit conjecture, we  prove that if $\alpha \beta=1$ for some $\alpha, \beta \in \mathbb{F}[G]$ such that $|supp(\alpha)|=3$, then $|supp(\beta)|\geq 9$. The latter improves a part of a result in  [Exp. Math., 24 (2015), 326-338] to arbitrary fields. }}

\end{center}

\vspace{1cc}

% Introduction ---------------------------------------------------------------------------------------------------------------

\section{\bf Introduction and Results}\label{Int}
Let $R$ be a ring. A non-zero element $\alpha$ of  $R$ is called a zero divisor if $\alpha\beta= 0$ or $\beta\alpha = 0$ for
some non-zero element $\beta \in R$. Let $G$ be a group. Denote by $R[G]$ the group ring of $G$ over $R$. If $R$ contains a zero divisor, then clearly so does $R[G]$. Also, if $G$ contains a non-identity torsion element $x$ of finite order $n$, then $R[G]$ contains zero divisors $\alpha=1-x$ and $\beta=1+x+\cdots+x^{n-1}$, since $\alpha \beta=0$.  Around $1950$, Irving Kaplansky conjectured that existence of a zero divisor in a group ring depends only on the existence of such elements in the ring or non-trivial torsions in the group by stating one of the most challenging problems in the field of group rings \cite{kap}.
\begin{conj}[Kaplansky's zero divisor conjecture] \label{conj-zero} Let $\mathbb{F}$ be a field and $G$ be a torsion-free group. Then $\mathbb{F}[G]$ does not contain a zero divisor.
\end{conj}

Another famous problem, namely the unit conjecture, also proposed by Kaplansky \cite{kap}, states that:
\begin{conj}[Kaplansky's unit conjecture] \label{conj-unit} Let $\mathbb{F}$ be a field and $G$ be a torsion-free group. Then $\mathbb{F}[G]$ has
no non-trivial units (i.e., units which are not non-zero scalar multiples of group elements).
\end{conj}
It can be shown that the zero divisor conjecture is true if the unit conjecture has an affirmative solution (see \cite[Lemma 13.1.2]{pass1}).

Over the years, some partial results have been obtained on Conjecture \ref{conj-zero} and it has been confirmed for special classes of groups which are torsion-free. One of the first known special families which satisfy Conjecture \ref{conj-zero} are unique product groups \cite[Chapter 13]{pass1}, in particular ordered groups. Furthermore, by the fact that Conjecture \ref{conj-zero} is known to hold valid for amalgamated free products when the group ring of the subgroup over which the amalgam is formed satisfies the Ore condition \cite{lewin}, it is proved by Formanek \cite{form} that supersolvable groups are another families which satisfy Conjecture \ref{conj-zero}. Another result, concerning large major sorts of groups for which Conjecture \ref{conj-zero} holds in the affirmative, is obtained for elementary amenable groups \cite{krop}. The latter result covers the cases in which the group is polycyclic-by-finite, which was firstly studied in \cite{brown} and \cite{farkas}, and then extended in \cite{sni}. Some other affirmative results are obtained on congruence subgroups in \cite{laz} and \cite{farkas2}, and certain hyperbolic groups \cite{del}. 
%one-relator groups \cite{lewin2}; 
Nevertheless, Conjecture \ref{conj-zero} has not been confirmed for any fixed field and it seems that confirming the conjecture even for the smallest finite field $\mathbb{F}_2$  with two elements is still out of reach.

The support of an element $\alpha=\sum_{x\in G}{a_x x}$ of $R[G]$, denoted by $supp(\alpha)$, is the set $\{x \in G \mid a_x\neq 0\}$. For any field $\mathbb{F}$ and any torsion-free group $G$, it is known that $\mathbb{F}[G]$ does not contain a zero divisor whose support is of size at most $2$ (see \cite[Theorem 2.1]{pascal}), but it is not  known a similar result for group algebra elements with the support of size $3$. By describing a combinatorial structure, named matched rectangles, Schweitzer \cite{pascal} showed that if $\alpha\beta=0$ for $\alpha,\beta\in \mathbb{F}_2[G]\setminus \{0\}$ when $\vert supp(\alpha)\vert=3$, then $\vert supp(\beta)\vert>6$. Also, with a computer-assisted approach, he showed that if $\vert supp(\alpha)\vert=3$, then $\vert supp(\beta)\vert>16$. In the following we provide a little more details of graph theory which are used in this paper. 

A graph $\mathcal{G}=(\mathcal{V_G},\mathcal{E_G}, \psi_{\mathcal{G}})$ consists of a non-empty set $\mathcal{V_G}$, a possibly empty set $\mathcal{E_G}$ and if $\mathcal{E_G}\neq \varnothing$ a function $\psi_{\mathcal{G}}: \mathcal{E_G}\rightarrow \mathcal{V_G}^{(1)} \cup \mathcal{V_G}^{(2)} \cup \mathcal{V_G} \times \mathcal{V_G}$, where $\mathcal{V_G}^{(i)}$ denotes the set of all $i$-element subsets of $\mathcal{V_G}$ for $i=1,2$. The elements of $\mathcal{V_G}$ and $\mathcal{E_G}$ are called vertices and edges of the graph $\mathcal{G}$, respectively. An edge $e\in \mathcal{E}_G$ is called an undirected loop if $\psi_{\mathcal{G}}(e)\in \mathcal{V_G}^{(1)}$. The edge $e$ is called directed if $\psi_{\mathcal{G}}(e)\in \mathcal{V_G} \times \mathcal{V_G}$ and it is called a directed loop if $ \psi_{\mathcal{G}}(e)=(v,v)$ for some $v\in \mathcal{V_G}$. The graph $\mathcal{G}$ is called undirected if $ \psi_{\mathcal{G}}(\mathcal{E_G}) \cap \left( \mathcal{V_G} \times \mathcal{V_G} \right) =\varnothing$.
  We say a vertex $u$ is adjacent to a vertex $v$, denoted by $u\sim v$, if $\psi_{\mathcal{G}}(e)=\{u,v\}$, $(u,v)$ or $(v,u)$ for some $e\in \mathcal{E_G}$; otherwise, we say $u$ is not adjacent to $v$, denoted by $u\not\sim v$. In the latter case the vertices $u$ and $v$ are called the endpoints of the edge $e$ and we say $e$ joins its endpoints. If a vertex $v$ is an endpoint of an edge $e$, we say $v$ is adjacent to $e$ or also $e$ is adjacent to $v$. Two edges $e_1$ and $e_2$ are called adjacent , denoted by $e_1 \sim e_2$, if the sets of their endpoints have non-empty intersection.  The graph $\mathcal{G}$ is called loopless whenever  $\psi_{\mathcal{G}}(\mathcal{E_G}) \cap \mathcal{V_G}^{(1)}= \varnothing$ and $\psi_{\mathcal{G}}(\mathcal{E_G}) \cap \{(v,v)|v\in\mathcal{V_G}\}= \varnothing$. We say that the graph $\mathcal{G}$ has multi-edge if $\psi_{\mathcal{G}}$ is {\bf not} injective. A simple graph is an undirected loopless graph having no multi-edge. So, the graph $\mathcal{G}$ is simple if $\mathcal{E_G}$ is empty or $\psi_{\mathcal{G}}$ is an injective function from $\mathcal{V_G}$ to $\mathcal{V_G}^{(2)}$.  The degree of a vertex $v$ of $\mathcal{G}$, denoted by ${\rm deg}(v)$, is the number of edges adjacent to $v$. If all vertices of a graph have the same degree $k$, we say that the graph is $k$-regular and as a special case, a cubic graph is the one which is $3$-regular. A path graph $\mathcal{P}$ is a simple graph with the vertex set $\{v_0, v_1,\ldots , v_n\}$ such that $v_{j-1}\sim v_{j}$ for $j = 1,\ldots , n$. The length of a path is the number of its edges. A cycle of length $n$, denoted by $C_n$, is a path of length $n$ in which its initial vertices ($v_0$ and $v_n$ by the above notation) are identified. The cycles $C_3$ and $C_4$ are called triangle and square, respectively. We say an undirected graph is connected if there is a path between every pair of distinct vertices. A subgraph of a graph $\mathcal{G}$ is a graph $\mathcal{H}$ such that $\mathcal{V_H}\subseteq \mathcal{V_G}$, $\mathcal{E_H}\subseteq \mathcal{E_G}$ and $\psi_{\mathcal{H}}$ is the restriction of $\psi_{\mathcal{G}}$ to $\mathcal{E_H}$. In a graph $\mathcal{G}$, an induced subgraph $\mathcal{I}$ on a set of vertices $W\subseteq \mathcal{V_G}$ is a subgraph in which $\mathcal{V_I}=W$ and $\mathcal{E_I}=\{e \in \mathcal{E_G} | \text{ the endpoints of } e \text{ are in } W \}$.  An isomorphism between two undirected and loopless graphs $\mathcal{G}$ and $\mathcal{H}$ is a pair of bijections $\phi_V: \mathcal{V_G} \rightarrow\mathcal{V_H}$ and $\phi_E: \mathcal{E_G} \rightarrow\mathcal{E_H}$ preserving adjacency and non-adjacency i.e.,  for any pair of vertices $u,v\in \mathcal{V_G}$, $\phi_V(u) \sim \phi_V(v) \Leftrightarrow u \sim v$ and  for any pair of edges $e_1,e_2 \in \mathcal{E_G}$, $\phi_E(e_1) \sim \phi_E(e_2) \Leftrightarrow e_1 \sim e_2$.  Two undirected and loopless graphs $\mathcal{G}$ and $\mathcal{H}$ are called isomorphic (denoted by $\mathcal{G}\cong\mathcal{H}$) if there is an isomorphism between them. We say that an undirected and loopless graph $\mathcal{H}$ is a forbidden subgraph of a graph $\mathcal{G}$ if there is no subgraph isomorphic to $\mathcal{H}$ in $\mathcal{G}$. An undirected graph is bipartite if its vertices can be partitioned into two sets (called partite sets) in such a way that no edge joins two vertices in the same set. A complete bipartite graph is a simple bipartite graph in which each vertex in one partite set is adjacent to all the vertices in the other one. A complete bipartite graph whose  partite sets have cardinalities $r$ and $s$ is denoted by $K_{r,s}$. A Cayley graph $Cay(G,S)$ for a group $G$ and a subset $S$ of $G$ with $1\not\in S=S^{-1}$, is the graph whose vertex set is $G$ and two vertices $g$ and $h$ are adjacent if $gh^{-1} \in S$.

Let $G$ be an arbitrary torsion-free group and let $\alpha \in \mathbb{F}[G]$ be a possible zero divisor such that  $\vert supp(\alpha)\vert=3$ and $\alpha \beta=0$ for some non-zero $\beta \in \mathbb{F}[G]$.  In this paper, we study the minimum possible size of the support of such an element  $\beta$. Suppose that $\mathbb{F}=\mathbb{F}_2$. In \cite[Definition 4.1]{pascal} a graph $K(M)$ is associated to the non-degenerate $3\times |supp(\beta)|$ matched rectangle $M$ corresponding to $\alpha$ and $\beta$ and it is proved in \cite[Theorem 4.2]{pascal} that the graph $K(M)$ is a simple cubic one without triangles. Here we define a graph $Z(a,b)$ associated to any pair $(a,b)$ of non-zero elements of a group algebra (not necessarily over a field with $2$ elements and not necessarily on a torsion-free group) such that $ab=0$ (see Definition \ref{def-Zgraph}). We call  $Z(a,b)$ the zero-divisor graph of $(a,b)$. The zero divisor graph $Z(\alpha,\beta)$ is isomorphic  to $K(M)$. Here we study subgraphs of some zero-divisor graphs. Our main results on Conjecture \ref{conj-zero} are the followings.

\begin{thm}\label{thm-F}
None of the graphs in Figure \ref{forbiddens-F} can be isomorphic to a subgraph of any zero-divisor graph of length $3$ over any field and on any torsion-free group.
\end{thm}

\begin{thm}
Let $\alpha$ and $\beta$ be non-zero elements of the group algebra of any torsion-free group over an arbitrary field. If $|supp(\alpha)|=3$ and $\alpha \beta=0$ then $|supp(\beta)|\geq 10$. 
\end{thm}

\begin{thm}\label{mainthm}
None of the graphs in Table \ref{tab-forbiddens} can be isomorphic to a subgraph of any zero-divisor graph of length $3$ over $\mathbb{F}_2$ and on any torsion-free group.
\end{thm}

By Theorem \ref{thm-F}, the graphs in Figure \ref{forbiddens-F} are forbidden subgraphs of any zero-divisor graph of length $3$ over $\mathbb{F}$ and on any torsion-free group. Also by Theorem \ref{mainthm}, the graphs in Table \ref{tab-forbiddens} are forbidden subgraphs of any zero-divisor graph of length $3$ over $\mathbb{F}_2$ and on any torsion-free group. In Appendix  of \cite{Ab-Ta}, some details of our computations needed in the proof of Theorem \ref{mainthm} are given for the reader's convenience.

The following result improves  parts (iii) and (iv) of  \cite[Theorem 1.3]{pascal}. 
\begin{thm}
Let $\alpha$ and $\beta$ be non-zero elements of the group algebra of any torsion-free group over the field with two elements. If $|supp(\alpha)|=3$ and $\alpha \beta=0$ then $|supp(\beta)|\geq 20$. 
\end{thm}

The best known result on Conjecture \ref{conj-unit}, which has the purely group-theoretic approach, is concerned with unique product groups \cite{pass1,pass2}. 
%Unique product groups are also two unique product groups \cite{stroj}. 
The latter result covers ordered groups, in particular torsion-free nilpotent groups. Nevertheless, it is still unknown whether or not Conjecture \ref{conj-unit} do hold true for supersolvable torsion-free groups. For any field $\mathbb{F}$ and any torsion-free group $G$, it is known that $\mathbb{F}[G]$ does not contain a unit element whose support is of size at most $2$ (see \cite[Theorem 4.2]{dyk}), but it is not  known a similar result for group algebra elements with the support of size $3$. Dykema et al. \cite{dyk} have shown that there exist no $\gamma , \delta \in \mathbb{F}_2[G]$ such that $\gamma \delta=1$, where $|supp(\gamma)|=3$ and $|supp(\delta)|\leq 11$. Concerning Conjecture \ref{conj-unit},  we prove the following result which improves a part of the result in \cite{dyk} to arbitrary fields. 

\begin{thm}
Let $\gamma$ and $\delta$ be elements of the group algebra of any torsion-free group over an arbitrary field. If $|supp(\gamma)|=3$ and $\gamma \delta =1$ then $|supp(\delta)|\geq 9$. 
\end{thm}

It is known that a group algebra of a field $\mathbb{F}$ over a torsion-free group $G$ contains a zero divisor if and only if it contains a non-zero element whose square is zero (see \cite[Lemma 13.1.2]{pass1}).  Using the latter fact,  it is mentioned in \cite[p. 691]{pascal} that it is sufficient to check  Conjecture \ref{conj-zero} only for the case that $|supp(\alpha)|=|supp(\beta)|$, but in the construction that, given a zero divisor produces an element of square zero, it is not clear how the length changes.  We clarify the latter by the following.
\begin{prop}
Let $G$ be a torsion-free group and $\mathbb{F}$ be a field with the property that there exists a positive integer $k$ such that  $\mathbb{F}[G]$ has no non-zero element $\alpha$ with $|supp(\alpha)|\leq k$ and $\alpha^2=0$. Then there exist no non-zero elements $\alpha_1,\alpha_2 \in \mathbb{F}[G]$ such that $\alpha_1\alpha_2=0$ and $|supp(\alpha_1)||supp(\alpha_2)|\leq k$.
\end{prop}

%-----------------------------------------------------------------------------------------------------------------------------------------
\begin{figure}[ht]
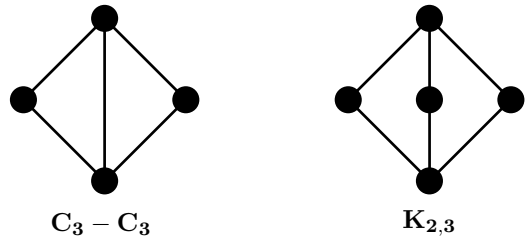

\centering
\psscalebox{0.9 0.9} % Change this value to rescale the drawing.
{

}
\caption{Two forbidden subgraphs of zero-divisor graphs of length $3$ over an arbitrary field and on a torsion-free group}\label{forbiddens-F}
\end{figure}

\begin{center}
\begin{longtable}{|c|c|c|c|}
\caption{{Forbidden subgraphs of zero-divisor graphs of length $3$ over $\mathbb{F}_2$ on any torsion-free group}}\label{tab-forbiddens}\\
\hline %\multicolumn{1}{|c|}{$n$} & \multicolumn{1}{l||}{$R$} & \multicolumn{1}{c|}{$n$} & \multicolumn{1}{l||}{$R$}\\\hline
\endfirsthead
\multicolumn{4}{c}%
{{\tablename\ \thetable{} -- continued from previous page}} \\
%\hline% \multicolumn{1}{|c|}{$n$} & \multicolumn{1}{l||}{$R$} & \multicolumn{1}{c|}{$n$} & \multicolumn{1}{l||}{$R$}\\\hline
\endhead
\hline \multicolumn{4}{|r|}{{Continued on next page}} \\\hline
\endfoot
\hline %\hline
\endlastfoot
\psscalebox{0.62 0.62} % Change this value to rescale the drawing.
{
\begin{pspicture}(0,-2.0985577)(3.594231,2.0985577)
\psdots[linecolor=black, dotsize=0.4](3.3971152,0.30144233)
\psdots[linecolor=black, dotsize=0.4](1.7971153,-1.2985576)
\psdots[linecolor=black, dotsize=0.4](0.19711533,0.30144233)
\psdots[linecolor=black, dotsize=0.4](1.7971153,1.9014423)
\psdots[linecolor=black, dotsize=0.4](1.7971153,0.30144233)
\psline[linecolor=black, linewidth=0.04](1.7971153,1.9014423)(0.19711533,0.30144233)(1.7971153,-1.2985576)(3.3971152,0.30144233)(1.7971153,1.9014423)(1.7971153,0.30144233)(1.7971153,-1.2985576)
\rput[bl](1.2,-2.0985577){$\mathbf{1)\ K_{2,3}}$}
\end{pspicture}
}&\psscalebox{0.62 0.62} % Change this value to rescale the drawing.
{
\begin{pspicture}(0,-2.0985577)(3.594231,2.0985577)
\psdots[linecolor=black, dotsize=0.4](3.3971152,0.30144233)
\psdots[linecolor=black, dotsize=0.4](1.7971153,-1.2985576)
\psdots[linecolor=black, dotsize=0.4](0.19711533,0.30144233)
\psdots[linecolor=black, dotsize=0.4](1.7971153,1.9014423)
\psdots[linecolor=black, dotsize=0.4](1.7971153,0.30144233)
\psdots[linecolor=black, dotsize=0.4](2.5971153,-0.49855766)
\psline[linecolor=black, linewidth=0.04](1.7971153,1.9014423)(0.19711533,0.30144233)(1.7971153,-1.2985576)(2.5971153,-0.49855766)(3.3971152,0.30144233)(1.7971153,1.9014423)(1.7971153,0.30144233)(1.7971153,-1.2985576)
\rput[bl](0.7,-2.0985577){$\mathbf{2) \ C_4--C_5}$}
\end{pspicture}
}&\psscalebox{0.62 0.62} % Change this value to rescale the drawing.
{
\begin{pspicture}(0,-2.0985577)(3.594231,2.0985577)
\psdots[linecolor=black, dotsize=0.4](3.3971152,0.30144233)
\psdots[linecolor=black, dotsize=0.4](1.7971153,-1.2985576)
\psdots[linecolor=black, dotsize=0.4](0.19711533,0.30144233)
\psdots[linecolor=black, dotsize=0.4](1.7971153,1.9014423)
\psdots[linecolor=black, dotsize=0.4](1.7971153,0.30144233)
\psdots[linecolor=black, dotsize=0.4](2.5971153,-0.49855766)
\psline[linecolor=black, linewidth=0.04](1.7971153,1.9014423)(0.19711533,0.30144233)(1.7971153,-1.2985576)(2.5971153,-0.49855766)(3.3971152,0.30144233)(1.7971153,1.9014423)(1.7971153,0.30144233)(1.7971153,-1.2985576)
\psdots[linecolor=black, dotsize=0.4](2.5971153,1.1014423)
\rput[bl](0.7,-2.0985577){$\mathbf{3) \ C_4--C_6}$}
\end{pspicture}
}&\psscalebox{0.62 0.62} % Change this value to rescale the drawing.
{
\begin{pspicture}(0,-1.8985577)(3.594231,1.8985577)
\psdots[linecolor=black, dotsize=0.4](0.9971155,1.7014424)
\psdots[linecolor=black, dotsize=0.4](2.5971155,1.7014424)
\psdots[linecolor=black, dotsize=0.4](1.7971154,1.3014424)
\psdots[linecolor=black, dotsize=0.4](1.7971154,0.5014423)
\psdots[linecolor=black, dotsize=0.4](2.5971155,-0.29855767)
\psdots[linecolor=black, dotsize=0.4](1.7971154,-1.0985577)
\psdots[linecolor=black, dotsize=0.4](0.9971155,-0.29855767)
\psdots[linecolor=black, dotsize=0.4](3.3971155,0.90144235)
\psdots[linecolor=black, dotsize=0.4](0.19711548,0.90144235)
\psline[linecolor=black, linewidth=0.04](0.9971155,1.7014424)(0.19711548,0.90144235)(0.9971155,-0.29855767)(1.7971154,-1.0985577)(2.5971155,-0.29855767)(3.3971155,0.90144235)(2.5971155,1.7014424)(1.7971154,1.3014424)(0.9971155,1.7014424)
\psline[linecolor=black, linewidth=0.04](1.7971154,1.3014424)(1.7971154,0.5014423)(0.9971155,-0.29855767)
\psline[linecolor=black, linewidth=0.04](1.7971154,0.5014423)(2.5971155,-0.29855767)
\rput[bl](0.3,-1.8985577){$\mathbf{4) \ C_4-C_5(-C_5-)}$}
\end{pspicture}
}\\
\end{longtable}
\end{center}

%--------------------------------------------------------------------------------------------------------------------------------------

\section{\bf Zero-divisor graphs and Unit graphs}

\begin{defn}\label{def-Zgraph}
{\rm
For any pair of non-zero elements $(\alpha,\beta)$ of a group algebra over a field $\mathbb{F}$ and a group $G$  such that $\alpha \beta=0$, we assign a graph $Z(\alpha,\beta)$ to $(\alpha,\beta)$ called the zero-divisor graph of $(\alpha,\beta)$ as follows: the vertex set is $supp(\beta)$, the edge set is
\begin{equation*}
\left\{   \{(h,h',g,g') ,(h',h,g',g)\}  \;|\; h,h'\in supp(\alpha), \; g,g'\in supp(\beta), \; g\neq g', \; hg=h'g' \right \},
\end{equation*}
and $\psi_{Z(\alpha,\beta)}: \mathcal{E}_{Z(\alpha,\beta)}\rightarrow \mathcal{V}_{Z(\alpha,\beta)}^{(2)}$ is defined by
$$\psi_{Z(\alpha,\beta)}(\{(h,h',g,g') ,(h',h,g',g)\})= \{g,g'\},$$
for all  $\{(h,h',g,g') ,(h',h,g',g)\}\in \mathcal{E}_{Z(\alpha,\beta)}$.

We call $Z(\alpha, \beta)$ a zero-divisor graph of length $|supp(\alpha)|$ over the field $\mathbb{F}$ and on the group $G$.
}
\end{defn}
  
\begin{rem}\label{r-Zgraph}
{\rm 
\begin{enumerate}
\item We note that the graph $Z(\alpha,\beta)$ is an undirected graph with no loops but it may happen that a zero-divisor graph has multi-edge; that is,  in general a zero-divisor graph may not be a simple graph.
\item   Being non-zero of $\beta$ ensures that the vertex set is non-empty
 and the zero-divisor condition $\alpha \beta=0$ and the conditions $\alpha\not=0$ and $\beta\not=0$ imply that the edge set of the graph is non-empty. 
 \item One may use instead $Z(\alpha,\beta)$ the notations such as $Z(\alpha, \beta, |supp(\alpha)|,|supp(\beta)|)$, 
 $Z_{\mathbb{F}}(\alpha,\beta)$ or $Z_{\mathbb{F}[G]}(\alpha,\beta)$, $\dots$ to indicate the support sizes of $\alpha$ and $\beta$, the underlying field or the  underlying group algebra, $\dots$.
% \item One may consider the directed zero-divisor graph $\overset{\rightarrow}{Z}(\alpha,\beta)$ by defining the %edge set as follows: 
 %\begin{equation*}
%\{   (h',h,g',g)  \;|\; h,h'\in supp(\alpha), \; g,g'\in supp(\beta), \; g\neq g', \; hg=h'g'  \}
%\end{equation*}
\item From the notation $Z(\alpha,\beta)$ we understand that $\alpha \beta=0$ for two non-zero elements of a group algebra and so $x\alpha \beta y=0$ for all group elements $x$ and $y$. So we may consider the zero-divisor graph 
 $Z(x\alpha,\beta y)$ which is isomorphic to $Z(\alpha,\beta)$ (see below, Lemma \ref{iso}).
\item In this paper we will only study the zero-divisor graphs of pairs whose first component has support size $3$ and mostly the field is $\mathbb{F}_2$ and the underlying group $G$ is always torsion-free.
\item If $|supp(\alpha)|=3$ and the underlying group $G$ is torsion-free, the zero-divisor graph is  simple so that between two distinct vertices there is at most one edge (see below Proposition \ref{simple}). 
\item If we choose $\beta$ of minimum support size with respect to the property $\alpha \beta=0$, the graph $Z(\alpha,\beta)$ will be connected (see below Lemma \ref{Z-connect}).
\end{enumerate}
} 
\end{rem}

\begin{exm}
\begin{enumerate}
\item
 Let $G_1=\langle x\rangle$ be the cyclic group of order $7$ and let $\alpha=1+x^2+x^3+x^4$ and $\beta=1+x+x^5$ in $\mathbb{F}_2[G_1]$, then $\alpha \beta=0$. Figures \ref{f-zgraph-1} and \ref{f-zgraph-2} show the graphs $Z(\alpha,\beta)$ and $Z(\beta,\alpha)$, respectively.
 
\begin{figure}[ht]
\psscalebox{0.9 0.9} % Change this value to rescale the drawing.
{
\begin{tikzpicture}
[every node/.style={inner sep=0pt}]
\node (1) [circle, minimum size=12.5pt, fill=black, line width=0.625pt, draw=black] at (125.0pt, -37.5pt) {\textcolor{black}{1}};
\node (2) [circle, minimum size=12.5pt, fill=black, line width=0.625pt, draw=black] at (87.5pt, -100.0pt) {\textcolor{black}{2}};
\node (3) [circle, minimum size=12.5pt, fill=black, line width=0.625pt, draw=black] at (162.5pt, -100.0pt) {\textcolor{black}{3}};
\draw [line width=1.25, color=black] (2) to  (1);
\draw [line width=1.25, color=black] (1) to  (3);
\draw [line width=1.25, color=black] (3) to  (2);
\draw [line width=1.25, color=black] (1) to  [in=90, out=209] (2);
\draw [line width=1.25, color=black] (2) to  [in=209, out=331] (3);
\draw [line width=1.25, color=black] (1) to  [in=90, out=331] (3);
\node at (125.0pt, -23.125pt) {\textcolor{black}{$1$}};
\node at ((72.75pt, -100.0pt) {\textcolor{black}{$x^5$}};
\node at (175.625pt, -100.0pt) {\textcolor{black}{$x$}};
\end{tikzpicture}
}
\caption{$Z(1+x^2+x^3+x^4,1+x+x^5)$ over  $\mathbb{F}_2$ and on $G_1$}\label{f-zgraph-1}
\end{figure}
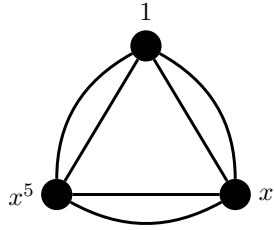

\begin{figure}[ht]
\psscalebox{0.9 0.9} % Change this value to rescale the drawing.
{
\begin{tikzpicture}
[every node/.style={inner sep=0pt}]
\node (1) [circle, minimum size=12.5pt, fill=black, line width=0.625pt, draw=black] at (100.0pt, -37.5pt) {\textcolor{black}{1}};
\node (2) [circle, minimum size=12.5pt, fill=black, line width=0.625pt, draw=black] at (162.5pt, -37.5pt) {\textcolor{black}{2}};
\node (3) [circle, minimum size=12.5pt, fill=black, line width=0.625pt, draw=black] at (162.5pt, -100.0pt) {\textcolor{black}{3}};
\node (4) [circle, minimum size=12.5pt, fill=black, line width=0.625pt, draw=black] at (100.0pt, -100.0pt) {\textcolor{black}{4}};
\draw [line width=1.25, color=black] (1) to  (4);
\draw [line width=1.25, color=black] (4) to  (3);
\draw [line width=1.25, color=black] (3) to  (2);
\draw [line width=1.25, color=black] (1) to  (2);
\draw [line width=1.25, color=black] (1) to  (3);
\draw [line width=1.25, color=black] (2) to  (4);
\node at (100.0pt, -23.125pt) {\textcolor{black}{$1$}};
\node at (162.5pt, -23.125pt) {\textcolor{black}{$x^2$}};
\node at (162.5pt, -114.375pt) {\textcolor{black}{$x^3$}};
\node at (100.0pt, -114.375pt) {\textcolor{black}{$x^4$}};
\end{tikzpicture}
}
\caption{$Z(1+x+x^5,1+x^2+x^3+x^4)$ over  $\mathbb{F}_2$ and on  $G_1$}\label{f-zgraph-2}
\end{figure}
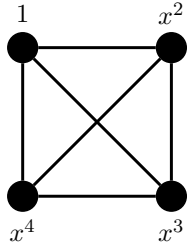
\end{enumerate}
\end{exm}

\begin{lem}\label{iso}
$Z(\alpha,\beta)\cong Z(x \alpha, \beta y)$ for all group elements $x$ and $y$. 
\end{lem}  
\begin{proof}
Consider the maps $\phi_V: supp(\beta) \rightarrow supp(\beta y)$ and $\phi_E: \mathcal{E}_{Z(\alpha,\beta)}\rightarrow \mathcal{E}_{Z(x \alpha, \beta y)}$ defined by $\phi_V(g)= gy$ for all $g\in supp(\beta)$ and 
\begin{equation*}
\phi_E( \{(h,h',g,g') ,(h',h,g',g)\} )=  \{(xh,xh',gy,g'y) ,(xh',xh,g'y,gy)\}
\end{equation*}
for all  $\{(h,h',g,g') ,(h',h,g',g)\}\in \mathcal{E}_{Z(\alpha,\beta)}$. Then it is easy to see that $(\phi_V,\phi_E)$ is an isomorphism from $Z(\alpha,\beta)$ to $Z(x \alpha, \beta y)$.
\end{proof}

\begin{lem}\label{supp}
Suppose that $\alpha$ and $\beta$ are non-zero elements of a group algebra such that $\alpha \beta=0$.
\begin{enumerate}
\item Suppose that if $\alpha \beta'=0$ for some non-zero element $\beta'$ of the group algebra, then $|supp(\beta')|\geq |supp(\beta)|$. Then  $supp(\beta) g^{-1}\subseteq \langle x supp(\alpha)  \rangle$  for all $g\in supp(\beta)$ and all group elements  $x$. 

\item Suppose that if $\alpha' \beta=0$ for some non-zero element  $\alpha'$ of the group algebra, then $|supp(\alpha')|\geq |supp(\alpha)|$. Then 
$h^{-1} supp(\alpha)\subseteq \langle  supp(\beta) y \rangle$ 
 for all $h\in supp(\alpha)$ and all group elements $y$.
\end{enumerate}
 \end{lem}
\begin{proof}
(1) \; Let  $H=\langle x supp(\alpha) \rangle$ and suppose that $\{t_1,t_2,\ldots,t_k\}$ be a subset of right coset representatives of $H$ in the group such that $supp(\beta) g^{-1} \cap Ht_i\neq \varnothing$ for all $i\in\{1,\dots,k\}$ and $supp(\beta)g^{-1} \subseteq \cup_{i=1}^k Ht_i$.   Since $x \alpha \beta g^{-1}=0$ and $Ht_{i}\cap Ht_{j}=\varnothing$ for all distinct $i,j\in \{1,2,\ldots,k\}$,  $(x \alpha )(\sum_{g'\in Ht_i} \beta_{g'} g')=0$ for all $i\in \{1,2,\ldots,k\}$, where $\beta=\sum_{g'\in G} \beta_{g'} g'$. Now it follows from the hypothesis that $k=1$ so that $supp(\beta)g^{-1} \subseteq Ht_1$. Hence $supp(\beta)g^{-1} \subseteq H$ for all $g\in supp(\beta)$.  

(2) \;  Consider the subgroup $K=\langle  supp(\beta) y \rangle$  and take  a finite subset $S$ of left coset representatives of $K$ in the group such that $h^{-1} supp(\alpha) \cap sK\neq \varnothing$ for all $s \in S$ and $h^{-1} supp(\alpha) \subseteq \cup_{s\in S} sK$. By a similar argument as in part (1), $|S|=1$ and 
  the inclusion is proved.
\end{proof}

\begin{lem}\label{Z-connect}
Let $Z(\alpha,\beta)$ be a zero-divisor graph such that $\beta$ has minimum possible support size among all non-zero elements $\gamma$ with $\alpha \gamma=0$. Then $Z(\alpha,\beta)$ is connected. 
\end{lem}
\begin{proof}
Suppose, for a contradiction, that $\mathcal{V}_{Z(\alpha,\beta)}=supp(\beta)$ is partitioned into two non-empty subsets $V_1$ and $V_2$ such that there is no edge between $V_1$ and $V_2$. It follows that $supp(\alpha\beta_1) \cap supp(\alpha\beta_2)=\varnothing$, where $\beta=\sum_{g\in G} \beta_g g$ and $\beta_i=\sum_{g\in V_i} \beta_g g$ ($i=1,2$). Since $\alpha \beta_1 + \alpha \beta_2=0$, 
\begin{equation} \label{eq1}
\alpha\beta_1=\alpha \beta_2=0,
\end{equation}  Since $|supp(\beta_i)|<|supp(\beta)|$, the equalities \ref{eq1} give contradiction. This completes the proof.
\end{proof}

\begin{defn}\label{def-Ugraph}
{\rm
For any pair of  elements $(\alpha,\beta)$ of a group algebra over a field $\mathbb{F}$ and on a group $G$  such that $\alpha \beta=1$, we assign to $(\alpha,\beta)$ a graph $U(\alpha,\beta)$ called the unit graph of $(\alpha,\beta)$ as follows: the vertex set is $supp(\beta)$, the edge set is
\begin{equation*}
\left \{   \{(h,h',g,g') ,(h',h,g',g)\}  \;|\; h,h'\in supp(\alpha), \; g,g'\in supp(\beta), \; g\neq g', \; hg=h'g' \right \},
\end{equation*}
and if $\mathcal{E}_{U(\alpha,\beta)}\neq \varnothing$, the function $\psi_{U(\alpha,\beta)}: \mathcal{E}_{U(\alpha,\beta)}\rightarrow \mathcal{V}_{U(\alpha,\beta)}^{(2)}$ is defined by
$$\psi_{U(\alpha,\beta)}(\{(h,h',g,g') ,(h',h,g',g)\})= \{g,g'\},$$
for all  $\{(h,h',g,g') ,(h',h,g',g)\}\in \mathcal{E}_{U(\alpha,\beta)}$.

We call $U(\alpha, \beta)$ a unit graph of length $|supp(\alpha)|$ over the field $\mathbb{F}$ and on the group $G$.
}
\end{defn}

\begin{rem}\label{r-Ugraph}
{\rm 
\begin{enumerate}
\item Note that the main difference between the definition of the zero-divisor graph $Z(\alpha,\beta)$ and the unit graph $U(\alpha,\beta)$ is in the conditions on $\alpha$ and $\beta$ and there is no difference in defining the vertex and edge sets.  
\item Any unit graph is an undirected graph with no loops but it may happen that a unit graph has multi-edge; that is,  in general a unit graph may not be a simple graph.
\item The vertex set of $U(\alpha,\beta)$ is non-empty, since $\beta$ is non-zero. 
\item If $\alpha \beta=1$, then there exists $h\in supp(\alpha)$ such that $h^{-1} \in supp(\beta)$. The latter allows  us to assume $1\in supp(\alpha) \cap supp(\beta)$ whenever we are studying the unit graph $U(\alpha,\beta)$ in view of Lemma \ref{U-iso}, below.
 \item One may use instead $U(\alpha,\beta)$ the notations such as $U(\alpha, \beta, |supp(\alpha)|,|supp(\beta)|)$, 
 $U_{\mathbb{F}}(\alpha,\beta)$ or $U_{\mathbb{F}[G]}(\alpha,\beta)$, $\dots$ to indicate the support sizes of $\alpha$ and $\beta$, the underlying field or the  underlying group algebra, $\dots$.
\item In this paper we will only study  unit graphs of pairs whose first component has support size $3$ and the underlying group $G$ is always torsion-free.
\item If $|supp(\alpha)|=3$ and the underlying group $G$ is torsion-free, the unit graph is  simple so that between two distinct vertices there is at most one edge (see below Proposition \ref{simple-unit}). 
\item If we choose $\beta$ of minimum support size with respect to the property $\alpha \beta=1$, the graph $U(\alpha,\beta)$ will be connected (see below  Lemma \ref{U-connect}).
\end{enumerate}
}
\end{rem}
\begin{exm}
\begin{enumerate}
\item
 Let $G_2=\langle x\rangle$ be the cyclic group of order $3$ and $\mathbb{F}_3$ be the  field with $3$ elements. Then $\alpha=-1+x-x^2$ is a non-trivial unit with inverse $\beta=1+x$ in $\mathbb{F}_3[G_2]$. Figures \ref{f-ugraph-1} and \ref{f-ugraph-2} show the graphs $U(\alpha,\beta)$ and $U(\beta,\alpha)$, respectively.
\begin{figure}[ht]
\psscalebox{0.9 0.9} % Change this value to rescale the drawing.
{
\begin{tikzpicture}
[every node/.style={inner sep=0pt}]
\node (1) [circle, minimum size=12.5pt, fill=black, line width=0.625pt, draw=black] at (100.0pt, -87.5pt) {\textcolor{black}{1}};
\node (2) [circle, minimum size=12.5pt, fill=black, line width=0.625pt, draw=black] at (187.5pt, -87.5pt) {\textcolor{black}{2}};
\draw [line width=1.25, color=black] (1) to  (2);
\draw [line width=1.25, color=black] (1) to  [in=151, out=29] (2);
\draw [line width=1.25, color=black] (1) to  [in=209, out=331] (2);
\node at (87.25pt, -87.5pt) {\textcolor{black}{$1$}};
\node at (200.625pt, -87.5pt) {\textcolor{black}{$x$}};
\end{tikzpicture}
}
\caption{$U(-1+x-x^2,1+x)$ over  $\mathbb{F}_3$ and on $G_2$}\label{f-ugraph-1}
\end{figure}
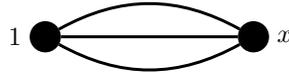

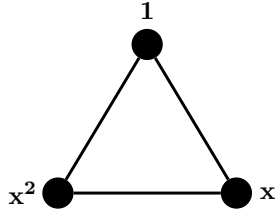
\begin{figure}[ht]
\psscalebox{0.9 0.9} % Change this value to rescale the drawing.
{
\begin{tikzpicture}
[every node/.style={inner sep=0pt}]
\node (1) [circle, minimum size=12.5pt, fill=black, line width=0.625pt, draw=black] at (125.0pt, -37.5pt) {\textcolor{black}{1}};
\node (2) [circle, minimum size=12.5pt, fill=black, line width=0.625pt, draw=black] at (87.5pt, -100.0pt) {\textcolor{black}{2}};
\node (3) [circle, minimum size=12.5pt, fill=black, line width=0.625pt, draw=black] at (162.5pt, -100.0pt) {\textcolor{black}{3}};
\draw [line width=1.25, color=black] (2) to  (1);
\draw [line width=1.25, color=black] (1) to  (3);
\draw [line width=1.25, color=black] (3) to  (2);
\node at (125.0pt, -23.125pt) {\textcolor{black}{$\mathbf{1}$}};
\node at (72.75pt, -100.0pt) {\textcolor{black}{$\mathbf{x^2}$}};
\node at (175.625pt, -100.0pt) {\textcolor{black}{$\mathbf{x}$}};
\end{tikzpicture}
}
\caption{$U(1+x,-1+x-x^2)$ over  $\mathbb{F}_3$ and on $G_2$}\label{f-ugraph-2}
\end{figure}

\item
Let $G_3=\langle x \rangle$ be the cyclic group of order $8$ and $\mathbb{F}$ be an arbitrary field whose characteristic is not $2$. Then $\alpha=-1-x+x^3+2\cdot x^4+x^5-x^7$ is a non-trivial unit with inverse $\beta=-1+x-x^3+2 \cdot x^4-x^5+x^7$ in $\mathbb{F}[G_3]$. Figure \ref{f-ugraph-3} shows the graphs $U(\alpha,\beta)$ and $U(\beta,\alpha)$ which are isomorphic.
\begin{figure}[h]
\psscalebox{0.9 0.9} % Change this value to rescale the drawing.
{
\begin{tikzpicture}
[every node/.style={inner sep=0pt}]
\node (1) [circle, minimum size=12.5pt, fill=black, line width=0.625pt, draw=black] at (187.5pt, -37.5pt) {\textcolor{black}{1}};
\node (4) [circle, minimum size=12.5pt, fill=black, line width=0.625pt, draw=black] at (187.5pt, -262.5pt) {\textcolor{black}{4}};
\node (2) [circle, minimum size=12.5pt, fill=black, line width=0.625pt, draw=black] at (87.5pt, -100.0pt) {\textcolor{black}{2}};
\node (3) [circle, minimum size=12.5pt, fill=black, line width=0.625pt, draw=black] at (87.5pt, -200.0pt) {\textcolor{black}{3}};
\node (5) [circle, minimum size=12.5pt, fill=black, line width=0.625pt, draw=black] at (287.5pt, -200.0pt) {\textcolor{black}{5}};
\node (6) [circle, minimum size=12.5pt, fill=black, line width=0.625pt, draw=black] at (287.5pt, -100.0pt) {\textcolor{black}{6}};
\draw [line width=1.25, color=black] (2) to  [in=217, out=29] (1);
\draw [line width=1.25, color=black] (2) to  (6);
\draw [line width=1.25, color=black] (5) to  (2);
\draw [line width=1.25, color=black] (3) to  [in=276, out=84] (2);
\draw [line width=1.25, color=black] (2) to  (4);
\draw [line width=1.25, color=black] (1) to  (3);
\draw [line width=1.25, color=black] (4) to  (1);
\draw [line width=1.25, color=black] (5) to  (1);
\draw [line width=1.25, color=black] (6) to  [in=323, out=151] (1);
\draw [line width=1.25, color=black] (5) to  [in=264, out=96] (6);
\draw [line width=1.25, color=black] (6) to  (4);
\draw [line width=1.25, color=black] (6) to  (3);
\draw [line width=1.25, color=black] (5) to  [in=29, out=217] (4);
\draw [line width=1.25, color=black] (5) to  (3);
\draw [line width=1.25, color=black] (4) to  [in=331, out=143] (3);
\draw [line width=1.25, color=black] (2) to  [in=114, out=307] (4);
\draw [line width=1.25, color=black] (2) to  [in=209, out=37] (1);
\draw [line width=1.25, color=black] (2) to  [in=204, out=45] (1);
\draw [line width=1.25, color=black] (2) to  [in=197, out=45] (1);
\draw [line width=1.25, color=black] (2) to  [in=127, out=294] (4);
\draw [line width=1.25, color=black] (2) to  [in=114, out=315] (4);
\draw [line width=1.25, color=black] (2) to  [in=151, out=336] (5);
\draw [line width=1.25, color=black] (2) to  [in=156, out=331] (5);
\draw [line width=1.25, color=black] (2) to  [in=143, out=343] (5);
\draw [line width=1.25, color=black] (2) to  [in=163, out=323] (5);
\draw [line width=1.25, color=black] (2) to  [in=135, out=349] (5);
\draw [line width=1.25, color=black] (1) to  [in=84, out=276] (4);
\draw [line width=1.25, color=black] (1) to  [in=96, out=264] (4);
\draw [line width=1.25, color=black] (1) to  [in=79, out=281] (4);
\draw [line width=1.25, color=black] (1) to  [in=101, out=259] (4);
\draw [line width=1.25, color=black] (1) to  [in=73, out=287] (4);
\draw [line width=1.25, color=black] (1) to  [in=114, out=307] (5);
\draw [line width=1.25, color=black] (1) to  [in=127, out=294] (5);
\draw [line width=1.25, color=black] (1) to  [in=114, out=315] (5);
\draw [line width=1.25, color=black] (5) to  [in=24, out=225] (4);
\draw [line width=1.25, color=black] (5) to  [in=37, out=209] (4);
\draw [line width=1.25, color=black] (5) to  [in=17, out=225] (4);
\draw [line width=1.25, color=black] (2) to  [in=96, out=264] (3);
\draw [line width=1.25, color=black] (2) to  [in=101, out=259] (3);
\draw [line width=1.25, color=black] (2) to  [in=107, out=253] (3);
\draw [line width=1.25, color=black] (3) to  [in=204, out=29] (6);
\draw [line width=1.25, color=black] (6) to  [in=24, out=209] (3);
\draw [line width=1.25, color=black] (6) to  [in=17, out=217] (3);
\draw [line width=1.25, color=black] (3) to  [in=197, out=37] (6);
\draw [line width=1.25, color=black] (6) to  [in=11, out=225] (3);
\draw [line width=1.25, color=black] (2) to  [in=174, out=6] (6);
\draw [line width=1.25, color=black] (2) to  [in=186, out=354] (6);
\draw [line width=1.25, color=black] (2) to  [in=169, out=11] (6);
\draw [line width=1.25, color=black] (6) to  [in=84, out=276] (5);
\draw [line width=1.25, color=black] (6) to  [in=79, out=281] (5);
\draw [line width=1.25, color=black] (6) to  [in=73, out=287] (5);
\draw [line width=1.25, color=black] (1) to  [in=143, out=331] (6);
\draw [line width=1.25, color=black] (1) to  [in=135, out=336] (6);
\draw [line width=1.25, color=black] (1) to  (6);
\draw [line width=1.25, color=black] (5) to  [in=6, out=174] (3);
\draw [line width=1.25, color=black] (5) to  [in=354, out=186] (3);
\draw [line width=1.25, color=black] (5) to  [in=349, out=191] (3);
\draw [line width=1.25, color=black] (3) to  [in=233, out=66] (1);
\draw [line width=1.25, color=black] (3) to  [in=246, out=53] (1);
\draw [line width=1.25, color=black] (3) to  [in=225, out=66] (1);
\draw [line width=1.25, color=black] (3) to  [in=151, out=323] (4);
\draw [line width=1.25, color=black] (3) to  [in=156, out=315] (4);
\draw [line width=1.25, color=black] (3) to  [in=163, out=315] (4);
\draw [line width=1.25, color=black] (6) to  [in=53, out=246] (4);
\draw [line width=1.25, color=black] (6) to  [in=66, out=233] (4);
\draw [line width=1.25, color=black] (6) to  [in=45, out=246] (4);
\node at (187.5pt, -23.125pt) {\textcolor{black}{$x$}};
\node at (187.5pt, -276.875pt) {\textcolor{black}{$x^5$}};
\node at (70.75pt, -100.0pt) {\textcolor{black}{$1$}};
\node at (70.75pt, -200.0pt) {\textcolor{black}{$x^7$}};
\node at (303.25pt, -200.0pt) {\textcolor{black}{$x^4$}};
\node at (303.25pt, -100.0pt) {\textcolor{black}{$x^3$}};
\end{tikzpicture}
}
\caption{$U(\alpha,\beta)\cong U(\beta,\alpha)$ over $\mathbb{F}$ with $char \mathbb{F}\neq 2$ and on $G_3$ where $\alpha=-1-x+x^3+2\cdot x^4+x^5-x^7$ and $\beta=-1+x-x^3+2\cdot x^4-x^5+x^7$}\label{f-ugraph-3}
\end{figure}
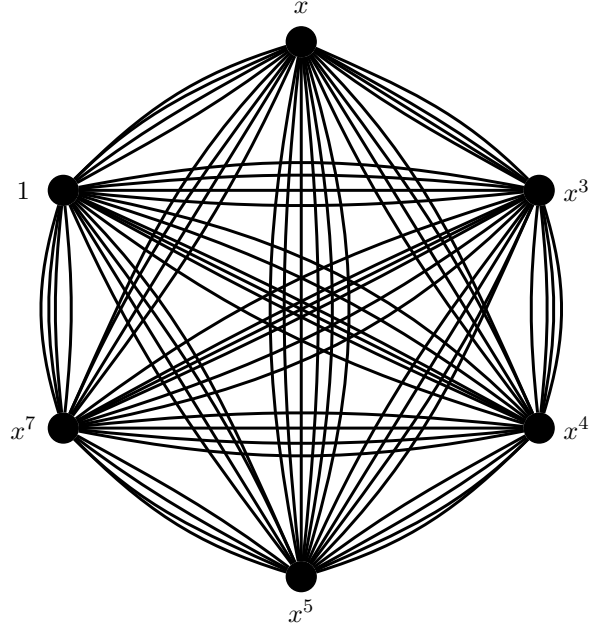
\end{enumerate}
\end{exm}

\begin{lem}\label{U-iso}
$U(\alpha,\beta)\cong U(x^{-1} \alpha, \beta x)$ for all group elements $x$. 
\end{lem}  
\begin{proof}
We note that $(x^{-1} \alpha) (\beta x)=1$ and so one may speak of $U(x^{-1} \alpha, \beta x)$. The proof is similar to that of Lemma \ref{iso}.
\end{proof}

\begin{lem}\label{U-connect}
Let $U(\alpha,\beta)$ be a unit graph such that $\beta$ has minimum possible support size among all  elements $\gamma$ with $\alpha \gamma=1$. Then $U(\alpha,\beta)$ is connected. 
\end{lem}
\begin{proof}
Suppose, for a contradiction, that $\mathcal{V}_{U(\alpha,\beta)}=supp(\beta)$ is partitioned into two non-empty subsets $V_1$ and $V_2$ such that there is no edge between $V_1$ and $V_2$. It follows that $supp(\alpha\beta_1) \cap supp(\alpha\beta_2)=\varnothing$, where $\beta=\sum_{g\in G} \beta_g g$ and $\beta_i=\sum_{g\in V_i} \beta_g g$ ($i=1,2$). Since $\alpha\beta_1 + \alpha\beta_2=1$,
\begin{equation} \label{eq2}
\alpha\beta_1=1 \; {\rm or} \; \alpha \beta_2=1.
\end{equation}  Since $|supp(\beta_i)|<|supp(\beta)|$, each of the equalities \ref{eq2} (if happens) is a contradiction. This completes the proof.
\end{proof}

\begin{lem}\label{supp-unit}
Suppose that $\alpha$ and $\beta$ are elements of a group algebra such that $\alpha \beta=1$.
\begin{enumerate}
\item Suppose that if $\alpha \beta'=1$ for some  element $\beta'$ of the group algebra, then $|supp(\beta')|\geq |supp(\beta)|$. Then $\langle  supp(\beta)\rangle \subseteq \langle  supp(\alpha)  \rangle$. 
 \item Suppose that if $\alpha' \beta=1$ for some  element $\alpha'$ of the group algebra, then $|supp(\alpha')|\geq |supp(\alpha)|$. Then $\langle  supp(\alpha)\rangle \subseteq \langle  supp(\beta)  \rangle$. 
\end{enumerate}
\end{lem}
\begin{proof}
(1) \; Let  $H=\langle supp(\alpha) \rangle$ and suppose that $\{t_1,t_2,\ldots,t_k\}$ be a subset of right coset representatives of $H$ in the group such that $supp(\beta)  \cap Ht_i\neq \varnothing$ for all $i\in\{1,\dots,k\}$ and $supp(\beta) \subseteq \cup_{i=1}^k Ht_i$.  Since 
\begin{equation*}
\alpha \beta =\alpha \left(\sum_{i=1}^k \sum_{g\in Ht_i} \beta_{g} g\right)=1 \; {\rm and} \; Ht_{i}\cap Ht_{j}=\varnothing
\end{equation*}
  for all distinct $i,j\in \{1,2,\ldots,k\}$,  it follows that there exists $i\in\{1,\dots,k\}$ such that $Ht_i=H$ and 
\begin{equation*}\alpha \left(\sum_{g\in Ht_\ell} \beta_{g} g\right)=\left\{ \begin{array}{ll} 1 & {\rm if} \; \ell=i \\ 0 & {\rm if} \; \ell \not=i \end{array} \right. 
\end{equation*} 
for all $\ell\in \{1,2,\ldots,k\}$, where $\beta=\sum_{g\in G} \beta_{g} g$. Now it follows from the hypothesis that $k=i=1$ so that $\langle supp(\beta) \rangle \subseteq H$. 

(2) \; Consider the subgroup $K=\langle  supp(\beta)  \rangle$  and take  a finite subset $S$ of left coset representatives of $K$ in the group such that $ supp(\alpha) \cap sK\neq \varnothing$ for all $s \in S$ and $ supp(\alpha) \subseteq \cup_{s\in S} sK$. 
The rest of the proof is similar to  the part (1).
\end{proof}

We finish this section with following question:

\begin{qu}
Which graphs can be isomorphic to a zero-divisor or unit graph?
\end{qu}

\section{\bf Zero-divisor graphs and Unit graphs for elements whose supports are of size  $3$}\label{S1-2}

\begin{cor}\label{cor}
Suppose that $\alpha$ and $\beta$ are non-zero elements of a group algebra of a torsion-free group $G$ such that $|supp(\alpha)|=3$, $\alpha\beta=0$ and if  $\alpha \beta'=0$ for some non-zero element $\beta'$ of the group algebra, then  $|supp(\beta')|\geq |supp(\beta)|$. Then $\langle h^{-1} supp(\alpha)\rangle = \langle  supp(\beta) g^{-1} \rangle$ for all $h\in supp(\alpha)$ and $g\in supp(\beta)$.
\end{cor}
\begin{proof}
It follows from \cite[Theorem 2.1]{pascal} that 
if $\alpha' \beta=0$ for some non-zero element $\alpha'$  of the group algebra, then $|supp(\alpha')|\geq |supp(\alpha)|=3$. Now  Lemma \ref{supp} completes the proof.
\end{proof}

\begin{cor}\label{Unit-cor}
Suppose that $\alpha$ and $\beta$ are elements of a group algebra of a torsion-free group $G$ such that $|supp(\alpha)|=3$, $\alpha\beta=1$ and if  $\alpha \beta'=1$ for some  element $\beta'$ of the group algebra, then  $|supp(\beta')|\geq |supp(\beta)|$. Then $\langle  supp(\alpha)\rangle = \langle  supp(\beta)  \rangle$.
\end{cor}
\begin{proof}
It follows from \cite[Theorem 4.2]{dyk} that 
if $\alpha' \beta=1$ for some non-zero element $\alpha'$  of the group algebra, then $|supp(\alpha')|\geq |supp(\alpha)|=3$. Now  Lemma \ref{supp-unit} completes the proof.
\end{proof}

\begin{rem}\label{r-G}
{\rm
Suppose that $\alpha$ and $\beta$ are non-zero elements of a group algebra of a torsion-free group $G$ such that $|supp(\alpha)|=3$, $\alpha\beta=0$ and if  $\alpha \beta'=0$ for some non-zero element $\beta'$ of the group algebra, then  $|supp(\beta')|\geq |supp(\beta)|$. In studying the zero-divisor graph $Z(\alpha,\beta)$, in view of Lemma \ref{iso} and Corollary \ref{cor} one may assume that $1 \in supp(\alpha) \cap supp(\beta)$ and $G=\langle supp(\alpha) \rangle= \langle supp(\beta) \rangle$. 
}
\end{rem}

\begin{rem}\label{Unit-r-G}
{\rm
Suppose that $\alpha$ and $\beta$ are elements of a group algebra of a torsion-free group $G$ such that $|supp(\alpha)|=3$, $\alpha\beta=1$ and if  $\alpha \beta'=1$ for some element $\beta'$ of the group algebra, then  $|supp(\beta')|\geq |supp(\beta)|$. In studying the unit graph $U(\alpha,\beta)$, by part (4) of Remark 
\ref{r-Ugraph} and  Corollary \ref{Unit-cor} one may assume that $1 \in supp(\alpha) \cap supp(\beta)$ and $G=\langle supp(\alpha) \rangle= \langle supp(\beta) \rangle$. 
}
\end{rem}

\begin{lem}\label{SS}
Let $A$ be a subset of size $3$ of a torsion-free group. If $S:=\{h_1^{-1} h_2 \;|\; h_1,h_2\in A, \; h_1\neq h_2\}$ is of size at most $5$ then there exists $h\in A$ such that $\langle h^{-1} A\rangle$ is an infinite cyclic group.
\end{lem}
\begin{proof}
Let $A=\{h_1,h_2,h_3\}$. Since $|S|\leq 5$, it follows that $h_{i}^{-1}h_j=h_{i'}^{-1}h_{j'}$ for some $(i,j)\not= (i',j')$ and $(i,i')\not=(j,j')$.
 It follows that $(i',j')$ is equal to $(j,i)$, $(j,k)$ or $(k,i)$, where $k \in\{1,2,3\}\setminus\{i,j\}$. If $(i',j')=(j,i)$, then $(h_i^{-1}h_j)^2=1$ and since the group is torsion-free, $h_i=h_j$, a contradiction. If $(i',j')=(j,k)$, then $h_i^{-1}h_k=(h_i^{-1}h_j)^2$ and if $(i',j')=(k,i)$ then 
 $h_k^{-1}h_j=(h_i^{-1} h_j)^2$. Thus $S=\{h_i^{-1} h_j,(h_i^{-1} h_j)^{-1}, (h_i^{-1} h_j)^2, (h_i^{-1} h_j)^{-2}\}$. It follows that $\langle h_i^{-1} A \rangle=\langle h_i^{-1} h_j \rangle$ is an infinite cyclic group. 
\end{proof}

\begin{lem}\label{size of S}
Let $\alpha$  be a non-zero element in the group algebra of a torsion-free group  such that $\alpha \beta=0$ for some non-zero element $\beta$ of the group algebra. If $|supp(\alpha)|=3$,  then $S=\{h^{-1} h' \;|\; h,h'\in supp(\alpha), h \neq h' \}$ has size $6$.
\end{lem}
\begin{proof}
Let $A:=supp(\alpha)$ and suppose, for a contradiction, that $|S|\leq 5$. It follows from 
Lemma \ref{SS} that $H:=\langle h^{-1} A\rangle$ is an infinite cyclic group for some $h\in A$. Now assume that  $|supp(\beta)|$ is minimum with respect to the property $\alpha \beta=0$.   It follows from  Lemma \ref{supp} that  $h^{-1} \alpha, \beta g^{-1}$ belong to the group algebra on $H$ for any arbitrary element $g\in supp(\beta)$. 
 Now $(h^{-1} \alpha) (\beta g^{-1})=0$ contradicts the fact that the group algebra of the infinite cyclic group has no zero-divisors (see \cite[Theorem 26.2]{pass2}). This completes  the proof.   
\end{proof}

\begin{prop}\label{simple}
Let $\alpha$  be a non-zero element in the group algebra of a torsion-free group such that $\alpha \beta=0$ for some non-zero element $\beta$ of the group algebra. If  $|supp(\alpha)|=3$ and $S:=\{h^{-1} h' \;|\; h,h'\in supp(\alpha), h \neq h' \}$, then
$Z(\alpha,\beta)$ is isomorphic to the induced subgraph of the Cayley graph $Cay(G,S)$ on $supp(\beta)$.
\end{prop}
\begin{proof}
Let $\Gamma$ be the  induced subgraph of the Cayley graph $Cay(G,S)$ on $supp(\beta)$.
 The map $\phi_E: \mathcal{E}_{Z(\alpha,\beta)} \rightarrow \mathcal{E}_{\Gamma}$ defined by $\{(h,h',g,g'), (h',h,g',g)\}\mapsto \{g,g'\}$ is a bijective map; for by Lemma \ref{size of S}, $|S|=6$ which implies that $h^{-1} h'= h_1^{-1} h_1'$ if and only if $(h,h')=(h_1,h_1')$, whenever the entries of the latter pairs are distinct and belong to $supp(\alpha)$. Now take $\phi_V$ be the identity map on $supp(\beta)$, then $(\phi_V,\phi_E)$ is an isomorphism from $Z(\alpha,\beta)$ to $\Gamma$. 
\end{proof}

\begin{lem}\label{size of S for Unit}
Let $\alpha$  be an element in the group algebra of a torsion-free group $G$ such that $\alpha \beta=1$ for some element $\beta$ of the group algebra. If $|supp(\alpha)|=3$,  then $S=\{h^{-1} h' \;|\; h,h'\in supp(\alpha), h \neq h' \}$ has size $6$.
\end{lem}
\begin{proof}
Let $A:=supp(\alpha)$ and suppose, for a contradiction, that $|S|\leq 5$. It follows from 
Lemma \ref{SS} that $H:=\langle h^{-1} A\rangle$ is an infinite cyclic group for some $h\in A$. Here we have $(h^{-1} \alpha) (\beta h)=1$. Now assume that  $|supp(\beta)|$ is minimum with respect to the property $ \alpha\beta=1$.   
 It follows from  Lemma \ref{supp-unit} that  $ h^{-1} \alpha, \beta h$ belong to the group algebra on $H$.
 Now $(h^{-1} \alpha) (\beta h)=1$ contradicts the fact that the group algebra of an infinite cyclic group has no non-trivial units (see \cite[Theorem 26.2]{pass2}). This completes  the proof.
\end{proof}

\begin{prop}\label{simple-unit}
Let $\alpha$  be an element in the group algebra of a torsion-free group such that $\alpha \beta=1$ for some  element $\beta$ of the group algebra. If  $|supp(\alpha)|=3$ and $S:=\{h^{-1} h' \;|\; h,h'\in supp(\alpha), h \neq h' \}$, then
$U(\alpha,\beta)$ is isomorphic to the induced subgraph of the Cayley graph $Cay(G,S)$ on $supp(\beta)$.
\end{prop}
\begin{proof}
The proof is similar to that of Proposition \ref{simple}.
\end{proof}

%-------------------------------------------------------------------------------------------------------------------------------

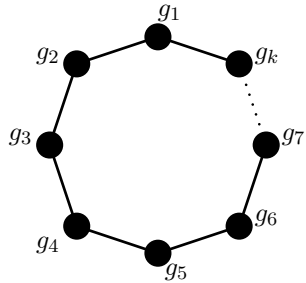
\begin{figure}[ht]
\psscalebox{0.9 0.9} % Change this value to rescale the drawing.
{
\begin{pspicture}(0,-2.105)(4.72,2.105)
\psdots[linecolor=black, dotsize=0.4](2.4,1.495)
\psdots[linecolor=black, dotsize=0.4](1.2,1.095)
\psdots[linecolor=black, dotsize=0.4](0.8,-0.105)
\psdots[linecolor=black, dotsize=0.4](1.2,-1.305)
\psdots[linecolor=black, dotsize=0.4](2.4,-1.705)
\psdots[linecolor=black, dotsize=0.4](3.6,-1.305)
\psdots[linecolor=black, dotsize=0.4](4.0,-0.105)
\psdots[linecolor=black, dotsize=0.4](3.6,1.095)
\psline[linecolor=black, linewidth=0.04](3.6,1.095)(2.4,1.495)(1.2,1.095)(0.8,-0.105)(1.2,-1.305)(2.4,-1.705)(3.6,-1.305)(4.0,-0.105)
\psline[linecolor=black, linewidth=0.04, linestyle=dotted, dotsep=0.10583334cm](3.6,1.095)(4.0,-0.105)
\rput[bl](2.4,1.73){$g_1$}
\rput[bl](0.6,1.095){$g_2$}
\rput[bl](0.2,-0.105){$g_3$}
\rput[bl](0.6,-1.705){$g_4$}
\rput[bl](2.5,-2.105){$g_5$}
\rput[bl](3.83,-1.305){$g_6$}
\rput[bl](4.23,-0.105){$g_7$}
\rput[bl](3.83,1.095){$g_k$}
\end{pspicture}
}
\caption{A cycle of length $k$ in a zero-divisor graph or a unit graph}\label{f-cycle}
\end{figure}

\begin{defn}\label{tuples}
{\rm
Let $\Gamma$ be a zero-divisor graph or a unit graph on a pair of elements $(\alpha,\beta)$ in a group algebra such that $|supp(\alpha)|=3$. Let $C$ be a cycle of length $k$ in $\Gamma$ as Figure \ref{f-cycle} and suppose that $\{g_1,g_2,g_3,\dots,g_k\}\subseteq supp(\beta)$ is the vertex set of $C$ such that $g_i\sim g_{i+1}$ for all $i\in \{1,2,\ldots, k-1\}$ and $g_1\sim g_k$. By an arrangement $\ell$ of the vertex set $C$, we mean  a sequence of all vertices as $x_1,\dots, x_k$  such that $x_i\sim x_{i+1}$ for all $i\in \{1,2,\ldots, k-1\}$ and $x_1\sim x_k$.
  There exist $a_1, b_1,a_2,b_2,\dots,a_k,b_k \in supp(\alpha)$  satisfying the following relations:
\begin{equation}\label{e-cycle1}
R:= \left\{
\begin{array}{l}
a_1g_1=b_1g_2\\
a_2g_2=b_2g_3\\
\vdots\\
a_kg_k=b_kg_1\\
\end{array} \right.
\end{equation}
 We assign the $2k$-tuple $T^{\ell}_C=[a_1,b_1,a_2,b_2,\dots,a_k,b_k]$ to $C$ corresponding to the above arrangement $\ell$ of  the vertex set of $C$. We denote by $R(T^{\ell}_C)$ the above set $R$ of relations.  

It can be derived from the relations \ref{e-cycle1} that 
$r(T^{\ell}_C):=(a_1^{-1}b_1)(a_2^{-1}b_2)\cdots(a_k^{-1}b_k)$ is equal to $1$. It follows from Lemma \ref{size of S} or Lemma \ref{size of S for Unit} that if $T^{\ell'}_C=[a'_1,b'_1,\dots,a'_k,b'_k]$ is the $2k$-tuple  of $C$ corresponding to another arrangement $\ell'$ of the vertex set of $C$, then $T^{\ell'}_C$ is one of  the following $2k$-tuples:

\begin{equation*}
\begin{matrix}
[a_1,b_1,a_2,b_2,\ldots, a_{k-1},b_{k-1}, a_k,b_k],\\
[a_k,b_k,a_1,b_1,\ldots, a_{k-2},b_{k-2}, a_{k-1},b_{k-1}],\\
\vdots\\
[a_2,b_2,a_3,b_3,\ldots, a_{k},b_{k},a_1,b_1],\\
[b_1,a_1,b_k,a_k, \ldots, b_3,a_3,b_2,a_2],\\
[b_2,a_2,b_1,a_1,\ldots, b_4,a_4,b_3,a_3],\\
\vdots\\
[b_k,a_k,b_{k-1},a_{k-1},\ldots, b_2,a_2,b_1,a_1].
\end{matrix}
\end{equation*}
The set of all such $2k$-tuples will be denoted by $\mathcal{T}(C)$. Also, $\mathcal{R}(C)=\{R(T) | T\in \mathcal{T}(C)\}$.
}
\end{defn}
\begin{defn}
{\rm
Let $\Gamma$ be a zero-divisor graph or a unit graph on a pair of elements $(\alpha,\beta)$ in a group algebra such that $|supp(\alpha)|=3$. Let $C$ be a cycle of length $k$ in $\Gamma$. Since $r(T_1)=1$ if and only if $r(T_2)=1$, for all $T_1,T_2\in \mathcal{T}(C)$, a member of $\{r(T) | T\in \mathcal{T}(C)\}$ is given as a representative and denoted by $r(C)$. Also, $r(C)=1$ is called the relation of $C$. 
}
\end{defn}

\begin{defn}\label{def-equ}
{\rm
Let $\Gamma$ be a zero-divisor graph or a unit graph on a pair of elements $(\alpha,\beta)$ in a group algebra such that $|supp(\alpha)|=3$. Let $C$ and $C'$ be two cycles of length $k$ in $\Gamma$.  We say that these two cycles  are equivalent, if $\mathcal{T}(C)\cap \mathcal{T}(C')\not=\varnothing$.
}
\end{defn}
\begin{rem}{\rm
Let $\Gamma$ be a zero-divisor graph or a unit graph on a pair of elements $(\alpha,\beta)$ in a group algebra such that $|supp(\alpha)|=3$. If $C$ and $C'$ are two equivalent cycles of length $k$ in $\Gamma$, then  $\mathcal{T}(C)=\mathcal{T}(C')$.
}\end{rem}

%-------------------------------------------------------------------------------------------------------------------------------
\begin{figure}[ht]
\centering
\begin{tikzpicture}[scale=0.7]
\node (2) [circle, minimum size=3pt, fill=black, line width=0.625pt, draw=black] at (125.0pt, -125.0pt)  {};
\node (3) [circle, minimum size=3pt, fill=black, line width=0.625pt, draw=black] at (50.0pt, -125.0pt)  {};
\node (1) [circle, minimum size=3pt, fill=black, line width=0.625pt, draw=black] at (87.5pt, -62.5pt)  {};
\draw [line width=1.25, color=black] (1) to  (2);
\draw [line width=1.25, color=black] (1) to  (3);
\draw [line width=1.25, color=black] (3) to  (2);
\node at (125.0pt, -141.375pt) {\textcolor{black}{$g_j$}};
\node at (50.0pt, -141.375pt) {\textcolor{black}{$g_k$}};
\node at (87.5pt, -47.125pt) {\textcolor{black}{$g_i$}};
\end{tikzpicture}
\caption{A triangle $C_3$ in a zero-divisor graph or a unit graph}\label{f-1}
\end{figure}
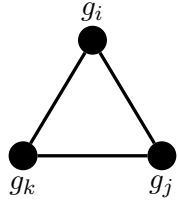

\begin{rem}\label{r-C3}
{\rm Let $\Gamma$ be a zero-divisor graph or a unit graph on a pair of elements $(\alpha,\beta)$ in a group algebra such that $|supp(\alpha)|=3$. If $C$ is a cycle of length $3$ (a triangle)  in $\Gamma$ and $T\in \mathcal{T}(C)$, then $T=[a_1,b_1,a_2,b_2,a_3,b_3]$ with exactly one of the following conditions:
\begin{equation}\label{e-1}
a_1\not=b_1\not=a_2 \not=b_2\not=a_3\not=b_3\not=a_1
\end{equation}
\begin{equation}\label{e-K3 over F}
a_1\not=b_1=a_2\not=b_2=a_3\not=b_3=a_1.
\end{equation}
We note that if the condition \ref{e-K3 over F} is satisfied, then $\{a_1,b_1,a_2,b_2,a_3,b_3\}=\{a_1,b_1,b_2\}=supp(\alpha)$.
}
\end{rem}
\begin{defn}
{\rm
Let $\Gamma$ be a zero-divisor graph or a unit graph on a pair of elements $(\alpha,\beta)$ in a group algebra such that $|supp(\alpha)|=3$.  Suppose that $C$ is a triangle in $\Gamma$ with $T=[a_1,b_1,a_2,b_2,a_3,b_3]\in \mathcal{T}(C)$. If the condition \ref{e-1} or \ref{e-K3 over F} is satisfied, we call $C$ a triangle of type (I) or type (II), respectively.
}
\end{defn}

We need the following remark in the sequel and we apply it without referring. 

\begin{rem}\label{r-baum}
{\rm
For integers $m$ and $n$, the Baumslag-Solitar group $BS(m,n)$ is the group given by the presentation $\langle a,b\mid b a^m b^{-1}=a^n\rangle $. Baumslag-Solitar groups are HNN-extensions of an infinite cyclic group and so they are torsion-free by \cite[Theorem 6.4.5]{robinson}. Since Baumslag-Solitar groups are one-relator, it follows from \cite{brod} that they are locally indicable. Now \cite[Lemmas 1.8 (iii) and 1.9]{pass1} imply that they satisfy both Conjectures \ref{conj-zero} and \ref{conj-unit}. It follows from the first part of the proof of \cite[Theorem 3.1]{pascal} that every torsion-free quotient of $BS(1,n)$ is either  abelian or it is isomorphic to $BS(1,n)$ itself. Therefore every torsion-free quotient of $B(1,n)$ satisfies  both Conjectures \ref{conj-zero} and \ref{conj-unit}.}
\end{rem}
\begin{lem}\label{Z-C3}
Suppose that $\alpha$ and $\beta$ are non-zero elements of a group algebra of a torsion-free group $G$ such that $|supp(\alpha)|=3$, $\alpha\beta=0$ and if  $\alpha \beta'=0$ for some non-zero element $\beta'$ of the group algebra, then  $|supp(\beta')|\geq |supp(\beta)|$. Then there is no triangle of type (I) in $Z(\alpha,\beta)$.
\end{lem}
\begin{proof}
By Remark \ref{r-G}, one may assume that $1\in supp(\alpha)$ and $G=\langle supp(\alpha) \rangle$. Let $supp(\alpha)=\{1,h_2,h_3\}$. Suppose, for a contradiction, that $C$ is a triangle of type (I) in $Z(\alpha,\beta)$. Then it is easy to see that there are $13$ non-equivalent  cases for $C$, and $r(C)$ corresponding to such cases are one of the members of the following set:
\begin{align*}
A=\{ h_2^3, h_2^2h_3,h_2^2h_3^{-1}h_2, h_2h_3^2, h_2h_3h_2^{-1}h_3, h_2h_3^{-1}h_2^{-1}h_3, h_2h_3^{-2}h_2,\\ h_2h_3^{-1}h_2h_3,(h_2h_3^{-1})^2h_2, h_3^3, h_3^2h_2^{-1}h_3, (h_3h_2^{-1})^2h_3, (h_2^{-1}h_3)^3 \}
\end{align*}
Therefore, $G$ is the group generated by $h_2$ and $h_3$ with a relation $a=1$ for some $a\in A$. We will arrive to a  contradiction since  such a  group $G$ has at least one of the following properties:
\begin{enumerate}
\item 
It is an abelian group,
\item
It is a quotient of $BS(1,k)$ or $BS(k,1)$ where $k\in \{-1,1\}$,
\item
It has a non-trivial torsion element. 
\end{enumerate}
This completes the proof.
\end{proof}

\begin{lem}\label{U-C3}
Suppose that $\alpha$ and $\beta$ are elements of a group algebra of a torsion-free group $G$ such that $|supp(\alpha)|=3$, $\alpha\beta=1$ and if  $\alpha \beta'=1$ for some element $\beta'$ of the group algebra, then  $|supp(\beta')|\geq |supp(\beta)|$. Then there is no triangle of type (I) in $U(\alpha,\beta)$.
\end{lem}
\begin{proof}
The proof is similar to that of Lemma \ref{Z-C3}.
\end{proof}
%K3-K3------------------------------------------------------------------------------------------------------------------------
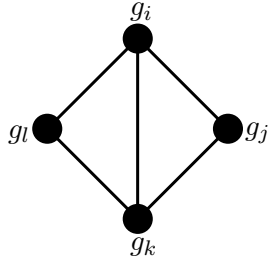
\begin{figure}[h]
\centering
\psscalebox{1.0 1.0} % Change this value to rescale the drawing.
{
\begin{pspicture}(0,-1.905)(3.91,1.905)
\psdots[linecolor=black, dotsize=0.4](2.0,1.295)
\psdots[linecolor=black, dotsize=0.4](2.0,-1.105)
\psdots[linecolor=black, dotsize=0.4](0.8,0.095)
\psdots[linecolor=black, dotsize=0.4](3.2,0.095)
\psline[linecolor=black, linewidth=0.04](2.0,1.295)(0.8,0.095)(2.0,-1.105)(3.2,0.095)(2.0,1.295)(2.0,-1.105)(2.0,-1.105)
\rput[bl](1.9,1.53){$g_i$}
\rput[bl](3.43,-0.1){$g_j$}
\rput[bl](1.9,-1.6){$g_k$}
\rput[bl](0.28,-0.1){$g_l$}
\end{pspicture}
}
\caption{Two triangles with a common edge in a zero-divisor graph or a unit graph}\label{f-K3-K3}
\end{figure}

\begin{thm}\label{Z-K3-K3}
Suppose that $\alpha$ and $\beta$ are non-zero elements of a group algebra of a torsion-free group such that $|supp(\alpha)|=3$, $\alpha\beta=0$ and if  $\alpha \beta'=0$ for some non-zero element $\beta'$ of the group algebra, then  $|supp(\beta')|\geq |supp(\beta)|$. Then $Z(\alpha,\beta)$ contains no subgraphs isomorphic to the graph in Figure \ref{f-K3-K3} i.e. two triangles with one edge in common. 
\end{thm}
\begin{proof}
Suppose, for a contradiction, that $Z(\alpha,\beta)$ contains two triangles with one edge in common as Figure \ref{f-K3-K3} for some distinct elements $g_i,g_j,g_k,g_l\in supp(\beta)$. Then $Z(\alpha,\beta)$ contains two triangles $C$ and $C'$ with vertex sets $\{g_i,g_j,g_k\}$ and $\{g_i,g_l,g_k\}$. Therefore by Remark \ref{r-C3} and Lemma \ref{Z-C3}, $C$ and $C'$ are triangles of type (II) and there are $R_C\in \mathcal{R}(C)$ and  $R_{C'}\in \mathcal{R}(C')$ as $a_1g_i=b_1g_k=c_1g_j$  and $a_1g_i=b_1g_k=c_2g_l$, respectively, where $\{a_1,b_1,c_1\}=\{a_1,b_1,c_2\}= supp(\alpha)$. Hence, $c_1=c_2$ which implies $g_j=g_l$, a contradiction. This completes the proof.
\end{proof}

\begin{thm}\label{U-K3-K3}
Suppose that $\alpha$ and $\beta$ are elements of a group algebra of a torsion-free group such that $|supp(\alpha)|=3$, $\alpha\beta=1$ and if  $\alpha \beta'=1$ for some element $\beta'$ of the group algebra, then  $|supp(\beta')|\geq |supp(\beta)|$. Then $U(\alpha,\beta)$ contains no subgraphs isomorphic to the graph in Figure \ref{f-K3-K3} i.e. two triangles with one edge in common. 
\end{thm}
\begin{proof}
The proof is similar to that of Theorem \ref{Z-K3-K3}.
\end{proof}

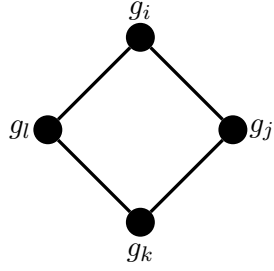
\begin{figure}[ht]
\centering
\begin{tikzpicture}[scale=0.7]
\node (1) [circle, minimum size=3pt, fill=black, line width=0.625pt, draw=black] at (75.0pt, -25.0pt)  {};
\node (2) [circle, minimum size=3pt, fill=black, line width=0.625pt, draw=black] at (125.0pt, -75.0pt)  {};
\node (3) [circle, minimum size=3pt, fill=black, line width=0.625pt, draw=black] at (25.0pt, -75.0pt)  {};
\node (5) [circle, minimum size=3pt, fill=black, line width=0.625pt, draw=black] at (75.0pt, -125.0pt)  {};
\draw [line width=1.25, color=black] (1) to  (2);
\draw [line width=1.25, color=black] (1) to  (3);
\draw [line width=1.25, color=black] (2) to  (5);
\draw [line width=1.25, color=black] (3) to  (5);
\node at (75.0pt, -10.625pt) {\textcolor{black}{$g_i$}};
\node at (140.875pt, -75.0pt) {\textcolor{black}{$g_j$}};
\node at (10.125pt, -75.0pt) {\textcolor{black}{$g_l$}};
\node at (75.0pt, -141.375pt) {\textcolor{black}{$g_k$}};
\end{tikzpicture}
\caption{An square in a zero-divisor graph or a unit graph}\label{f-C4}
\end{figure}

\begin{rem}\label{r-Z-K3-K3}
{\rm Suppose that $\alpha$ and $\beta$ are non-zero elements of a group algebra of a torsion-free group such that $|supp(\alpha)|=3$, $\alpha\beta=0$ and if  $\alpha \beta'=0$ for some non-zero element $\beta'$ of the group algebra, then  $|supp(\beta')|\geq |supp(\beta)|$. By Theorem \ref{Z-K3-K3}, if $C$ is a cycle of length $4$ (an square)  in $Z(\alpha,\beta)$ as Figure \ref{f-C4} with vertex set  $g_i,g_j,g_k,g_l\in supp(\beta)$, then $g_i \not\sim g_k$ and $g_j \not\sim g_l$. So, if $T\in \mathcal{T}(C)$, then $T=[a_1,b_1,a_2,b_2,a_3,b_3,a_4,b_4]$ with the condition $a_1\not= b_1\not= a_2 \not= b_2\not= a_3\not= b_3\not= a_4\not= b_4\not= a_1$.
}
\end{rem}

\begin{rem}\label{r-U-K3-K3}
{\rm Suppose that $\alpha$ and $\beta$ are elements of a group algebra of a torsion-free group such that $|supp(\alpha)|=3$, $\alpha\beta=1$ and if  $\alpha \beta'=1$ for some element $\beta'$ of the group algebra, then  $|supp(\beta')|\geq |supp(\beta)|$. By Theorem \ref{U-K3-K3}, if $C$ is an square in $U(\alpha,\beta)$ as Figure \ref{f-C4} with vertex set $g_i,g_j,g_k,g_l\in supp(\beta)$, then $g_i \not\sim g_k$ and $g_j \not\sim g_l$. So, if $T\in \mathcal{T}(C)$, then $T=[a_1,b_1,a_2,b_2,a_3,b_3,a_4,b_4]$ with the condition $a_1\not= b_1\not= a_2 \not= b_2\not= a_3\not= b_3\not= a_4\not= b_4\not= a_1$.
}
\end{rem}

%the corresponding relations of such an square are as follows:
%\begin{equation}\label{e-4}
%\left\{
%\begin{array}{l}
%a_1g_i=b_1g_j\\
%a_2g_j=b_2g_k\\
%a_3g_k=b_3g_l\\
%a_4g_l=b_4g_i,\\
%\text{where } a_s,b_s \in supp(\alpha) \text{ for all } s\in \{1,2,3,4\}\\
%\text{and } a_1\not= b_1\not= a_2 \not= b_2\not= a_3\not= b_3\not= a_4\not= b_4\not= a_1.
%\end{array} \right.
%\end{equation} 

It is proved in \cite[Theorem 4.2]{pascal} that a triangle is a forbidden subgraph for any zero-divisor graph of length $3$ over the field $\mathbb{F}_2$ on any torsion-free group (see below, Theorem \ref{thm-graph}). To finish this section, we consider zero-divisor graphs and unit graphs of length $3$ over a field $\mathbb{F}$ and on any torsion-free group containing a subgraph isomorphic to an square.  We have not been able to prove  that squares are forbidden subgraphs for such latter graphs even for the case that $\mathbb{F}=\mathbb{F}_2$. However we show that by existence of squares, zero-divisor graphs (unit graphs, respectively) of length $3$ over a field $\mathbb{F}$ and on any torsion-free group give us certain slightly significant relations on elements of the support of a possible zero-divisor (on elements of the support of a possible unit) (see below,  Theorems \ref{Z-C4} and \ref{U-C4}).

\begin{table}[h]
\centering
\caption{The possible relations of an square in a zero-divisor graph or a unit graph of length $3$ on a torsion-free group}\label{tab-C4}
\begin{tabular}{|c|l|l||c|l|l|}\hline
$n$&$R$&$ \ E \ $&$n$&$R$&$ \ E \ $\\\hline
$1$&$h_2^4=1$& $T $&$19$&$h_2 h_3^{-1} (h_2^{-1} h_3)^2=1$& $BS(2,1) $\\
$2$&$h_2^3 h_3=1$& $ A$&$20$&$h_2 h_3^{-2} h_2^{-1} h_3=1$& $BS(1,2)$\\
$3$&$h_2^3 h_3^{-1} h_2=1$& $ A$&$21$&$h_2 h_3^{-3} h_2=1$& $ *$\\
$4$&$h_2^2 h_3^2=1$& $BS(1,-1) $&$22$&$h_2 h_3^{-2} h_2 h_3=1$& $ *$\\
$5$&$h_2^2 h_3 h_2^{-1} h_3=1$& $ *$&$23$&$h_2 h_3^{-1} (h_3^{-1} h_2)^2=1$& $BS(-2,1) $\\
$6$&$h_2^2 h_3^{-1} h_2^{-1} h_3=1$& $BS(1,2) $&$24$&$(h_2 h_3^{-1} h_2)^2=1$& $ A$\\
$7$&$h_2^2 h_3^{-2} h_2=1$& $* $&$25$&$h_2 h_3^{-1} h_2 h_3^2=1$& $ *$\\
$8$&$h_2^2 h_3^{-1} h_2 h_3=1$& $ BS(1,-2)$&$26$&$h_2 h_3^{-1} h_2 h_3 h_2^{-1} h_3=1$& $ *$\\
$9$&$h_2 (h_2 h_3^{-1})^2 h_2=1$& $BS(1,-1) $&$27$&$(h_2 h_3^{-1})^2 h_2^{-1} h_3=1$& $BS(2,1) $\\
$10$&$(h_2 h_3)^2=1$& $ A$&$28$&$(h_2 h_3^{-1})^2 h_3^{-1} h_2=1$& $ BS(1,-2)$\\
$11$&$h_2 h_3 h_2 h_3^{-1} h_2=1$& $ BS(-2,1)$&$29$&$(h_2 h_3^{-1})^2 h_2 h_3=1$& $* $\\
$12$&$h_2 h_3^3=1$& $ A$&$30$&$(h_2 h_3^{-1})^3 h_2=1$& $A $\\
$13$&$h_2 h_3^2 h_2^{-1} h_3=1$& $BS(1,-2) $&$31$&$h_3^4=1$& $ T$\\
$14$&$h_2 h_3 h_2^{-2} h_3=1$& $ *$&$32$&$h_3^3 h_2^{-1} h_3=1$& $ A$\\
$15$&$h_2 h_3 h_2^{-1} h_3^{-1} h_2=1$& $ BS(2,1)$&$33$&$h_3 (h_3 h_2^{-1})^2 h_3=1$& $BS(1,-1) $\\
$16$&$h_2 h_3 h_2^{-1} h_3^2=1$& $BS(-2,1) $&$34$&$(h_3 h_2^{-1} h_3)^2=1$& $ A$\\
$17$&$h_2 (h_3 h_2^{-1})^2 h_3=1$& $* $&$35$&$(h_3 h_2^{-1})^3 h_3=1$& $ A$\\
$18$&$h_2 h_3^{-1} h_2^{-1} h_3^2=1$& $BS(2,1) $&$36$&$(h_2^{-1} h_3)^4=1$& $A $\\\hline
\end{tabular}
\end{table}

\begin{thm}\label{Z-C4-1}
Suppose that $\alpha$ and $\beta$ are non-zero elements of a group algebra of a torsion-free group $G$ such that $|supp(\alpha)|=3$, $\alpha\beta=0$ and if  $\alpha \beta'=0$ for some non-zero element $\beta'$ of the group algebra, then  $|supp(\beta')|\geq |supp(\beta)|$. If $Z(\alpha,\beta)$ contains an square $C$, then there are $9$ non-equivalent cases for $C$ and $r(C)=1$ is one of the relations $5$, $7$, $14$, $17$, $21$, $22$, $25$, $26$ or $29$  in Table \ref{tab-C4}, for $\{h_2,h_3\}=supp(\alpha)\setminus \{1\}$.
\end{thm}
\begin{proof} 
By Remark \ref{r-G}, one may assume that $1\in supp(\alpha)$ and $G=\langle supp(\alpha) \rangle$. Let $supp(\alpha)=\{1,h_2,h_3\}$ and $C$ be an square in $Z(\alpha,\beta)$. By Remark \ref{r-Z-K3-K3}, if $T\in \mathcal{T}(C)$ then $T=[a_1,b_1,a_2,b_2,a_3,b_3,a_4,b_4]$ with the condition $a_1\not= b_1\not= a_2 \not= b_2\not= a_3\not= b_3\not= a_4\not= b_4\not= a_1$. Therefore by using GAP \cite{gap}, there are $36$ non-equivalent cases for $C$. The relations of such non-equivalent cases are listed in the column labelled by $R$ of Table \ref{tab-C4}. It is easy to see that each of such relations, except the $9$ cases marked by ``$*$''s in the column labelled by $E$ of Table \ref{tab-C4}, gives a contradiction because the group $G$ generated by $h_2$ and $h_3$ with one of such relations has at least one of the following properties:
\begin{enumerate}
\item 
It is an abelian group,
\item
It is a quotient of $BS(1,k)$ or $BS(k,1)$ where $k\in \{-2,-1,1,2\}$,
\item
It has a non-trivial torsion element. 
\end{enumerate}
Each relation which leads to being $G$ an abelian group or $G$ having  a non-trivial torsion element is marked by an $A$ or a $T$ in the column labelled by $E$, respectively. Also, if $G$ is a quotient of a Baumslag-Solitar group, then it is denoted by $BS(1,k)$ or $BS(k,1)$ for $k\in \{-2,-1,1,2\}$,  in the column $E$. This completes the proof.
\end{proof}

\begin{thm}\label{U-C4-1}
Suppose that $\alpha$ and $\beta$ are elements of a group algebra of a torsion-free group $G$ such that $|supp(\alpha)|=3$, $\alpha\beta=1$ and if  $\alpha \beta'=1$ for some element $\beta'$ of the group algebra, then  $|supp(\beta')|\geq |supp(\beta)|$. If $U(\alpha,\beta)$ contains an square $C$, then there are $9$ non-equivalent cases for $C$ and $r(C)=1$ is one of the relations $5$, $7$, $14$, $17$, $21$, $22$, $25$, $26$ or $29$  in Table \ref{tab-C4}, for $\{h_2,h_3\}=supp(\alpha)\setminus \{1\}$.
\end{thm}
\begin{proof}
The proof is similar to that of Theorem \ref{Z-C4-1}.
\end{proof}

\begin{thm}\label{Z-C4}
Suppose that $\alpha$ and $\beta$ are non-zero elements of a group algebra of a torsion-free group $G$ over a field $\mathbb{F}$ such that $|supp(\alpha)|=3$, $\alpha\beta=0$ and if  $\alpha \beta'=0$ for some non-zero element $\beta'$ of the group algebra, then  $|supp(\beta')|\geq |supp(\beta)|$. If $Z(\alpha,\beta)$ contains an square $C$, then there exist non-trivial group elements $x$ and $y$ such that  $x^2=y^3$ and either $\{1,x,y\}$ or $\{1,y,y^{-1}x\}$ is the support of a zero divisor in $\mathbb{F}[G]$.
\end{thm}
\begin{proof}
By Remark \ref{r-G}, one may assume that $1\in supp(\alpha)$ and $G=\langle supp(\alpha) \rangle$. Let $\alpha=\alpha_1\cdot 1 +\alpha_2h_2+\alpha_3h_3$.  By Theorem \ref{Z-C4-1}, $r(C)$ is one of the relations $5$, $7$, $14$, $17$, $21$, $22$, $25$, $26$ or $29$  in Table \ref{tab-C4}. So, we have the followings:
\begin{enumerate}
\item[(5)]
$h_2^2 h_3 h_2^{-1} h_3=1$: Let $x=h_2^{-1}h_3$ and $y=h_2^{-1}$. So, $x^2=y^3$. Also since $\alpha\beta=0$, we have $h_2^{-1}(\alpha\beta)=0$. So,  $\alpha_2\cdot 1+\alpha_3x+\alpha_1y$ is a zero divisor with the support $\{1,x,y\}$.
\item[(7)]
$h_2^2 h_3^{-2} h_2=1$: Let $x=h_3$ and $y=h_2$. So $x^2=y^3$ and $\alpha_1\cdot 1 +\alpha_2y+\alpha_3x$ is a zero divisor with the support $\{1,x,y\}$.
\item[(14)]
$h_2 h_3 h_2^{-2} h_3=1$: Let $x=h_2h_3$ and $y=h_2$. So $x^2=y^3$ and $\alpha_1\cdot 1+\alpha_2y+\alpha_3y^{-1}x=\alpha_1\cdot 1 +\alpha_2h_2+\alpha_3h_3$ is a zero divisor with the support $\{1,y,y^{-1}x\}$.
\item[(17)] 
$h_2 (h_3 h_2^{-1})^2 h_3=1$: Let $x=h_2^{-1}$ and $y=h_3h_2^{-1}$. So $x^2=y^3$. Also since $\alpha\beta=0$, we have $\alpha h_2^{-1}h_2\beta=0$. So, $\alpha_2\cdot 1 +\alpha_1x+\alpha_3y$ is a zero divisor with the support $\{1,x,y\}$.
\item[(21)]
$h_2 h_3^{-3} h_2=1$: By interchanging $h_2$ and $h_3$ in (7) and with the same discussion, the statement is true.
\item[(22)]
$h_2 h_3^{-2} h_2 h_3=1$: By interchanging $h_2$ and $h_3$ in (14) and with the same discussion, the statement is true.
\item[(25)]
$h_2 h_3^{-1} h_2 h_3^2=1$: By interchanging $h_2$ and $h_3$ in (5) and with the same discussion, the statement is true.
\item[(26)] 
$h_2 h_3^{-1} h_2 h_3 h_2^{-1} h_3=1$: Let $x=h_2h_3^{-1}h_2$ and $y=h_3^{-1}h_2$. So, $x^2=y^3$. Also since $\alpha_1\cdot 1 +\alpha_2h_2+\alpha_3h_3=\alpha_2xy^{-1}+\alpha_3xy^{-2}+\alpha_1 x^2y^{-3}$, we have $x^{-1}(\alpha_2xy^{-1}+\alpha_3xy^{-2}+\alpha_1 x^2y^{-3})yy^{-1}\beta=0$. Therefore, $\alpha_2\cdot 1+\alpha_3y^{-1}+\alpha_1 xy^{-2}$ is also a zero divisor with the support of size $3$. Furthermore, $(\alpha_3\cdot 1+\alpha_2y+\alpha_1y^{-1}x)(y^{-3}\beta)=(y^{-1}(\alpha_2\cdot 1+\alpha_3y^{-1}+\alpha_1 xy^{-2})y^2)(y^{-3}\beta)=0$. Hence, $\alpha_3\cdot 1+\alpha_2y+\alpha_1y^{-1}x$ is a zero divisor with the support $\{1,y,y^{-1}x\}$.
\item[(29)]
$(h_2 h_3^{-1})^2 h_2 h_3=1$: By interchanging $h_2$ and $h_3$ in (17) and with the same discussion, the statement is true.
\end{enumerate}
This completes the proof.
\end{proof}

\begin{thm}\label{U-C4}
Suppose that $\alpha$ and $\beta$ are elements of a group algebra of a torsion-free group $G$ over a field $\mathbb{F}$ such that $|supp(\alpha)|=3$, $\alpha\beta=1$ and if  $\alpha \beta'=1$ for some element $\beta'$ of the group algebra, then  $|supp(\beta')|\geq |supp(\beta)|$. If $U(\alpha,\beta)$ contains an square, then there exist non-trivial group elements $x$ and $y$ such that  $x^2=y^3$ and either $\{1,x,y\}$ or $\{1,y,y^{-1}x\}$ is the support of a unit in $\mathbb{F}[G]$.
\end{thm}
\begin{proof}
The proof is similar to that of Theorem \ref{Z-C4}.
\end{proof}

In the following, we discuss about the existence of two squares in a zero-divisor graph or a unit graph of length $3$ over an arbitrary field and on a torsion-free group. 
\begin{lem}\label{Z-C4,C4}
Suppose that $\alpha$ and $\beta$ are non-zero elements of a group algebra of a torsion-free group $G$ such that $|supp(\alpha)|=3$, $\alpha\beta=0$ and if  $\alpha \beta'=0$ for some non-zero element $\beta'$ of the group algebra, then  $|supp(\beta')|\geq |supp(\beta)|$. Every two squares of $Z(\alpha,\beta)$ are equivalent.
\end{lem}
\begin{proof}
By Remark \ref{r-G}, one may assume that $1\in supp(\alpha)$ and $G=\langle supp(\alpha) \rangle$. Let $supp(\alpha)=\{1,h_2,h_3\}$. By Theorem \ref{Z-C4-1}, if there exist two squares in  $Z(\alpha,\beta)$, then these two cycles must be one of the $9$ non-equivalent cases with the relations $5$, $7$, $14$, $17$, $21$, $22$, $25$, $26$ or $29$  in Table \ref{tab-C4}. We may choose two relations similar to or different from each other.  When choosing two relations different from each other, there are $\binom{9}{2}= 36$ cases. Using GAP \cite{gap}, each group with two generators $h_2$ and $h_3$, and two of the relations of the $36$ latter cases is finite and solvable, that is a contradiction. So, if there exist two squares in the graph $Z(\alpha,\beta)$, then such cycles must be equivalent.
\end{proof}

\begin{lem}\label{U-C4,C4}
Suppose that $\alpha$ and $\beta$ are elements of a group algebra of a torsion-free group $G$ such that $|supp(\alpha)|=3$, $\alpha\beta=1$ and if  $\alpha \beta'=1$ for some element $\beta'$ of the group algebra, then  $|supp(\beta')|\geq |supp(\beta)|$.  Every two squares of $U(\alpha,\beta)$ are equivalent.
\end{lem}
\begin{proof}
The proof is similar to that of Lemma \ref{Z-C4,C4}.
\end{proof}

%1) K2,3------------------------------------------------------------------------------------------------------------------------
\begin{figure}[ht]
\centering
\begin{tikzpicture}[scale=0.7]
\node (1) [circle, minimum size=3pt, fill=black, line width=0.625pt, draw=black] at (75.0pt, -25.0pt)  {};
\node (2) [circle, minimum size=3pt, fill=black, line width=0.625pt, draw=black] at (125.0pt, -75.0pt)  {};
\node (3) [circle, minimum size=3pt, fill=black, line width=0.625pt, draw=black] at (25.0pt, -75.0pt)  {};
\node (4) [circle, minimum size=3pt, fill=black, line width=0.625pt, draw=black] at (75.0pt, -75.0pt)  {};
\node (5) [circle, minimum size=3pt, fill=black, line width=0.625pt, draw=black] at (75.0pt, -125.0pt)  {};
\draw [line width=1.25, color=black] (1) to  (2);
\draw [line width=1.25, color=black] (1) to  (4);
\draw [line width=1.25, color=black] (1) to  (3);
\draw [line width=1.25, color=black] (4) to  (5);
\draw [line width=1.25, color=black] (2) to  (5);
\draw [line width=1.25, color=black] (3) to  (5);
\node at (75.0pt, -10.625pt) {\textcolor{black}{$g_i$}};
\node at (143.875pt, -75.0pt) {\textcolor{black}{$g_m$}};
\node at (9.875pt, -75.0pt) {\textcolor{black}{$g_l$}};
\node at (90.875pt, -75.0pt) {\textcolor{black}{$g_j$}};
\node at (75.0pt, -141.375pt) {\textcolor{black}{$g_k$}};
\end{tikzpicture}
\caption{The complete bipartite graph $K_{2,3}$ in a zero-divisor graph or a unit graph}\label{f-K2,3}
\end{figure}
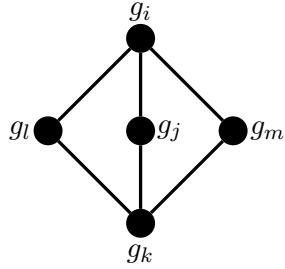
\begin{thm}\label{Z-K2,3}
Suppose that $\alpha$ and $\beta$ are non-zero elements of a group algebra of a torsion-free group $G$ such that $|supp(\alpha)|=3$, $\alpha\beta=0$ and if  $\alpha \beta'=0$ for some non-zero element $\beta'$ of the group algebra, then  $|supp(\beta')|\geq |supp(\beta)|$. Then  $Z(\alpha,\beta)$ contains no subgraph isomorphic to the complete bipartite graph $K_{2,3}$.
\end{thm}
\begin{proof}
Suppose, for a contradiction, that $Z(\alpha,\beta)$ contains $K_{2,3}$ as a subgraph. Then it contains two squares $C$ and $C'$ with two edges in common as Figure \ref{f-K2,3}, for some distinct elements $g_i,g_j,g_k,g_l,g_m\in supp(\beta)$. So, there are $T_C\in \mathcal{T}(C)$ and $T_{C'}\in \mathcal{T}(C')$ such that $T_C=[a_1,b_1,a_2,b_2,a_3,b_3,a_4,b_4]$ and $T_{C'}=[a_1,b_1,a_2,b_2,a_3',b_3',a_4',b_4']$. Also by Remark \ref{r-Z-K3-K3}, the following conditions are satisfied:
\begin{align}\label{e-7-2}
a_1\not= b_1\not= a_2 \not= b_2\not= a_3\not= b_3\not= a_4\not= b_4\not= a_1 \text{ and } b_2\not= a_3'\not= b_3'\not= a_4'\not= b_4'\not= a_1.
\end{align}

Since the graph with the vertex set $\{g_i,g_m,g_k,g_l\}$ in $K_{2,3}$ is also an square, by Remark \ref{r-Z-K3-K3} the following conditions are also  satisfied:
\begin{align}\label{e-5}
a_3 \neq a_3'.
\end{align}
\begin{align}\label{e-6}
b_4 \neq b_4'.
\end{align}

By Lemma \ref{Z-C4,C4}, the cycles $C$ and $C '$ are equivalent.  So, $T_{C'}$ must be in $\mathcal{T}(C)$. In the following, we show that this gives contradictions.
\begin{enumerate}
\item
Let $T_{C'}=[a_1,b_1,a_2,b_2,a_3,b_3,a_4,b_4]$: So $a_3=a_3'$, that is a contradiction with \ref{e-5}.
\item
Let $T_{C'}=[a_4,b_4,a_1,b_1,a_2,b_2,a_3,b_3]$: Therefore, $T_C=[a_1,b_1,a_1,b_1,a_3,b_3,a_1,b_1]$ and  $T_{C'}=[a_1,b_1,a_1,b_1,a_1,b_1,a_3,b_3]$. By \ref{e-5} and \ref{e-6}, we have $a_3\neq a_1$ and $b_3\neq b_1$. Also, in such $8$-tuples we have  $a_3\neq b_1$, $b_3\neq a_1$ and $a_1\neq b_1$. Therefore, $a_3=b_3$ since $a_3,b_3 \in supp(\alpha)$, that is a contradiction.
\item
Let $T_{C'}=[a_3,b_3,a_4,b_4,a_1,b_1,a_2,b_2]$: So $a_3=a_1=a_3'$, that is a contradiction with \ref{e-5}.
\item
Let $T_{C'}=[a_2,b_2,a_3,b_3,a_4,b_4,a_1,b_1]$: Therefore, $T_C=[a_1,b_1,a_1,b_1,a_1,b_1,a_4,b_4]$ and  $T_{C'}=[a_1,b_1,a_1,b_1,a_4,b_4,a_1,b_1]$. By \ref{e-5} and \ref{e-6}, we have $a_4\neq a_1$ and $b_4\neq b_1$. Also, in such $8$-tuples we have  $a_4\neq b_1$, $b_4\neq a_1$ and $a_1\neq b_1$. Therefore, $a_4=b_4$ since $a_4,b_4 \in supp(\alpha)$, that is a contradiction.
\item
Let $T_{C'}=[b_1,a_1,b_4,a_4,b_3,a_3,b_2,a_2]$: So $a_1=b_1$, that is a contradiction with \ref{e-7-2}.
\item
Let $T_{C'}=[b_2,a_2,b_1,a_1,b_4,a_4,b_3,a_3]$: So $b_1=a_2$, that is a contradiction with \ref{e-7-2}.
\item
Let $T_{C'}=[b_3,a_3,b_2,a_2,b_1,a_1,b_4,a_4]$: So $a_2=b_2$, that is a contradiction with  \ref{e-7-2}.
\item
Let $T_{C'}=[b_4,a_4,b_3,a_3,b_2,a_2,b_1,a_1]$: So $b_2=a_3$, that is a contradiction with  \ref{e-7-2}.
\end{enumerate}
This completes the proof.
\end{proof}

\begin{thm}\label{U-K2,3}
Suppose that $\alpha$ and $\beta$ are elements of a group algebra of a torsion-free group $G$ such that $|supp(\alpha)|=3$, $\alpha\beta=1$ and if  $\alpha \beta'=1$ for some element $\beta'$ of the group algebra, then  $|supp(\beta')|\geq |supp(\beta)|$. Then  $U(\alpha,\beta)$ contains no subgraph isomorphic to the complete bipartite graph $K_{2,3}$.
\end{thm}
\begin{proof}
The proof is similar to that of Theorem \ref{Z-K2,3}.
\end{proof}

In the next three sections, we discuss about zero-divisor graphs of length $3$ over  $\mathbb{F}_2$ on any torsion-free group and give some forbidden subgraphs for such graphs.

\section{\bf Zero-divisor graphs of length $3$ over $\mathbb{F}_2$ on any torsion-free group \\and some of their subgraphs containing an square}\label{S-C4}

Throughout this section suppose that $\alpha$ and $\beta$ are non-zero elements of the group algebra $ \mathbb{F}_2[G]$ of a torsion-free group $G$ such that $|supp(\alpha)|=3$, $\alpha\beta=0$ and if  $\alpha \beta'=0$ for some non-zero element $\beta'$ of the group algebra, then  $|supp(\beta')|\geq |supp(\beta)|$. By Remark \ref{r-G}, one may assume that $1 \in supp(\alpha)$ and $G=\langle supp(\alpha) \rangle$. Let $supp(\alpha)=\{ 1,h_2,h_3\}$ and $n:=|supp(\beta)|$. The following theorem is obtained in \cite{pascal}.

\begin{thm}[Theorem 4.2 of \cite{pascal}]\label{thm-graph}
The zero-divisor graph $Z(\alpha,\beta)$ is a connected simple cubic one containing no subgraph isomorphic to a triangle. 
\end{thm} 
\begin{proof}
The connectedness follows from Lemma \ref{Z-connect}. Also by Proposition \ref{simple}, the graph is simple. 

Furthermor, for $(h,g)\in supp(\alpha)\times supp(\beta)$, there is a unique $(h',g')\in supp(\alpha)\times supp(\beta)$ such that $(h,g)\not=(h',g')$ and $hg=h'g'$. So, $Z(\alpha,\beta)$ is a cubic graph. 

Suppose, for a contradiction, that $C$ is a triangle in $Z(\alpha,\beta)$. If $C$ is of type (II) and $T\in \mathcal{T}(C)$, then by Remark \ref{r-C3}, $T=[a_1,b_1,a_2,b_2,a_3,b_3]$ with the condition $a_1\not=b_1=a_2\not=b_2=a_3\not=b_3=a_1$ where $\{a_1,b_1,b_2\}=supp(\alpha)$. So, there are distinct elements $a,b,c\in supp(\beta)$ such that $a_1a=b_1b=b_2c$, that is a contradiction. Therefore, $C$ is of type (I) and so by Lemma \ref{Z-C3}, there is no triangle in $Z(\alpha,\beta)$.
\end{proof}

\begin{rem}\label{r-F2}
{\rm It follows from Theorem \ref{thm-graph} that $n=|supp(\beta)|$, which is the number of vertices of $Z(\alpha,\beta)$, is always an even number because the number of vertices of any simple cubic graph is even.}
\end{rem}

\begin{figure}[ht]
\centering
\begin{tikzpicture}[scale=0.9]
\node (2) [circle, minimum size=3pt, fill=black, line width=0.625pt, draw=black] at (100.0pt, -75.0pt)  {};
\node (4) [circle, minimum size=3pt, fill=black, line width=0.625pt, draw=black] at (100.0pt, -25.0pt)  {};
\node (1) [circle, minimum size=3pt, fill=black, line width=0.625pt, draw=black] at (50.0pt, -25.0pt)  {};
\node (3) [circle, minimum size=3pt, fill=black, line width=0.625pt, draw=black] at (50.0pt, -75.0pt)  {};
\node (5) [circle, minimum size=3pt, fill=black, line width=0.625pt, draw=black] at (150.0pt, -25.0pt)  {};
\node (6) [circle, minimum size=3pt, fill=black, line width=0.625pt, draw=black] at (150.0pt, -75.0pt)  {};
\draw [line width=1.25, color=black] (1) to  (3);
\draw [line width=1.25, color=black] (3) to  (2);
\draw [line width=1.25, color=black] (1) to  (4);
\draw [line width=1.25, color=black] (4) to  (5);
\draw [line width=1.25, color=black] (6) to  (5);
\draw [line width=1.25, color=black] (2) to  (6);
\draw [line width=1.25, color=black] (2) to  (4);
\node at (100.0pt, -89.375pt) {\textcolor{black}{$g_j$}};
\node at (100.0pt, -10.625pt) {\textcolor{black}{$g_i$}};
\node at (50.0pt, -10.625pt) {\textcolor{black}{$g_l$}};
\node at (50.0pt, -89.375pt) {\textcolor{black}{$g_k$}};
\node at (150.0pt, -10.625pt) {\textcolor{black}{$g_p$}};
\node at (150.0pt, -89.375pt) {\textcolor{black}{$g_m$}};
\end{tikzpicture}
\caption{Two squares with one common edge in $Z(\alpha,\beta)$}\label{f-C4-C4}
\end{figure}
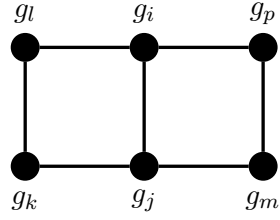

\begin{thm}\label{C4-C4}
Suppose that $Z(\alpha,\beta)$ contains two squares with exactly one edge in common. Then exactly one of the relations $14$, $22$ or $26$ of Table \ref{tab-C4} is satisfied in $G$.
\end{thm}

So, there are $T_C\in \mathcal{T}(C)$ and $T_{C'}\in \mathcal{T}(C')$ such that $T_C=[a_1,b_1,a_2,b_2,a_3,b_3,a_4,b_4]$ and $T_{C'}=[a_1,b_1,a_2,b_2,a_3',b_3',a_4',b_4']$.

\begin{proof}
Suppose that the graph $Z(\alpha,\beta)$ contains two squares $C$ and $C'$ with exactly one common edge as Figure \ref{f-C4-C4}, for some distinct elements $g_i,g_j,g_k,g_l,g_m,g_p\in supp(\beta)$. So, there are $T_C\in \mathcal{T}(C)$ and $T_{C'}\in \mathcal{T}(C')$ such that  $T_C=[a_1,b_1,a_2,b_2,a_3,b_3,a_4,b_4]$ and $T_{C'}=[a_1,b_1,a_2',b_2',a_3',b_3',a_4',b_4']$, and $R(T_C)$ and $R(T_{C'})$ are as follows:
\begin{equation}\label{e-8}
R(T_C)= \left\{
\begin{array}{l}
a_1g_i=b_1g_j\\
a_2g_j=b_2g_k\\
a_3g_k=b_3g_l\\
a_4g_l=b_4g_i\\
\end{array} \right.
\qquad \qquad
R(T_{C'})= \left\{
\begin{array}{l}
a_1g_i=b_1g_j\\
a_2'g_j=b_2'g_m\\
a_3'g_m=b_3'g_p\\
a_4'g_p=b_4'g_i\\
\end{array} \right.
\end{equation}
By Remark \ref{r-Z-K3-K3}, we have
\begin{align}\label{e-new1}
a_1\neq b_1\neq \cdots \neq a_4\neq b_4\neq a_1 \text{ and } b_1\neq  a_2'\neq \cdots \neq a_4' \neq b_4'\neq a_1.
\end{align}
Now we prove that $a_2 \neq a_2'$ and $b_4 \neq b_4'$. Suppose, for a contradiction, that  $a_2=a_2'$.  So by \ref{e-8},  $b_2g_k=a_2'g_j=b_2'g_m$. Since $\alpha\beta=0$ in $\mathbb{F}_2[G]$, there are $g_a \in supp(\beta)\setminus \{g_j,g_k,g_m\}$ and $h_a \in supp(\alpha) \setminus \{a_2',b_2,b_2'\}$ such that $b_2g_k=a_2'g_j=b_2'g_m=h_ag_a$, a contradiction because $|supp(\alpha)|=3$ and $ \{a_2',b_2,b_2'\}=supp(\alpha)$. Hence, 
\begin{align}\label{e-9}
a_2 \neq a_2'.
\end{align}
Also with the same discussion such as above, 
\begin{align}\label{e-10}
b_4 \neq b_4'.
\end{align} 
By Lemma \ref{Z-C4,C4}, the cycles $C$ and $C'$ are equivalent. So, $T_{C'}$ must be in $\mathcal{T}(C)$. In the following, we explain each cases in details.
\begin{enumerate}
\item
Let $T_{C'}=[a_1,b_1,a_2,b_2,a_3,b_3,a_4,b_4]$: So $a_2= a_2'$, that is a contradiction with \ref{e-9}.
\item
Let $T_{C'}=[a_4,b_4,a_1,b_1,a_2,b_2,a_3,b_3]$: So, we have 
\begin{equation*}
T_{C}=[a_1,b_1,a_2,b_2,a_3,b_3,a_1,b_1] \text{ and } T_{C'}=[a_1,b_1,a_1,b_1,a_2, b_2,a_3,b_3].
\end{equation*}
By \ref{e-new1}, \ref{e-9} and \ref{e-10}, $a_2\neq a_1$, $a_2\neq b_1$, $b_3\neq a_1$, $b_3\neq b_1$ and $a_1\neq b_1$. So, $a_2=b_3$ because $b_3 \in supp(\alpha)=\{a_1,b_1,a_2\}$. So,  $T_{C}=[a_1,b_1,a_2,b_2,a_3,a_2,a_1,b_1]$ and  $T_{C'}=[a_1,b_1,a_1,b_1,a_2,b_2,a_3,a_2]$. By \ref{e-new1},  $b_2\neq a_2$, $a_3\neq a_2$ and $b_2\neq a_3$. Also,  $supp(\alpha)=\{a_1,b_1,a_2\}$. So, there are exactly two cases for $b_2,a_3\in supp(\alpha)$. In the following, we show that each of such cases gives a  contradiction.
\begin{enumerate}
\item[i)]
Let $b_2=a_1$ and $a_3=b_1$: So, $T_C=[a_1,b_1,a_2,a_1,b_1,a_2,a_1,b_1]$ and the relation of $C$ is $a_1^{-1}b_1a_2^{-1}a_1b_1^{-1}a_2a_1^{-1}b_1=1$. Since $1\in supp(\alpha)=\{a_1,b_1,a_2\}$, the possible cases are as follows:
\begin{enumerate}
\item[a)]
Let $a_1=1$: $a_2^{-1}b_1a_2=b_1^{\ 2}$ and so $G$ is a quotient of $BS(1,2)$, a contradiction.
\item[b)]
Let $b_1=1$: $a_2^{-1}a_1a_2=a_1^{\ 2}$ and so $G$ is a quotient of $BS(1,2)$, a contradiction.
\item[c)]
Let $a_2=1$: $a_1^{-1}b_1a_1b_1^{-1}a_1^{-1}b_1=1$. If $x=a_1^{-1}b_1$ and $y=b_1^{-1}$, then $y^{-1}xy=x^2$ and $G$ is a quotient of $BS(1,2)$, a contradiction.
\end{enumerate}
\end{enumerate}
\begin{enumerate}
\item[ii)]
Let $b_2=b_1$ and $a_3=a_1$: So, $T_C=[a_1,b_1,a_2,b_1,a_1,a_2,a_1,b_1]$ and the relation of $C$ is $a_1^{-1}b_1a_2^{-1}b_1a_1^{-1}a_2a_1^{-1}b_1=1$. With the same discussion as item (i), the possible cases are as follows:
\begin{enumerate}
\item[a)]
Let $a_1=1$: $a_2^{-1}b_1a_2=b_1^{-2}$ and so $G$ is a quotient of $BS(1,-2)$, a contradiction.
\item[b)]
Let $b_1=1$: $a_2^{-1}a_1a_2=a_1^{-2}$ and so $G$ is a quotient of $BS(1,-2)$, a contradiction.
\item[c)]
Let $a_2=1$: $a_1^{-1}b_1^{\ 2}a_1^{-2}b_1=1$. If $x=b_1^{-1}a_1$ and $y=b_1^{-1}$, then $y^{-1}xy=x^{-2}$ and $G$ is a quotient of $BS(1,-2)$, a contradiction.
\end{enumerate}
\end{enumerate}
Therefore, $T_{C'}\neq [a_4,b_4,a_1,b_1,a_2,b_2,a_3,b_3]$.
\item
Let $T_{C'}=[a_3,b_3,a_4,b_4,a_1,b_1,a_2,b_2]$: So, we have
\begin{align*}
T_{C}=[a_1,b_1,a_2,b_2,a_1,b_1,a_4,b_4] \text{ and } T_{C'}=[a_1,b_1,a_4,b_4,a_1,b_1,a_2,b_2].
\end{align*}
By \ref{e-new1}, $a_2 \neq b_1$. In the following, we show that $a_2\neq a_1$. 

Suppose, for a contradiction, that $a_2=a_1$. So, $T_{C}=[a_1,b_1,a_1, b_2,a_1,b_1,a_4,b_4]$ and  $T_{C'}=[a_1,b_1,a_4,b_4,a_1,b_1,a_1,b_2]$. By \ref{e-new1}, \ref{e-9} and \ref{e-10},  $b_1\neq a_1$, $b_1\neq a_4$, $b_4\neq a_1$, $b_4\neq a_4$ and $a_1\neq a_4$. So, $b_1=b_4$ because $b_4 \in supp(\alpha)=\{a_1,b_1,a_4\}$. Now by \ref{e-new1}, \ref{e-9} and \ref{e-10}, $b_2\neq a_1$, $b_2\neq b_1$, $a_4\neq a_1$, $a_4\neq b_1$ and $a_1\neq b_1$. Hence, $b_2=a_4$ because $b_2\in supp(\alpha)=\{a_1,b_1,a_4\}$. So, $T_C=[a_1,b_1,a_1,b_2,a_1,b_1,b_2,b_1]$ and the relation of $C$ is $a_1^{-1}b_1a_1^{-1}b_2a_1^{-1}b_1b_2^{-1}b_1=1$. Since $1\in supp(\alpha)=\{a_1,b_1,b_2\}$,  the possible cases are as follows:
\begin{enumerate}
\item[a)]
Let $a_1=1$: $b_2^{-1}b_1^{-2}b_2=b_1$ and so $G$ is a quotient of $BS(-2,1)$, a contradiction.
\item[b)]
Let $b_1=1$: $b_2^{-1}a_1^{-2}b_2=a_1$ and so $G$ is a quotient of $BS(-2,1)$, a contradiction.
\item[c)]
Let $b_2=1$: $a_1^{-1}b_1a_1^{-2}b_1^{\ 2}=1$. If $x=b_1a_1^{-1}$ and $y=a_1^{-1}$, then $y^{-1}x^{-2}y=x$ and $G$ is a quotient of $BS(-2,1)$, a contradiction.
\end{enumerate}
Therefore, $a_2\neq a_1$. 

By \ref{e-new1}, \ref{e-9} and \ref{e-10}, $a_1\neq b_1$, $a_4\neq a_2$, $a_4\neq b_1$ and $b_1\neq a_2$. So, $a_1=a_4$ because $a_1\in supp(\alpha)=\{b_1,a_2,a_4\}$. Now by \ref{e-new1}, \ref{e-9} and \ref{e-10}, $a_2\neq a_1$, $a_2\neq b_2$, $b_4\neq a_1$, $b_4\neq b_2$ and $a_1\neq b_2$. So, $a_2=b_4$ because $b_4\in supp(\alpha)=\{a_1,a_2,b_2\}$. Also, $b_2\neq a_1$, $b_2\neq a_2$, $b_1\neq a_1$, $b_1\neq a_2$ and $a_1\neq a_2$. Therefore, $b_1=b_2$ because $b_1 \in supp(\alpha)=\{a_1,a_2,b_2\}$. Hence, $T_C=[a_1,b_1,a_2,b_1,a_1,b_1,a_1,a_2]$ and the relation of $C$ is $a_1^{-1}b_1a_2^{-1}b_1a_1^{-1}b_1a_1^{-1}a_2=1$. Since $1\in supp(\alpha)=\{a_1,b_1,a_2\}$,  the possible cases are as follows:
\begin{enumerate}
\item[a)]
Let $a_1=1$: $a_2^{-1}b_1^{-2}a_2=b_1$ and so $G$ is a quotient of $BS(-2,1)$, a contradiction.
\item[b)]
Let $b_1=1$: $a_2^{-1}a_1^{-2}a_2=a_1$ and so $G$ is a quotient of $BS(-2,1)$, a contradiction.
\item[c)]
Let $a_2=1$: $a_1^{-1}b_1^{\ 2}a_1^{-1}b_1a_1^{-1}=1$. If $x=a_1b_1^{-1}$ and $y=a_1^{-1}$, then $y^{-1}x^{-2}y=x$ and $G$ is a quotient of $BS(-2,1)$, a contradiction.
\end{enumerate}
Therefore, $T_{C'}\neq [a_3,b_3,a_4,b_4,a_1,b_1,a_2,b_2]$.
\item
Let $T_{C'}=[a_2,b_2,a_3,b_3,a_4,b_4,a_1,b_1]$: So, we have
\begin{align*}
T_{C}=[a_1,b_1,a_1,b_1,a_3,b_3,a_4,b_4] \text{ and } T_{C'}=[a_1,b_1,a_3,b_3,a_4,b_4,a_1,b_1]. 
\end{align*}
By \ref{e-new1}, \ref{e-9} and \ref{e-10}, we have $a_3\neq a_1$, $a_3\neq b_1$, $b_4\neq a_1$, $b_4\neq b_1$ and $a_1\neq b_1$. Therefore, $a_3=b_4$ because $b_4 \in supp(\alpha)=\{a_1,b_1,a_3\}$. By \ref{e-new1}, $b_3\neq a_3$, $a_4\neq a_3$ and $b_3\neq a_4$. Also,  $supp(\alpha)=\{a_1,b_1,a_3\}$. So, there are exactly two cases for  $b_3,a_4\in supp(\alpha)$. In the following, we show that each of such cases gives a  contradiction.
\begin{enumerate}
\item[i)]
Let $b_3=a_1$ and $a_4=b_1$: So, $T_C=[a_1,b_1,a_1,b_1,a_3,a_1,b_1,a_3]$ and the relation of $C$ is $a_1^{-1}b_1a_1^{-1}b_1a_3^{-1}a_1b_1^{-1}a_3=1$. Since $1\in supp(\alpha)=\{a_1,b_1,a_3\}$, the possible cases are as follows:
\begin{enumerate}
\item[a)]
Let $a_1=1$: $a_3^{-1}b_1a_3=b_1^{\ 2}$ and so $G$ is a quotient of $BS(1,2)$, a contradiction.
\item[b)]
Let $b_1=1$: $a_3^{-1}a_1a_3=a_1^{\ 2}$ and so $G$ is a quotient of $BS(1,2)$, a contradiction.
\item[c)]
Let $a_3=1$: $a_1^{-1}b_1a_1b_1^{-1}a_1^{-1}b_1=1$. If $x=a_1^{-1}b_1$ and $y=b_1^{-1}$, then $y^{-1}xy=x^2$ and $G$ is a quotient of $BS(1,2)$, a contradiction.
\end{enumerate}
\end{enumerate}
\begin{enumerate}
\item[ii)]
Let $b_3=b_1$ and $a_4=a_1$: So, $T_C=[a_1,b_1,a_1,b_1,a_3,b_1,a_1,a_3]$ and the relation of $C$ is $a_1^{-1}b_1a_1^{-1}b_1a_3^{-1}b_1a_1^{-1}a_3=1$. With the same discussion as item (i), the possible cases are as follows:
\begin{enumerate}
\item[a)]
Let $a_1=1$: $a_3^{-1}b_1a_3=b_1^{-2}$ and so $G$ is a quotient of $BS(1,-2)$, a contradiction.
\item[b)]
Let $b_1=1$: $a_3^{-1}a_1a_3=a_1^{-2}$ and so $G$ is a quotient of $BS(1,-2)$, a contradiction.
\item[c)]
Let $a_3=1$: $a_1^{-1}b_1^{\ 2}a_1^{-2}b_1=1$. If $x=b_1^{-1}a_1$ and $y=b_1^{-1}$, then $y^{-1}xy=x^{-2}$ and $G$ is a quotient of $BS(1,-2)$, a contradiction.
\end{enumerate}
\end{enumerate}
Therefore, $T_{C'}\neq [a_2,b_2,a_3,b_3,a_4,b_4,a_1,b_1]$.
\item
Let $T_{C'}=[b_1,a_1,b_4,a_4,b_3,a_3,b_2,a_2]$: So $a_1 =b_1$, that is a contradiction with \ref{e-new1}.
\item
Let $T_{C'}=[b_4,a_4,b_3,a_3,b_2,a_2,b_1,a_1]$: So $a_1=b_4$, that is a contradiction with \ref{e-new1}.
\item
Let $T_{C'}=[b_2,a_2,b_1,a_1,b_4,a_4,b_3,a_3]$: So $b_1=a_2$, that is a contradiction with \ref{e-new1}.
\item
Let $T_{C'}=[b_3,a_3,b_2,a_2,b_1,a_1,b_4,a_4]$: So, we have
\begin{align*}
T_{C}=[a_1,b_1,a_2,b_2,b_1,a_1,a_4,b_4] \text{ and } T_{C'}=[a_1,b_1,b_2,a_2,b_1,a_1,b_4,a_4].
\end{align*}
We have $a_2=a_1$ or $a_2\neq a_1$. In the following, we explain each cases in details.
\begin{enumerate}
\item[A)] Let $a_2=a_1$: So, $T_{C}=[a_1,b_1,a_1,b_2,b_1,a_1, a_4,b_4]$ and  $T_{C'}=[a_1,b_1,b_2,a_1,b_1,a_1,b_4,a_4]$. By \ref{e-new1}, $b_2\neq a_1$, $b_2\neq b_1$ and $a_1\neq b_1$. So, $\{a_1,b_1,b_2\}=supp(\alpha)$. Also, $a_4\neq a_1$, $b_4\neq a_1$ and $a_4\neq b_4$. Therefore, there are exactly two cases for $a_4,b_4\in supp(\alpha)$. In the following, we show that $a_4=b_1$ and $b_4=b_2$.

Suppose, for a contradiction, that $a_4=b_2$ and $b_4=b_1$. So, $T_C=[a_1,b_1,a_1,b_2,b_1,a_1,b_2,b_1]$ and the relation of $C$ is $a_1^{-1}b_1a_1^{-1}b_2b_1^{-1}a_1b_2^{-1}b_1=1$. Since $1\in supp(\alpha)=\{a_1,b_1,b_2\}$,  the possible cases are as follows:
\begin{enumerate}
\item[a)]
Let $a_1=1$: $b_2^{-1}b_1^{\ 2}b_2=b_1$ and so $G$ is a quotient of $BS(2,1)$, a contradiction.
\item[b)]
Let $b_1=1$: $b_2^{-1}a_1^{\ 2}b_2=a_1$ and so $G$ is a quotient of $BS(2,1)$, a contradiction.
\item[c)]
Let $b_2=1$: $a_1^{-1}b_1a_1^{-1}b_1^{-1}a_1b_1=1$. If $x=b_1a_1^{-1}$ and $y=a_1^{-1}$, then $y^{-1}x^2y=x$ and $G$ is a quotient of $BS(2,1)$, a contradiction.
\end{enumerate}
Therefore, $a_4=b_1$ and $b_4=b_2$. So, $T_C=[a_1,b_1,a_1,b_2,b_1,a_1,b_1,b_2]$ and the relation of $C$ is $a_1^{-1}b_1a_1^{-1}b_2b_1^{-1}a_1b_1^{-1}b_2=1$. Since $1\in supp(\alpha)=\{a_1,b_1,b_2\}$, the possible cases are as follows:
\begin{enumerate}
\item[a)]
Let $a_1=1$: $b_1b_2b_1^{-2}b_2=1$, where $\{b_1,b_2\}=\{h_2,h_3\}$.
\item[b)]
Let $b_1=1$: $a_1b_2a_1^{-2}b_2=1$, where $\{a_1,b_2\}=\{h_2,h_3\}$.
\item[c)]
Let $b_2=1$: $b_1 a_1^{-1} b_1 a_1 b_1^{-1} a_1=1$, where $\{a_1,b_1\}=\{h_2,h_3\}$.
\end{enumerate}
So, $a_2=a_1$ implies that $T_{C}=[a_1,b_1,a_1,b_2,b_1, a_1,b_1,b_2]$, $T_{C'}=[a_1,b_1,b_2,a_1,b_1,a_1,b_2,b_1]$ and $\{a_1,b_1, b_2\}=supp(\alpha)$. Also, exactly one of the relations $14$, $22$ or $26$ of Table \ref{tab-C4} is satisfied in $G$.

\item[B)] Let $a_2\neq a_1$: By \ref{e-new1}, $a_2\neq a_1$, $a_2\neq b_1$ and $a_1\neq b_1$. So, $\{a_1,b_1,a_2\}=supp(\alpha)$. Also, $b_2\neq b_1$ and $b_2\neq a_2$. So, $b_2= a_1$,  $T_{C}=[a_1,b_1,a_2,a_1,b_1,a_1,a_4,b_4]$ and  $T_{C'}=[a_1,b_1,a_1,a_2,b_1,a_1,b_4,a_4]$. Also by \ref{e-new1}, $a_4\neq a_1$, $b_4\neq a_1$ and $a_4\neq b_4$. Therefore, there are exactly two cases for $a_4,b_4\in supp(\alpha)$. In the following, we show that $a_4=a_2$ and $b_4=b_1$.

Suppose, for a contradiction, that $a_4=b_1$ and $b_4=a_2$. So,  $T_C=[a_1,b_1,a_2,a_1,b_1,a_1,b_1,a_2]$ and the relation of $C$ is $a_1^{-1}b_1a_2^{-1}a_1b_1^{-1}a_1b_1^{-1}a_2=1$. Since $1\in supp(\alpha)=\{a_1,b_1,a_2\}$, the possible cases are as follows:
\begin{enumerate}
\item[a)]
Let $a_1=1$: $a_2^{-1}b_1^{\ 2}a_2=b_1$ and so $G$ is a quotient of $BS(2,1)$, a contradiction.
\item[b)]
Let $b_1=1$: $a_2^{-1}a_1^{\ 2}a_2=a_1$ and so $G$ is a quotient of $BS(2,1)$, a contradiction.
\item[c)]
Let $a_2=1$: $a_1^{-1}b_1a_1b_1^{-1}a_1b_1^{-1}=1$. If $x=a_1b_1^{-1}$ and $y=b_1^{-1}$, then $y^{-1}x^2y=x$ and $G$ is a quotient of $BS(2,1)$, a contradiction.
\end{enumerate}
Therefore, $a_4=a_2$ and $b_4=b_1$. So, $T_C=[a_1,b_1,a_2,a_1,b_1,a_1,a_2,b_1]$ and the relation of $C$ is $a_1^{-1}b_1a_2^{-1}a_1b_1^{-1}a_1a_2^{-1}b_1=1$. Since $1\in supp(\alpha)=\{a_1,b_1,a_2\}$, the possible cases are as follows:
\begin{enumerate}
\item[a)]
Let $a_1=1$: $b_1a_2b_1^{-2}a_2=1$, where $\{b_1,a_2\}=\{h_2,h_3\}$.
\item[b)]
Let $b_1=1$: $a_1a_2a_1^{-2}a_2=1$, where $\{a_1,a_2\}=\{h_2,h_3\}$.
\item[c)]
Let $a_2=1$: $b_1 a_1^{-1} b_1 a_1 b_1^{-1} a_1=1$, where $\{a_1,b_1\}=\{h_2,h_3\}$.
\end{enumerate}
So, $a_2\neq a_1$ implies that $T_{C}=[a_1,b_1,a_2,a_1,b_1, a_1,a_2,b_1]$,  $T_{C'}=[a_1,b_1,a_1,a_2,b_1,a_1,b_1,a_2]$ and $\{a_1,b_1, a_2\}=supp(\alpha)$. Also, exactly one of the relations $14$, $22$ or $26$ of Table \ref{tab-C4} is satisfied in $G$.
\end{enumerate}
\end{enumerate}
This completes the proof.
\end{proof}

\begin{rem}\label{r-1}
{\rm Suppose that $Z(\alpha, \beta)$ contains two squares $C$ and $C'$ with exactly one common edge as Figure \ref{f-C4-C4}, for some distinct elements $g_i,g_j,g_k,g_l,g_m,g_p\in supp(\beta)$. Let $T_C$ and $T_{C'}$ be $8$-tuples of $C$ and $C'$, respectively, with $R(T_C)$ and $R(T_{C'})$ as follows:
\begin{equation*}
R(T_C)= \left\{
\begin{array}{l}
a_1g_i=b_1g_j\\
a_2g_j=b_2g_k\\
a_3g_k=b_3g_l\\
a_4g_l=b_4g_i\\
\end{array} \right.
\qquad \qquad
R(T_{C'})= \left\{
\begin{array}{l}
a_1g_i=b_1g_j\\
a_2'g_j=b_2'g_m\\
a_3'g_m=b_3'g_p\\
a_4'g_p=b_4'g_i\\
\end{array} \right.
\end{equation*}
where $a_1,b_1,a_2,b_2,a_3,b_3,a_4,b_4,a_2',b_2',a_3',b_3',a_4',b_4'\in supp(\alpha)$.\\
By the proof of Theorem \ref{C4-C4}, the possible cases for $\{T_C,T_{C'}\}$ and the relation of $C$ and $C'$ are:
\begin{enumerate}
\item[1)]
$\{T_C,T_{C'}\}=\{[h_2,1,h_2,h_3,1,h_2,1,h_3],[h_2,1,h_3,h_2,1,h_2,h_3,1]\}$ and the relation of $C$ and $C'$ is relation $14$ of Table \ref{tab-C4}.
\item[2)]
$\{T_C,T_{C'}\}=\{[1,h_2,1,h_3,h_2,1,h_2,h_3],[1,h_2,h_3,1,h_2,1,h_3,h_2]\}$ and the relation of $C$ and $C'$ is relation $14$ of Table \ref{tab-C4}.
\item[3)]
$\{T_C,T_{C'}\}=\{[h_3,1,h_3,h_2,1,h_3,1,h_2],[h_3,1,h_2,h_3,1,h_3,h_2,1]\}$ and the relation of $C$ and $C'$ is relation $22$ of Table \ref{tab-C4}.
\item[4)]
$\{T_C,T_{C'}\}=\{[1,h_3,1,h_2,h_3,1,h_3,h_2],[1,h_3,h_2,1,h_3,1,h_2,h_3]\}$ and the relation of $C$ and $C'$ is relation $22$ of Table \ref{tab-C4}.
\item[5)]
$\{T_C,T_{C'}\}=\{[h_2,h_3,h_2,1,h_3,h_2,h_3,1],[h_2,h_3,1,h_2,h_3,h_2,1,h_3]\}$ and the relation of $C$ and $C'$ is relation $26$ of Table \ref{tab-C4}.
\item[6)]
$\{T_C,T_{C'}\}=\{[h_3,h_2,h_3,1,h_2,h_3,h_2,1],[h_3,h_2,1,h_3,h_2,h_3,1,h_2]\}$ and the relation of $C$ and $C'$ is relation $26$ of Table \ref{tab-C4}.
\end{enumerate}
}
\end{rem}

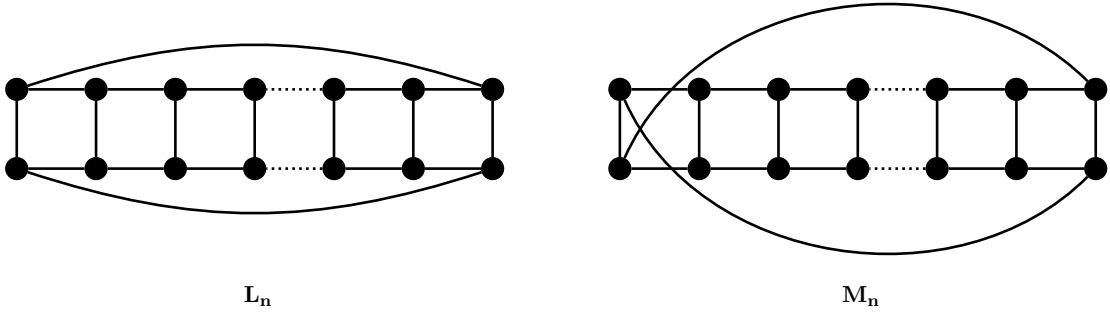
\begin{figure}[ht]
\centering
\psscalebox{0.8 0.8} {\begin{tikzpicture}
\node (14) [circle, minimum size=3pt, fill=black, line width=0.625pt, draw=black] at (75.0pt, -62.5pt)  {};
\node (12) [circle, minimum size=3pt, fill=black, line width=0.625pt, draw=black] at (75.0pt, -100.0pt)  {};
\node (15) [circle, minimum size=3pt, fill=black, line width=0.625pt, draw=black] at (112.5pt, -62.5pt)  {};
\node (16) [circle, minimum size=3pt, fill=black, line width=0.625pt, draw=black] at (112.5pt, -100.0pt)  {};
\node (10) [circle, minimum size=3pt, fill=black, line width=0.625pt, draw=black] at (150.0pt, -100.0pt)  {};
\node (9) [circle, minimum size=3pt, fill=black, line width=0.625pt, draw=black] at (150.0pt, -62.5pt)  {};
\node (19) [circle, minimum size=3pt, fill=black, line width=0.625pt, draw=black] at (187.5pt, -100.0pt)  {};
\node (20) [circle, minimum size=3pt, fill=black, line width=0.625pt, draw=black] at (187.5pt, -62.5pt)  {};
\node (22) [circle, minimum size=3pt, fill=black, line width=0.625pt, draw=black] at (225.0pt, -62.5pt)  {};
\node (21) [circle, minimum size=3pt, fill=black, line width=0.625pt, draw=black] at (225.0pt, -100.0pt)  {};
\node (2) [circle, minimum size=3pt, fill=black, line width=0.625pt, draw=black] at (262.5pt, -100.0pt)  {};
\node (4) [circle, minimum size=3pt, fill=black, line width=0.625pt, draw=black] at (262.5pt, -62.5pt)  {};
\node (6) [circle, minimum size=3pt, fill=black, line width=0.625pt, draw=black] at (300.0pt, -100.0pt)  {};
\node (5) [circle, minimum size=3pt, fill=black, line width=0.625pt, draw=black] at (300.0pt, -62.5pt)  {};
\draw [line width=1.25, color=black] (14) to  (15);
\draw [line width=1.25, color=black] (16) to  (15);
\draw [line width=1.25, color=black] (12) to  (16);
\draw [line width=1.25, color=black] (12) to  (14);
\draw [line width=1.25, color=black] (2) to  (6);
\draw [line width=1.25, color=black] (5) to  (6);
\draw [line width=1.25, color=black] (16) to  (10);
\draw [line width=1.25, color=black] (15) to  (9);
\draw [line width=1.25, color=black] (21) to  (22);
\draw [line width=1.25, dotted, color=black] (21) to  (19);
\draw [line width=1.25, dotted, color=black] (22) to  (20);
\draw [line width=1.25, color=black] (20) to  (19);
\draw [line width=1.25, color=black] (10) to  (9);
\draw [line width=1.25, color=black] (10) to  (19);
\draw [line width=1.25, color=black] (20) to  (9);
\draw [line width=1.25, color=black] (4) to  (2);
\draw [line width=1.25, color=black] (5) to  (4);
\draw [line width=1.25, color=black] (2) to  (21);
\draw [line width=1.25, color=black] (4) to  (22);
\draw [line width=1.25, color=black] (14) to  [in=162, out=18] (5);
\draw [line width=1.25, color=black] (6) to  [in=342, out=198] (12);
%\node at (66.875pt, -73.75pt) {\textcolor{black}{$g_i$}};
%\node at (65.625pt, -88.125pt) {\textcolor{black}{$g_{i'}$}};
%\node at (104.375pt, -73.75pt) {\textcolor{black}{$g_j$}};
%\node at (103.125pt, -88.125pt) {\textcolor{black}{$g_{j'}$}};
%\node at (139.375pt, -88.125pt) {\textcolor{black}{$g_{k'}$}};
%\node at (140.625pt, -73.75pt) {\textcolor{black}{$g_k$}};
%\node at (178.125pt, -88.125pt) {\textcolor{black}{$g_{l'}$}};
%\node at (179.375pt, -73.75pt) {\textcolor{black}{$g_l$}};
%\node at (214.375pt, -73.75pt) {\textcolor{black}{$g_y$}};
%\node at (213.125pt, -88.125pt) {\textcolor{black}{$g_{y'}$}};
%\node at (251.875pt, -88.125pt) {\textcolor{black}{$g_{x'}$}};
%\node at (253.125pt, -73.75pt) {\textcolor{black}{$g_x$}};
%\node at (289.375pt, -88.125pt) {\textcolor{black}{$g_{z'}$}};
%\node at (290.625pt, -73.75pt) {\textcolor{black}{$g_z$}};
\node at (189.125pt, -160.375pt) {\textcolor{black}{$\mathbf{L_n}$}};
\end{tikzpicture}}
\hspace{2.5cc}
\psscalebox{0.8 0.8} {\begin{tikzpicture}
\node (14) [circle, minimum size=3pt, fill=black, line width=0.625pt, draw=black] at (75.0pt, -62.5pt)  {};
\node (12) [circle, minimum size=3pt, fill=black, line width=0.625pt, draw=black] at (75.0pt, -100.0pt)  {};
\node (15) [circle, minimum size=3pt, fill=black, line width=0.625pt, draw=black] at (112.5pt, -62.5pt)  {};
\node (16) [circle, minimum size=3pt, fill=black, line width=0.625pt, draw=black] at (112.5pt, -100.0pt)  {};
\node (10) [circle, minimum size=3pt, fill=black, line width=0.625pt, draw=black] at (150.0pt, -100.0pt)  {};
\node (9) [circle, minimum size=3pt, fill=black, line width=0.625pt, draw=black] at (150.0pt, -62.5pt)  {};
\node (19) [circle, minimum size=3pt, fill=black, line width=0.625pt, draw=black] at (187.5pt, -100.0pt)  {};
\node (20) [circle, minimum size=3pt, fill=black, line width=0.625pt, draw=black] at (187.5pt, -62.5pt)  {};
\node (22) [circle, minimum size=3pt, fill=black, line width=0.625pt, draw=black] at (225.0pt, -62.5pt)  {};
\node (21) [circle, minimum size=3pt, fill=black, line width=0.625pt, draw=black] at (225.0pt, -100.0pt)  {};
\node (2) [circle, minimum size=3pt, fill=black, line width=0.625pt, draw=black] at (262.5pt, -100.0pt)  {};
\node (4) [circle, minimum size=3pt, fill=black, line width=0.625pt, draw=black] at (262.5pt, -62.5pt)  {};
\node (6) [circle, minimum size=3pt, fill=black, line width=0.625pt, draw=black] at (300.0pt, -100.0pt)  {};
\node (5) [circle, minimum size=3pt, fill=black, line width=0.625pt, draw=black] at (300.0pt, -62.5pt)  {};
\draw [line width=1.25, color=black] (14) to  (15);
\draw [line width=1.25, color=black] (16) to  (15);
\draw [line width=1.25, color=black] (12) to  (16);
\draw [line width=1.25, color=black] (12) to  (14);
\draw [line width=1.25, color=black] (2) to  (6);
\draw [line width=1.25, color=black] (5) to  (6);
\draw [line width=1.25, color=black] (16) to  (10);
\draw [line width=1.25, color=black] (15) to  (9);
\draw [line width=1.25, color=black] (21) to  (22);
\draw [line width=1.25, dotted, color=black] (21) to  (19);
\draw [line width=1.25, dotted, color=black] (22) to  (20);
\draw [line width=1.25, color=black] (20) to  (19);
\draw [line width=1.25, color=black] (10) to  (9);
\draw [line width=1.25, color=black] (10) to  (19);
\draw [line width=1.25, color=black] (20) to  (9);
\draw [line width=1.25, color=black] (4) to  (2);
\draw [line width=1.25, color=black] (5) to  (4);
\draw [line width=1.25, color=black] (2) to  (21);
\draw [line width=1.25, color=black] (4) to  (22);
\draw [line width=1.25, color=black] (5) to  [in=66, out=135] (12);
\draw [line width=1.25, color=black] (6) to  [in=294, out=225] (14);
%\node at (66.875pt, -73.75pt) {\textcolor{black}{$g_i$}};
%\node at (65.625pt, -88.125pt) {\textcolor{black}{$g_{i'}$}};
%\node at (104.375pt, -73.75pt) {\textcolor{black}{$g_j$}};
%\node at (103.125pt, -88.125pt) {\textcolor{black}{$g_{j'}$}};
%\node at (139.375pt, -88.125pt) {\textcolor{black}{$g_{k'}$}};
%\node at (140.625pt, -73.75pt) {\textcolor{black}{$g_k$}};
%\node at (178.125pt, -88.125pt) {\textcolor{black}{$g_{l'}$}};
%\node at (179.375pt, -73.75pt) {\textcolor{black}{$g_l$}};
%\node at (214.375pt, -73.75pt) {\textcolor{black}{$g_y$}};
%\node at (213.125pt, -88.125pt) {\textcolor{black}{$g_{y'}$}};
%\node at (251.875pt, -88.125pt) {\textcolor{black}{$g_{x'}$}};
%\node at (253.125pt, -73.75pt) {\textcolor{black}{$g_x$}};
%\node at (289.375pt, -88.125pt) {\textcolor{black}{$g_{z'}$}};
%\node at (290.625pt, -73.75pt) {\textcolor{black}{$g_z$}};
\node at (189.125pt, -160.375pt) {\textcolor{black}{$\mathbf{M_n}$}};
\end{tikzpicture}}
\caption{Two graphs which are not isomorphic to $Z(\alpha,\beta)$}\label{f-3}
\end{figure}

\begin{thm}\label{L-M}
The zero-divisor graph $Z(\alpha,\beta)$ is isomorphic to none of the graphs $L_n$ and $M_n$ in Figure \ref{f-3}.
\end{thm}
\begin{proof}
Let $n$ be the number of vertices of the graph $L_n$ or $M_n$. Then by Figure \ref{f-3}, the number of cycles of length $4$ in $L_n$ or $M_n$ is equal to $n/2$. Also, each two consecutive squares have a common edge. So, if $Z(\alpha,\beta)$ contains no two squares with exactly one common edge, then the latter graph is isomorphic to none of the graphs $L_n$ and $M_n$. 
\begin{figure}[h]
\psscalebox{0.8 0.8} {\begin{tikzpicture}
\node (14) [circle, minimum size=3pt, fill=black, line width=0.625pt, draw=black] at (75.0pt, -62.5pt)  {};
\node (12) [circle, minimum size=3pt, fill=black, line width=0.625pt, draw=black] at (75.0pt, -100.0pt)  {};
\node (15) [circle, minimum size=3pt, fill=black, line width=0.625pt, draw=black] at (112.5pt, -62.5pt)  {};
\node (16) [circle, minimum size=3pt, fill=black, line width=0.625pt, draw=black] at (112.5pt, -100.0pt)  {};
\node (10) [circle, minimum size=3pt, fill=black, line width=0.625pt, draw=black] at (150.0pt, -100.0pt)  {};
\node (9) [circle, minimum size=3pt, fill=black, line width=0.625pt, draw=black] at (150.0pt, -62.5pt)  {};
\node (19) [circle, minimum size=3pt, fill=black, line width=0.625pt, draw=black] at (187.5pt, -100.0pt)  {};
\node (20) [circle, minimum size=3pt, fill=black, line width=0.625pt, draw=black] at (187.5pt, -62.5pt)  {};
\node (22) [circle, minimum size=3pt, fill=black, line width=0.625pt, draw=black] at (225.0pt, -62.5pt)  {};
\node (21) [circle, minimum size=3pt, fill=black, line width=0.625pt, draw=black] at (225.0pt, -100.0pt)  {};
\node (2) [circle, minimum size=3pt, fill=black, line width=0.625pt, draw=black] at (262.5pt, -100.0pt)  {};
\node (4) [circle, minimum size=3pt, fill=black, line width=0.625pt, draw=black] at (262.5pt, -62.5pt)  {};
\node (6) [circle, minimum size=3pt, fill=black, line width=0.625pt, draw=black] at (300.0pt, -100.0pt)  {};
\node (5) [circle, minimum size=3pt, fill=black, line width=0.625pt, draw=black] at (300.0pt, -62.5pt)  {};
\node (1) [circle, minimum size=3pt, fill=black, line width=0.625pt, draw=black] at (337.5pt, -62.5pt)  {};
\node (3) [circle, minimum size=3pt, fill=black, line width=0.625pt, draw=black] at (337.5pt, -100.0pt)  {};
\draw [line width=1.25, color=black] (14) to  (15);
\draw [line width=1.25, color=black] (16) to  (15);
\draw [line width=1.25, color=black] (12) to  (16);
\draw [line width=1.25, color=black] (12) to  (14);
\draw [line width=1.25, color=black] (2) to  (6);
\draw [line width=1.25, color=black] (5) to  (6);
\draw [line width=1.25, color=black] (16) to  (10);
\draw [line width=1.25, color=black] (15) to  (9);
\draw [line width=1.25, color=black] (21) to  (22);
\draw [line width=1.25, dotted, color=black] (21) to  (19);
\draw [line width=1.25, dotted, color=black] (22) to  (20);
\draw [line width=1.25, color=black] (20) to  (19);
\draw [line width=1.25, color=black] (10) to  (9);
\draw [line width=1.25, color=black] (10) to  (19);
\draw [line width=1.25, color=black] (20) to  (9);
\draw [line width=1.25, color=black] (4) to  (2);
\draw [line width=1.25, color=black] (5) to  (4);
\draw [line width=1.25, color=black] (2) to  (21);
\draw [line width=1.25, color=black] (4) to  (22);
\draw [line width=1.25, color=black] (5) to  (1);
\draw [line width=1.25, color=black] (1) to  (3);
\draw [line width=1.25, color=black] (6) to  (3);
\node at (66.875pt, -73.75pt) {\textcolor{black}{$a_1$}};
\node at (65.625pt, -88.125pt) {\textcolor{black}{$b_1$}};
\node at (104.375pt, -73.75pt) {\textcolor{black}{$a_2$}};
\node at (103.125pt, -88.125pt) {\textcolor{black}{$b_2$}};
\node at (139.375pt, -88.125pt) {\textcolor{black}{$b_3$}};
\node at (140.625pt, -73.75pt) {\textcolor{black}{$a_3$}};
\node at (178.125pt, -88.125pt) {\textcolor{black}{$b_4$}};
\node at (179.375pt, -73.75pt) {\textcolor{black}{$a_4$}};
\node at (211.375pt, -73.75pt) {\textcolor{black}{$a_{m-2}$}};
\node at (211.375pt, -88.125pt) {\textcolor{black}{$b_{m-2}$}};
\node at (249.875pt, -88.125pt) {\textcolor{black}{$b_{m-1}$}};
\node at (249.875pt, -73.75pt) {\textcolor{black}{$a_{m-1}$}};
\node at (289.375pt, -88.125pt) {\textcolor{black}{$b_m$}};
\node at (290.625pt, -73.75pt) {\textcolor{black}{$a_m$}};
\node at (324.875pt, -73.75pt) {\textcolor{black}{$a_{m+1}$}};
\node at (324.875pt, -88.125pt) {\textcolor{black}{$b_{m+1}$}};
\end{tikzpicture}}
\caption{Consecutive cycles of length $4$ in the graph $Z(\alpha,\beta)$}\label{f-4}
\end{figure}
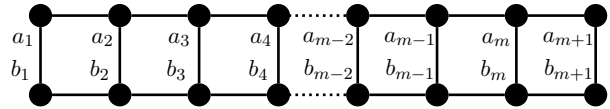

Suppose that $Z(\alpha,\beta)$ contains a subgraph as Figure \ref{f-4}, for some elements $a_1,b_1,\ldots,a_{m+1},b_{m+1}\in supp(\beta)$, in which the number of consecutive $C_4$ cycles is denoted by $m\geq 2$. We denote by $C_1,C_2,\ldots,C_{m}$ the consecutive squares in Figure \ref{f-4}, from the left to the right respectively. For $i\in \{1,2,\ldots,m-1\}$, let $T^i_{C_i}$ and $T^i_{C_{i+1}}$ be $8$-tuples of $C_i$ and $C_{i+1}$, respectively, with $R(T^i_{C_i})$ and $R(T^i_{C_{i+1}})$ as follows:
\begin{equation}\label{e-new2}
R(T^i_{C_i})=\left\{
\begin{array}{l}
c_1a_{i+1}=d_1b_{i+1}\\
c_2b_{i+1}=d_2b_i\\
c_3b_i=d_3a_i\\
c_4a_i=d_4a_{i+1}\\
\end{array} \right.
\qquad \qquad
R(T^i_{C_{i+1}})=\left\{
\begin{array}{l}
c_1a_{i+1}=d_1b_{i+1}\\
c_2'b_{i+1}=d_2'b_{i+2}\\
c_3'b_{i+2}=d_3'a_{i+2}\\
c_4'a_{i+2}=d_4'a_{i+1}\\
\end{array} \right.
\end{equation}
where $c_1,d_1,c_2,d_2,c_3,d_3,c_4,d_4,c_2',d_2',c_3',d_3',c_4',d_4'\in supp(\alpha)$.\\
We claim that for each $i$, $i\in \{1,2,\ldots,m-1\}$, $T^i_{C_i}=T^1_{C_1}$ and $T^i_{C_{i+1}}=T^1_{C_2}$. Note that by Remark \ref{r-1}, there are $6$ possible cases for $\{T^i_{C_i},T^i_{C_{i+1}}\}$, where $i\in \{1,2,\ldots,m-1\}$, and in particular for $\{T^1_{C_1},T^1_{C_2}\}$. For each one of the latter cases, we prove our claim by induction on $m$:

{\bf Case 1:} Let $\{T^1_{C_1},T^1_{C_2}\}=\{A,B\}$, where $A=[h_2,1,h_2,h_3,1,h_2,1,h_3]$ and $B=[h_2,1,h_3,h_2,1,h_2,h_3,1]$. So, there are two subcases as follows:
\begin{enumerate}
\item[a)]
Let  $T^1_{C_1}=A$ and $T^1_{C_2}=B$. If $m=2$, then the statement is obviously true. Suppose that the statement is true for $m-1$. Therefore, $T^{m-2}_{C_{m-2}}=A$ and $T^{m-2}_{C_{m-1}}=B$. So by \ref{e-new2}, $R(T^{m-2}_{C_{m-1}})$ is: 
\begin{equation*}
R(T^{m-2}_{C_{m-1}})=\left\{
\begin{array}{l}
h_2a_{m-1}=b_{m-1}\\
h_3b_{m-1}=h_2b_{m}\\
b_{m}=h_2a_{m}\\
h_3a_{m}=a_{m-1}\\
\end{array} \right.
\end{equation*}
Hence, $T^{m-1}_{C_{m-1}}=A$ and so by Remark \ref{r-1}, $T^{m-1}_{C_{m}}=B$.
\item[b)]
Let  $T^1_{C_1}=B$ and $T^1_{C_2}=A$. If $m=2$, then the statement is obviously true. Suppose that the statement is true for $m-1$. Therefore, $T^{m-2}_{C_{m-2}}=B$ and $T^{m-2}_{C_{m-1}}=A$. So by \ref{e-new2}, $R(T^{m-2}_{C_{m-1}})$ is: 
\begin{equation*}
R(T^{m-2}_{C_{m-1}})=\left\{
\begin{array}{l}
h_2a_{m-1}=b_{m-1}\\
h_2b_{m-1}=h_3b_{m}\\
b_{m}=h_2a_{m}\\
a_{m}=h_3a_{m-1}\\
\end{array} \right.
\end{equation*}
Hence, $T^{m-1}_{C_{m-1}}=B$ and so by Remark \ref{r-1}, $T^{m-1}_{C_{m}}=A$.
\end{enumerate}

{\bf Cases 2-6:} The proof of our claim for the remaining five possible cases of $\{T^1_{C_1},T^1_{C_2}\}$ are similar to the proof for case 1.

Therefore for each $i\in \{1,2,\ldots,m-1\}$, $T^i_{C_i}=T^1_{C_1}$ and $T^i_{C_{i+1}}=T^1_{C_2}$.

Now suppose, for a contradiction, that $Z(\alpha,\beta)$ is isomorphic to the graph $L_n$. So if $m=n/2$, then $Z(\alpha,\beta)$ contains a subgraph as Figure \ref{f-4}, where $\{a_1,b_1,\ldots,a_m, b_m\}=supp(\beta)$, $a_{m+1}=a_1$ and $b_{m+1}=b_1$. We denote by $C_1,C_2,\ldots,C_{m}$ the consecutive squares in Figure \ref{f-4}, from the left to the right respectively. With the above discussion, $T^i_{C_i}=T^1_{C_1}$ and $T^i_{C_{i+1}}=T^1_{C_2}$, for $i\in \{1,2,\ldots,m-1\}$. By Remark \ref{r-1}, there are $6$ possible cases for $\{T^1_{C_1},T^1_{C_2}\}$. Here we show that for each one of the latter cases, there is a contradiction:

{\bf Case 1:} Let $\{T^1_{C_1},T^1_{C_2}\}=\{A,B\}$, where $A=[h_2,1,h_2,h_3,1,h_2,1,h_3]$ and $B=[h_2,1,h_3,h_2,1,h_2,h_3,1]$. So, there are two subcases as follows:
\begin{enumerate}
\item[a)]
Let  $T^1_{C_1}=A$ and $T^1_{C_2}=B$. Then $T^i_{C_i}=A$ and $T^i_{C_{i+1}}=B$, for $i\in \{1,2,\ldots,m-1\}$. So by \ref{e-new2}, $a_1=h_3a_2, a_2=h_3a_3, \ldots,a_m=h_3a_1$. Therefore, $a_1=h_3^ma_1$ and so $h_3^m=1$, a contradiction because $G$ is torsion-free and $\vert supp(\alpha)\vert=3$.
\item[b)]
Let  $T^1_{C_1}=B$ and $T^1_{C_2}=A$. Then $T^i_{C_i}=B$ and $T^i_{C_{i+1}}=A$, for $i\in \{1,2,\ldots,m-1\}$. So by \ref{e-new2}, $h_3a_1=a_2, h_3a_2=a_3, \ldots,h_3a_m=a_1$. Therefore, $a_1=h_3^ma_1$ and so $h_3^m=1$, a contradiction because $G$ is torsion-free and $\vert supp(\alpha)\vert=3$.
\end{enumerate}

{\bf Cases 2-6:} The proof for the remaining five possible cases of $\{T^1_{C_1},T^1_{C_2}\}$ are similar to the proof for case 1.

Therefore, the graph $Z(\alpha,\beta)$ cannot be isomorphic to the graph $L_n$.

Now suppose, for a contradiction, that $Z(\alpha,\beta)$ is isomorphic to the graph $M_n$. So if $m=n/2$, then $Z(\alpha,\beta)$ contains a subgraph as Figure \ref{f-4}, where $\{a_1,b_1,\ldots,a_m, b_m\}=supp(\beta)$, $a_{m+1}=b_1$ and $b_{m+1}=a_1$. We denote by $C_1,C_2,\ldots,C_m$ the consecutive squares in Figure \ref{f-4}, from the left to the right respectively. With the above discussion, $T^i_{C_i}=T^1_{C_1}$ and $T^i_{C_{i+1}}=T^1_{C_2}$, for $i\in \{1,2,\ldots,m-1\}$. By Remark \ref{r-1}, there are $6$ possible cases for $\{T^1_{C_1},T^1_{C_2}\}$. Here we show that for each one of the latter cases, there is a contradiction:

{\bf Case 1:} Let $\{T^1_{C_1},T^1_{C_2}\}=\{A,B\}$, where $A=[h_2,1,h_2,h_3,1,h_2,1,h_3]$ and $B=[h_2,1,h_3,h_2,1,h_2,h_3,1]$. So, there are two subcases as follows:
\begin{enumerate}
\item[a)]
Let  $T^1_{C_1}=A$ and $T^1_{C_2}=B$. Then $T^i_{C_i}=A$ and $T^i_{C_{i+1}}=B$, for $i\in \{1,2,\ldots,m-1\}$. So by \ref{e-new2}, $a_1=h_3a_2, a_2=h_3a_3, \ldots,a_{m-1}=h_3a_m$, $a_m=h_3b_1$ and $b_1=h_2a_1$. Therefore, $a_1=h_3^mh_2a_1$ and so $h_3^mh_2=1$. Hence, $h_2=h_3^{-m}$ and so $G$ is abelian, a contradiction because abelian torsion-free groups satisfy the Conjecture \ref{conj-zero} (see \cite[Theorem 26.2]{pass2}).
\item[b)]
Let  $T^1_{C_1}=B$ and $T^1_{C_2}=A$. Then $T^i_{C_i}=B$ and $T^i_{C_{i+1}}=A$, for $i\in \{1,2,\ldots,m-1\}$. So by \ref{e-new2}, $h_3a_1=a_2, h_3a_2=a_3, \ldots,h_3a_{m-1}=a_m$, $h_3a_m=b_1$ and $b_1=h_2a_1$. Therefore, $a_1=h_2^{-1}h_3^ma_1$ and so  $h_2^{-1}h_3^m=1$. Hence, $h_2=h_3^m$ and so $G$ is abelian, a contradiction because abelian torsion-free groups satisfy the Conjecture \ref{conj-zero}.
\end{enumerate}

{\bf Cases 2-6:} The proof for the remaining five possible cases of $\{T^1_{C_1},T^1_{C_2}\}$ are similar to the proof for case 1.

Therefore, the graph $Z(\alpha,\beta)$ cannot be isomorphic to the graph $M_n$. This completes the proof.
\end{proof}

%------------------------------------------------------------------------------------------------------------------------

\section{\bf Forbidden subgraphs of Zero-divisor graphs over $\mathbb{F}_2$ on any torsion-free group}\label{S2}
In Section \ref{S1-2}, we studied the existence of triangles and squares in zero-divisor graphs of length $3$ over an arbitrary field and on any torsion-free group and we showed that $C_3$ and $K_{2,3}$ are two forbidden subgraphs of such graphs.

Throughout this section suppose that $\alpha$ and $\beta$ are non-zero elements of the group algebra $ \mathbb{F}_2[G]$ of a torsion-free group $G$ such that $|supp(\alpha)|=3$, $\alpha\beta=0$ and if  $\alpha \beta'=0$ for some non-zero element $\beta'$ of the group algebra, then  $|supp(\beta')|\geq |supp(\beta)|$. By Remark \ref{r-G}, one may assume that $1 \in supp(\alpha)$ and $G=\langle supp(\alpha) \rangle$. Let $supp(\alpha)=\{ 1,h_2,h_3\}$ and $n=|supp(\beta)|$.

With the same discussion such as about $C_3$ and $C_4$ cycles, we can study cycles of other lengths in the graph $Z(\alpha,\beta)$ by using their relations. In this section, by using cycles up to lengths $7$ and their relations, we find additional forbidden subgraphs of the graph $Z(\alpha,\beta)$. The procedure of finding such graphs is similar to the procedure of finding previous examples. So, the frequent tedious details are  omitted. In Appendix  of \cite{Ab-Ta}, some details of our computations are given for the reader's convenience. Forbidden subgraphs of $Z(\alpha,\beta)$ are listed in Table \ref{tab-forbiddens}. In the following, we give some details about such subgraphs. 

\subsection{$\mathbf{K_{2,3}}$}
%1) K2,3------------------------------------------------------------------------------------------------------------------------

By Theorem \ref{Z-K2,3}, $Z(\alpha,\beta)$ contains no $K_{2,3}$ as a subgraph.

\subsection{$\mathbf{C_4--C_5}$}
%2) C4--C5------------------------------------------------------------------------------------------------------------------------
%$\mathbf{2) \ C_4--C_5}$ \textbf{subgraph:} 
It can be seen that there are $121$ different cases for the relations of the cycles $C_4$ and $C_5$ in the subgraph $C_4--C_5$. Using GAP \cite{gap}, we see that in $111$ cases of these $121$ cases, the groups which are obtained are finite and solvable, a contradiction. So, there are just $10$ cases which may lead to the existence of a subgraph isomorphic to the graph $C_4--C_5$ in $Z(\alpha,\beta)$. It can be seen that each of such cases gives a contradiction and so, the graph $Z(\alpha,\beta)$ contains no subgraph isomorphic to the graph $C_4--C_5$.

We note that each group with two generators $h_2$ and $h_3$ and two relations which is one of the $10$ latter cases is a quotient of $B(1,k)$, for some integer $k$, or has a torsion element.

\subsection{$\mathbf{C_4--C_6}$}
%3) C4--C6------------------------------------------------------------------------------------------------------------------------
%$\mathbf{3) \ C_4--C_6}$ \textbf{subgraph:} 
It can be seen that there are $658$ different cases for the relations of the cycles $C_4$ and $C_6$ in the subgraph $C_4--C_6$. Using GAP \cite{gap}, we see that in $632$ cases of these $658$ cases, the groups which are obtained are finite or solvable, a contradiction. So, there are just $20$ cases which may lead to the existence of a subgraph isomorphic to the graph $C_4--C_6$ in $Z(\alpha,\beta)$. It can be seen that each of such cases gives a contradiction and so, the graph $Z(\alpha,\beta)$ contains no subgraph isomorphic to the graph $C_4--C_6$.

We note that each group with two generators $h_2$ and $h_3$ and two relations which is one of the  $20$ latter cases is a quotient of $B(1,k)$, for some integer $k$, has a torsion element or is a cyclic group.

\subsection{$\mathbf{C_4-C_5(-C_5-)}$}
%4) C4-C5(-C5-)----------------------------------------------------------------------------------------------------------------------
%$\mathbf{4) \ C_4-C_5(-C_5-)}$ \textbf{subgraph:} 
It can be seen that there are $42$ different cases for the relations of a cycle $C_4$ and two cycles $C_5$ in the subgraph $C_4-C_5(-C_5-)$. Using GAP \cite{gap}, we see that in $38$ cases of these $42$ cases, the groups which are obtained are finite and solvable, a contradiction. So, there are just $4$ cases which may lead to the existence of a subgraph isomorphic to the graph $C_4-C_5(-C_5-)$ in $Z(\alpha,\beta)$. It can be seen that each of such cases gives a contradiction and so, the graph $Z(\alpha,\beta)$ contains no subgraph isomorphic to the graph $C_4-C_5(-C_5-)$.

We note that each group with two generators $h_2$ and $h_3$ and three relations which is one of the  $4$ latter cases is a quotient of $B(1,k)$, for some integer $k$, has a torsion element.

\subsection{$\mathbf{C_4-C_5(-C_4-)}$}
%5) C4-C5(-C4-)----------------------------------------------------------------------------------------------------------------------
%$\mathbf{5) \ C_4-C_5(-C_4-)}$ \textbf{subgraph:} 
It can be seen that there are $4$ different cases for the relations of two cycles $C_4$ and a cycle $C_5$ in the subgraph $C_4-C_5(-C_4-)$. Using GAP \cite{gap}, we see that in all of such cases, the groups which are obtained are finite and solvable, a contradiction. So, the graph $Z(\alpha,\beta)$ contains no subgraph isomorphic to the graph $C_4-C_5(-C_4-)$.

\subsection{$\mathbf{C_4-C_5(-C_6--)}$}
%6) C4-C5(-C_6--)-------------------------------------------------------------------------------------------------------------------
%$\mathbf{6) \ C_4-C_5(-C_6--)}$ \textbf{subgraph:} 
It can be seen that there are $126$ different cases for the relations of a cycle $C_4$, a cycle $C_5$ and a cycle $C_6$ in the subgraph $C_4-C_5(-C_6--)$. Using GAP \cite{gap}, we see that in $122$ cases of these $126$ cases, the groups which are obtained are finite and solvable, a contradiction. So, there are just $4$ cases which may lead to the existence of a subgraph isomorphic to the graph $C_4-C_5(-C_6--)$ in $Z(\alpha,\beta)$. It can be seen that each of such cases gives a contradiction and so, the graph $Z(\alpha,\beta)$ contains no subgraph isomorphic to the graph $C_4-C_5(-C_6--)$.

We note that each group with two generators $h_2$ and $h_3$ and three relations which is one of the  $4$ latter cases is a quotient of $B(1,k)$, for some integer $k$.

\subsection{$\mathbf{C_4-C_5(-C_6-)}$}
%7) C4-C5(-C6-)----------------------------------------------------------------------------------------------------------------------
%$\mathbf{7) \ C_4-C_5(-C_6-)}$ \textbf{subgraph:} 
It can be seen that there are $462$ different cases for the relations of a cycle $C_4$, a cycle $C_5$ and a cycle $C_6$ in the subgraph $C_4-C_5(-C_6-)$. Using GAP \cite{gap}, we see that in $436$ cases of these $462$ cases, the groups which are obtained are finite or solvable, a contradiction. So, there are just $22$ cases which may lead to the existence of a subgraph isomorphic to the graph $C_4-C_5(-C_6-)$ in $Z(\alpha,\beta)$. It can be seen that each of such cases gives a contradiction and so, the graph $Z(\alpha,\beta)$ contains no subgraph isomorphic to the graph $C_4-C_5(-C_6-)$.

We note that each group with two generators $h_2$ and $h_3$ and three relations which is one of the  $22$ latter cases is a quotient of $B(1,k)$, for some integer $k$, is a cyclic group or is a solvable group.

\subsection{$\mathbf{C_4-C_5(-C_7--)}$}
%8) C4-C5(-C7--)---------------------------------------------------------------------------------------------------------------------
%$\mathbf{8) \ C_4-C_5(-C_7--)}$ \textbf{subgraph:} 
It can be seen that there are $648$ different cases for the relations of a cycle $C_4$, a cycle $C_5$ and a cycle $C_7$ in the subgraph $C_4-C_5(-C_7--)$. Using GAP \cite{gap}, we see that in $608$ cases of these $648$ cases, the groups which are obtained are finite and solvable, a contradiction. So, there are just $40$ cases which may lead to the existence of a subgraph isomorphic to the graph $C_4-C_5(-C_7--)$ in $Z(\alpha,\beta)$. It can be seen that each of such cases gives a contradiction and so, the graph $Z(\alpha,\beta)$ contains no subgraph isomorphic to the graph $C_4-C_5(-C_7--)$.

We note that each group with two generators $h_2$ and $h_3$ and three relations which is one of the $40$ latter cases is a quotient of $B(1,k)$, for some integer $k$, has a torsion element, is a cyclic group or is a solvable group.

\subsection{$\mathbf{C_5--C_5(--C_5)}$}
%9) C5--C5(--C5)--------------------------------------------------------------------------------------------------------------------
%$\mathbf{9) \ C_5--C_5(--C_5)}$ \textbf{subgraph:} 
It can be seen that there are $192$ different cases for the relations of three cycles $C_5$ in the subgraph $C_5--C_5(--C_5)$. Using GAP \cite{gap}, we see that in $188$ cases of these $192$ cases, the groups which are obtained are finite and solvable, a contradiction. So, there are just $4$ cases which may lead to the existence of a subgraph isomorphic to the graph $C_5--C_5(--C_5)$ in $Z(\alpha,\beta)$. It can be seen that each of such cases gives a contradiction and so, the graph $Z(\alpha,\beta)$ contains no subgraph isomorphic to the graph $C_5--C_5(--C_5)$.

We note that each group with two generators $h_2$ and $h_3$ and three relations which is one of the $4$ latter cases is a quotient of $B(1,k)$, for some integer $k$.

\subsection{$\mathbf{C_5--C_5(--C_6)}$}
%10) C5--C5(--C6)-------------------------------------------------------------------------------------------------------------------
%$\mathbf{10) \ C_5--C_5(--C_6)}$ \textbf{subgraph:} 
It can be seen that there are $1006$ different cases for the relations of two cycles $C_5$ and a cycle $C_6$ in the subgraph $C_5--C_5(--C_6)$. Using GAP \cite{gap}, we see that in $986$ cases of these $1006$ cases, the groups which are obtained are finite and solvable, a contradiction. So, there are just $20$ cases which may lead to the existence of a subgraph isomorphic to the graph $C_5--C_5(--C_6)$ in $Z(\alpha,\beta)$. It can be seen that each of such cases gives a contradiction and so, the graph $Z(\alpha,\beta)$ contains no subgraph isomorphic to the graph $C_5--C_5(--C_6)$.

We note that each group with two generators $h_2$ and $h_3$ and three relations which is one of the $20$ latter cases is a quotient of $B(1,k)$, for some integer $k$, is a cyclic group or is a solvable group.

\subsection{$\mathbf{C_4-C_6(--C_7--)(C_7-1)}$}
%11) C4-C6(--C7--)(--C7--)---------------------------------------------------------------------------------------------------------
%$\mathbf{11) \ C_4-C_6(--C_7--)(C_7-1)}$ \textbf{subgraph:} 
It can be seen that there are $176$ different cases for the relations of a cycle $C_4$, two cycles $C_7$ and a cycle $C_6$ in the subgraph $C_4-C_6(--C_7--)(C_7-1)$. Using GAP \cite{gap}, we see that in all of such cases, the groups which are obtained are finite and solvable, a contradiction. So, the graph $Z(\alpha,\beta)$ contains no subgraph isomorphic to the graph $C_4-C_6(--C_7--)(C_7-1)$.

\subsection{$\mathbf{C_4-C_6(--C_7--)(--C_5-)}$}
%12) C4-C6(--C7--)(--C5-)----------------------------------------------------------------------------------------------------------
%$\mathbf{12) \ C_4-C_6(--C_7--)(--C_5-)}$ \textbf{subgraph:} 
It can be seen that there are $28$ different cases for the relations of a cycle $C_4$, a cycle $C_6$, a cycle $C_7$ and a cycle $C_5$ in the subgraph $C_4-C_6(--C_7--)(--C_5-)$. Using GAP \cite{gap}, we see that in $24$ cases of these $28$ cases, the groups which are obtained are finite and solvable, a contradiction. So, there are just $4$ cases which may lead to the existence of a subgraph isomorphic to the graph $C_4-C_6(--C_7--)(--C_5-)$ in $Z(\alpha,\beta)$. It can be seen that each of such cases gives a contradiction and so, the graph $Z(\alpha,\beta)$ contains no subgraph isomorphic to the graph $C_4-C_6(--C_7--)(--C_5-)$.

We note that each group with two generators $h_2$ and $h_3$ and four relations which is one of the $4$ latter cases is a cyclic group.

\subsection{$\mathbf{C_4-C_6(-C_6--)(-C_4-)}$}
%13) \ C4-C6(-C6--)(-C4-)----------------------------------------------------------------------------------------------------------
%$\mathbf{13) \ C_4-C_6(-C_6--)(-C_4-)}$ \textbf{subgraph:} 
It can be seen that there is no case for the relations of two cycles $C_4$ and two cycles $C_6$ in the subgraph $C_4-C_6(-C_6--)(-C_4-)$. It means that the graph $Z(\alpha,\beta)$ contains no subgraph isomorphic to the graph $C_4-C_6(-C_6--)(-C_4-)$.

\subsection{$\mathbf{C_4-C_6(-C_6--)(--C_5-)}$}
%14) C4-C6(-C6--)(--C5-)-----------------------------------------------------------------------------------------------------------
%$\mathbf{14) \ C_4-C_6(-C_6--)(--C_5-)}$ \textbf{subgraph:} 
It can be seen that there are $22$ different cases for the relations of a cycle $C_4$, two cycles $C_6$ and a cycle $C_5$ in the subgraph $C_4-C_6(-C_6--)(--C_5-)$. Using GAP \cite{gap}, we see that in all of such cases, the groups which are obtained are finite and solvable, a contradiction. So, the graph $Z(\alpha,\beta)$ contains no subgraph isomorphic to the graph $C_4-C_6(-C_6--)(--C_5-)$.

\subsection{$\mathbf{C_4-C_6(-C_6--)(C_6---)}$}
%15) C4-C6(-C6--)(C6---)-----------------------------------------------------------------------------------------------------------
%$\mathbf{15) \ C_4-C_6(-C_6--)(C_6---)}$ \textbf{subgraph:} 
It can be seen that there are $66$ different cases for the relations of a cycle $C_4$ and three cycles $C_6$ in the subgraph $C_4-C_6(-C_6--)(C_6---)$. Using GAP \cite{gap}, we see that in $62$ cases of these $66$ cases, the groups which are obtained are finite and solvable, a contradiction. So, there are just $4$ cases which may lead to the existence of a subgraph isomorphic to the graph $C_4-C_6(-C_6--)(C_6---)$ in $Z(\alpha,\beta)$. It can be seen that each of such cases gives a contradiction and so, the graph $Z(\alpha,\beta)$ contains no subgraph isomorphic to the graph $C_4-C_6(-C_6--)(C_6---)$.

We note that each group with two generators $h_2$ and $h_3$ and four relations which is one of the $4$ latter cases is a quotient of $B(1,k)$, for some integer $k$.

\subsection{$\mathbf{C_4-C_6(-C_6--)(---C_4)}$}
%16) C4-C6(-C6--)(---C4)----------------------------------------------------------------------------------------------------------
%$\mathbf{16) \ C_4-C_6(-C_6--)(---C_4)}$ \textbf{subgraph:} 
It can be seen that there is no case for the relations of two cycles $C_4$ and two cycles $C_6$ in the subgraph $C_4-C_6(-C_6--)(---C_4)$. It means that the graph $Z(\alpha,\beta)$ contains no subgraph isomorphic to the graph $C_4-C_6(-C_6--)(---C_4)$.

\subsection{$\mathbf{C_5-C_5(--C_6--)}$}
%17) C5-C5(--C6--)------------------------------------------------------------------------------------------------------------------
%$\mathbf{17) \ C_5-C_5(--C_6--)}$ \textbf{subgraph:} 
It can be seen that there are $440$ different cases for the relations of two cycles $C_5$ and a cycle $C_6$ in the subgraph $C_5-C_5(--C_6--)$. Using GAP \cite{gap}, we see that in $404$ cases of these $440$ cases, the groups which are obtained are finite and solvable, a contradiction. So, there are just $36$ cases which may lead to the existence of a subgraph isomorphic to the graph $C_5-C_5(--C_6--)$ in $Z(\alpha,\beta)$. It can be seen that each of such cases gives a contradiction and so, the graph $Z(\alpha,\beta)$ contains no subgraph isomorphic to the graph $C_5-C_5(--C_6--)$.

We note that each group with two generators $h_2$ and $h_3$ and three relations which is one of the $36$ latter cases is a quotient of $B(1,k)$, for some integer $k$, has a torsion element or is a cyclic group.

\subsection{$\mathbf{C_5-C_5(-C_6--)(C_6---)}$}
%18) C5-C5(-C6--)(C6---)-----------------------------------------------------------------------------------------------------------
%$\mathbf{18) \ C_5-C_5(-C_6--)(C_6---)}$ \textbf{subgraph:} 
It can be seen that there are $56$ different cases for the relations of two cycles $C_5$ and two cycles $C_6$ in the subgraph $C_5-C_5(-C_6--)(C_6---)$. Using GAP \cite{gap}, we see that in $54$ cases of these $56$ cases, the groups which are obtained are finite or solvable, a contradiction. So, there are just $2$ cases which may lead to the existence of a subgraph isomorphic to the graph $C_5-C_5(-C_6--)(C_6---)$ in $Z(\alpha,\beta)$. It can be seen that each of such cases gives a contradiction and so, the graph $Z(\alpha,\beta)$ contains no subgraph isomorphic to the graph $C_5-C_5(-C_6--)(C_6---)$.

We note that each group with two generators $h_2$ and $h_3$ and four relations which is one of the $2$ latter cases is a cyclic group.

\subsection{$\mathbf{C_5-C_5(-C_6--)(--C_6-1)}$}
%19) C5-C5(-C6--)(--C6- 1)---------------------------------------------------------------------------------------------------------
%$\mathbf{19) \ C_5-C_5(-C_6--)(--C_6-1)}$ \textbf{subgraph:} 
It can be seen that there are $56$ different cases for the relations of two cycles $C_5$ and two cycles $C_6$ in the subgraph $C_5-C_5(-C_6--)(--C_6-1)$. Using GAP \cite{gap}, we see that in all of such cases, the groups which are obtained are finite and solvable, a contradiction. So, the graph $Z(\alpha,\beta)$ contains no subgraph isomorphic to the graph $C_5-C_5(-C_6--)(--C_6-1)$.

\subsection{$\mathbf{C_5-C_5(-C_6--)(-C_5--)}$}
%20) C5-C5(-C6--)(-C5--)----------------------------------------------------------------------------------------------------------
%$\mathbf{20) \ C_5-C_5(-C_6--)(-C_5--)}$ \textbf{subgraph:} 
It can be seen that there are $14$ different cases for the relations of three cycles $C_5$ and one cycle $C_6$ in the subgraph $C_5-C_5(-C_6--)(-C_5--)$. Using GAP \cite{gap}, we see that in all of such cases, the groups which are obtained are finite and solvable, a contradiction. So, the graph $Z(\alpha,\beta)$ contains no subgraph isomorphic to the graph $C_5-C_5(-C_6--)(-C_5--)$.

\subsection{$\mathbf{C_6---C_6(C_6---C_6)}$}
%21) C6---C6(C6---C6)--------------------------------------------------------------------------------------------------------------
%$\mathbf{21) \ C_6---C_6(C_6---C_6)}$ \textbf{subgraph:} 
It can be seen that there are $46$ different cases for the relations of four $C_6$ cycles in the subgraph $C_6---C_6(C_6---C_6)$. Using GAP \cite{gap}, we see that in $30$ cases of these $46$ cases, the groups which are obtained are finite or solvable, a contradiction. So, there are just $16$ cases which may lead to the existence of a subgraph isomorphic to the graph $C_6---C_6(C_6---C_6)$ in $Z(\alpha,\beta)$. It can be seen that each of such cases gives a contradiction and so, the graph $Z(\alpha,\beta)$ contains no subgraph isomorphic to the graph $C_6---C_6(C_6---C_6)$.

We note that each group with two generators $h_2$ and $h_3$ and four relations which is one of the $16$ latter cases is a quotient of $B(1,k)$, for some integer $k$, or has a torsion element.

\subsection{$\mathbf{C_6---C_6(C_6)(C_6)(C_6)}$}
%22) C_6---C_6(C_6)(C_6)(C_6)------------------------------------------------------------------------------------------------
%$\mathbf{22) \ C_6---C_6(C_6)(C_6)(C_6)}$ \textbf{subgraph:} 
It can be seen that there are $10$ different cases for the relations of five $C_6$ cycles in the subgraph $C_6---C_6(C_6)(C_6)(C_6)$. Using GAP \cite{gap}, we see that in $6$ cases of these $10$ cases, the groups which are obtained are finite or solvable, a contradiction. So, there are just $4$ cases which may lead to the existence of a subgraph isomorphic to the graph $C_6---C_6(C_6)(C_6)(C_6)$ in $Z(\alpha,\beta)$. It can be seen that each of such cases gives a contradiction and so, the graph $Z(\alpha,\beta)$ contains no subgraph isomorphic to the graph $C_6---C_6(C_6)(C_6)(C_6)$.

We note that each group with two generators $h_2$ and $h_3$ and five relations which is one of the $4$ latter cases is a quotient of $B(1,k)$, for some integer $k$, or is a cyclic group.

\subsection{$\mathbf{C_5(--C_6--)C_5(---C_6)}$}
%23) C_5(--C_6--)C_5(---C_6)----------------------------------------------------------------------------------------------------
%$\mathbf{23) \ C_5(--C_6--)C_5(---C_6)}$ \textbf{subgraph:} 
It can be seen that there are $134$ different cases for the relations of two cycles $C_5$ and two cycles $C_6$ in the subgraph $C_5(--C_6--)C_5(---C_6)$. Using GAP \cite{gap}, we see that in $130$ cases of these $134$ cases, the groups which are obtained are finite or solvable, a contradiction. So, there are just $4$ cases which may lead to the existence of a subgraph isomorphic to the graph $C_5(--C_6--)C_5(---C_6)$ in $Z(\alpha,\beta)$. It can be seen that each of such cases gives a contradiction and so, the graph $Z(\alpha,\beta)$ contains no subgraph isomorphic to the graph $C_5(--C_6--)C_5(---C_6)$.

We note that each group with two generators $h_2$ and $h_3$ and four relations which is one of the $4$ latter cases is a quotient of $B(1,k)$, for some integer $k$, or is a cyclic group.

\subsection{$\mathbf{C_6--C_6(C_6--C_6)}$}
%24) C_6--C_6(C_6--C_6)--------------------------------------------------------------------------------------------------------
%$\mathbf{24) \ C_6--C_6(C_6--C_6)}$ \textbf{subgraph:} 
It can be seen that there are $5119$ different cases for the relations of four $C_6$ cycles in the subgraph $C_6--C_6(C_6--C_6)$. Using GAP \cite{gap}, we see that in $4983$ cases of these $5119$ cases, the groups which are obtained are finite or solvable, a contradiction. So, there are just $136$ cases which may lead to the existence of a subgraph isomorphic to the graph $C_6--C_6(C_6--C_6)$ in $Z(\alpha,\beta)$. It can be seen that each of such cases gives a contradiction and so, the graph $Z(\alpha,\beta)$ contains no subgraph isomorphic to the graph $C_6--C_6(C_6--C_6)$.

We note that each group with two generators $h_2$ and $h_3$ and four relations which is one of the $136$ latter cases is a quotient of $B(1,k)$, for some integer $k$, has a torsion element or is a cyclic group.

\subsection{$\mathbf{C_6---C_6(C_6--C_6)}$}
%25) C_6---C_6(C_6--C_6)--------------------------------------------------------------------------------------------------------
%$\mathbf{25) \ C_6---C_6(C_6--C_6)}$ \textbf{subgraph:} 
It can be seen that there are $1594$ different cases for the relations of four $C_6$ cycles in the subgraph $C_6---C_6(C_6--C_6)$. Using GAP \cite{gap}, we see that in $1446$ cases of these $1594$ cases, the groups which are obtained are finite or solvable, a contradiction. So, there are just $148$ cases which may lead to the existence of a subgraph isomorphic to the graph $C_6---C_6(C_6--C_6)$ in $Z(\alpha,\beta)$. It can be seen that each of such cases gives a contradiction and so, the graph $Z(\alpha,\beta)$ contains no subgraph isomorphic to the graph $C_6---C_6(C_6--C_6)$.

We note that each group with two generators $h_2$ and $h_3$ and four relations which is one of the $148$ latter cases is a quotient of $B(1,k)$, for some integer $k$, has a torsion element or is a cyclic group.

\subsection{$\mathbf{C_6---C_6(-C_5-)}$}
%26) C_6---C_6(-C_5-)--------------------------------------------------------------------------------------------------------------
%$\mathbf{26) \ C_6---C_6(-C_5-)}$ \textbf{subgraph:} 
It can be seen that there are $1482$ different cases for the relations of two cycles $C_6$ and a cycle $C_5$ in the subgraph $C_6---C_6(-C_5-)$. Using GAP \cite{gap}, we see that in $1358$ cases of these $1482$ cases, the groups which are obtained are finite and solvable, a contradiction. So, there are just $124$ cases which may lead to the existence of a subgraph isomorphic to the graph $C_6---C_6(-C_5-)$ in $Z(\alpha,\beta)$. It can be seen that each of such cases gives a contradiction and so, the graph $Z(\alpha,\beta)$ contains no subgraph isomorphic to the graph $C_6---C_6(-C_5-)$.

We note that each group with two generators $h_2$ and $h_3$ and three relations which is one of the $124$ latter cases is a quotient of $B(1,k)$, for some integer $k$, has a torsion element or is a cyclic group.

\subsection{$\mathbf{C_4-C_6(--C_7--)(---C_6)}$}
%27) C_4-C_6(--C_7--)(---C_6)-----------------------------------------------------------------------------------------------------
%$\mathbf{27) \ C_4-C_6(--C_7--)(---C_6)}$ \textbf{subgraph:} 
It can be seen that there are $124$ different cases for the relations of a cycle $C_4$, two cycles $C_6$ and a cycle $C_7$ in the subgraph $C_4-C_6(--C_7--)(---C_6)$. Using GAP \cite{gap}, we see that in $112$ cases of these $124$ cases, the groups which are obtained are finite and solvable, a contradiction. So, there are just $12$ cases which may lead to the existence of a subgraph isomorphic to the graph $C_4-C_6(--C_7--)(---C_6)$ in $Z(\alpha,\beta)$. It can be seen that each of such cases gives a contradiction and so, the graph $Z(\alpha,\beta)$ contains no subgraph isomorphic to the graph $C_4-C_6(--C_7--)(---C_6)$.

We note that each group with two generators $h_2$ and $h_3$ and four relations which is one of the $12$ latter cases is a quotient of $B(1,k)$, for some integer $k$.

\subsection{$\mathbf{C_4-C_6(--C_7--)(C_4)(C_4)}$}
%28) C_4-C_6(--C_7--)(C_4)(C_4)--------------------------------------------------------------------------------------------------
%$\mathbf{28) \ C_4-C_6(--C_7--)(C_4)(C_4)}$ \textbf{subgraph:} 
It can be seen that there are $8$ different cases for the relations of three cycles $C_4$, a cycle $C_7$ and a cycle $C_6$ in the subgraph $C_4-C_6(--C_7--)(C_4)(C_4)$. Using GAP \cite{gap}, we see that in all of such cases, the groups which are obtained are finite and solvable, a contradiction. So, the graph $Z(\alpha,\beta)$ contains no subgraph isomorphic to the graph $C_4-C_6(--C_7--)(C_4)(C_4)$.

\subsection{$\mathbf{C_6---C_6(-C_5--)}$}
%29) C_6---C_6(-C_5--)------------------------------------------------------------------------------------------------------------
%$\mathbf{29) \ C_6---C_6(-C_5--)}$ \textbf{subgraph:} 
It can be seen that there are $418$ different cases for the relations of two cycles $C_6$ and a cycle $C_5$ in the subgraph $C_6---C_6(-C_5--)$. Using GAP \cite{gap}, we see that in $358$ cases of these $418$ cases, the groups which are obtained are finite and solvable, a contradiction. So, there are just $60$ cases which may lead to the existence of a subgraph isomorphic to the graph $C_6---C_6(-C_5--)$ in $Z(\alpha,\beta)$. It can be seen that each of such cases gives a contradiction and so, the graph $Z(\alpha,\beta)$ contains no subgraph isomorphic to the graph $C_6---C_6(-C_5--)$.

We note that each group with two generators $h_2$ and $h_3$ and three relations which is one of the $60$ latter cases is a quotient of $B(1,k)$, for some integer $k$, has a torsion element, is a cyclic group or is a solvable group.

\subsection{$\mathbf{C_6--C_6(--C_5-)(-C_5-)}$}
%30) C_6--C_6(--C_5-)(-C_5-)-----------------------------------------------------------------------------------------------------
%$\mathbf{30) \ C_6--C_6(--C_5-)(-C_5-)}$ \textbf{subgraph:} 
It can be seen that there are $62$ different cases for the relations of two cycles $C_5$ and two cycles $C_6$ in the subgraph $C_6--C_6(--C_5-)(-C_5-)$. Using GAP \cite{gap}, we see that in $56$ cases of these $62$ cases, the groups which are obtained are finite or solvable, a contradiction. So, there are just $6$ cases which may lead to the existence of a subgraph isomorphic to the graph $C_6--C_6(--C_5-)(-C_5-)$ in $Z(\alpha,\beta)$. It can be seen that each of such cases gives a contradiction and so, the graph $Z(\alpha,\beta)$ contains no subgraph isomorphic to the graph $C_6--C_6(--C_5-)(-C_5-)$.

We note that each group with two generators $h_2$ and $h_3$ and four relations which is one of the $6$ latter cases has a torsion element or is a cyclic group.

\subsection{$\mathbf{C_6--C_6(--C_5-)(C_6---)}$}
%31) C_6--C_6(--C_5-)(C_6---)----------------------------------------------------------------------------------------------------
%$\mathbf{31) \ C_6--C_6(--C_5-)(C_6---)}$ \textbf{subgraph:} 
It can be seen that there are $76$ different cases for the relations of a cycle $C_5$ and three cycles $C_6$ in the subgraph $C_6--C_6(--C_5-)(C_6---)$. Using GAP \cite{gap}, we see that in $64$ cases of these $76$ cases, the groups which are obtained are finite or solvable, a contradiction. So, there are just $12$ cases which may lead to the existence of a subgraph isomorphic to the graph $C_6--C_6(--C_5-)(C_6---)$ in $Z(\alpha,\beta)$. It can be seen that each of such cases gives a contradiction and so, the graph $Z(\alpha,\beta)$ contains no subgraph isomorphic to the graph $C_6--C_6(--C_5-)(C_6---)$.

We note that each group with two generators $h_2$ and $h_3$ and four relations which is one of the $12$ latter cases is a quotient of $B(1,k)$, for some integer $k$, has a torsion element or is a cyclic group.

\subsection{$\mathbf{C_5(--C_6--)C_5(C_6)}$}
%32) C_5(--C_6--)C_5(C_6)--------------------------------------------------------------------------------------------
%$\mathbf{32) \ C_5(--C_6--)C_5(C_6)}$ \textbf{subgraph:} 
It can be seen that there are $120$ different cases for the relations of two cycles $C_5$ and two cycles $C_6$ in the subgraph $C_5(--C_6--)C_5(C_6)$. Using GAP \cite{gap}, we see that in $104$ cases of these $120$ cases, the groups which are obtained are finite or solvable, a contradiction. So, there are just $16$ cases which may lead to the existence of a subgraph isomorphic to the graph $C_5(--C_6--)C_5(C_6)$ in $Z(\alpha,\beta)$. It can be seen that each of such cases gives a contradiction and so, the graph $Z(\alpha,\beta)$ contains no subgraph isomorphic to the graph $C_5(--C_6--)C_5(C_6)$.

We note that each group with two generators $h_2$ and $h_3$ and four relations which is one of the $16$ latter cases is a quotient of $B(1,k)$, for some integer $k$, has a torsion element or is a cyclic group.

\subsection{$\mathbf{C_5(--C_6--)C_5(C_7)}$}
%33) C_5(--C_6--)C_5(C_7)--------------------------------------------------------------------------------------------------
%$\mathbf{33) \ C_5(--C_6--)C_5(C_7)}$ \textbf{subgraph:} 
It can be seen that there are $248$ different cases for the relations of two cycles $C_5$, a cycle $C_6$ and a cycle $C_7$ in the subgraph $C_5(--C_6--)C_5(C_7)$. Using GAP \cite{gap}, we see that in $220$ cases of these $248$ cases, the groups which are obtained are finite or solvable, a contradiction. So, there are just $28$ cases which may lead to the existence of a subgraph isomorphic to the graph $C_5(--C_6--)C_5(C_7)$ in $Z(\alpha,\beta)$. It can be seen that each of such cases gives a contradiction and so, the graph $Z(\alpha,\beta)$ contains no subgraph isomorphic to the graph $C_5(--C_6--)C_5(C_7)$.

We note that each group with two generators $h_2$ and $h_3$ and four relations which is one of the $28$ latter cases is a quotient of $B(1,k)$, for some integer $k$, has a torsion element or is a cyclic group.

\subsection{$\mathbf{C_5-C_5(--C_7--)(--C_5)}$}
%34) C_5-C_5(--C_7--)(--C_5)-----------------------------------------------------------------------------------------------------
%$\mathbf{34) \ C_5-C_5(--C_7--)(--C_5)}$ \textbf{subgraph:} 
It can be seen that there are $394$ different cases for the relations of a cycle $C_7$ and three cycles $C_5$ in the subgraph $C_5-C_5(--C_7--)(--C_5)$. Using GAP \cite{gap}, we see that in $352$ cases of these $394$ cases, the groups which are obtained are finite or solvable, a contradiction. So, there are just $42$ cases which may lead to the existence of a subgraph isomorphic to the graph $C_5-C_5(--C_7--)(--C_5)$ in $Z(\alpha,\beta)$. It can be seen that each of such cases gives a contradiction and so, the graph $Z(\alpha,\beta)$ contains no subgraph isomorphic to the graph $C_5-C_5(--C_7--)(--C_5)$.

We note that each group with two generators $h_2$ and $h_3$ and four relations which is one of the $42$ latter cases is a quotient of $B(1,k)$, for some integer $k$, has a torsion element, is a cyclic group or is a solvable group.

\subsection{$\mathbf{C_5-C_5(--C_7--)(-C_5-)}$}
%35) C_5-C_5(--C_7--)(-C_5-)------------------------------------------------------------------------------------------------------
%$\mathbf{35) \ C_5-C_5(--C_7--)(-C_5-)}$ \textbf{subgraph:} 
It can be seen that there are $138$ different cases for the relations of a cycle $C_7$ and three cycles $C_5$ in the subgraph $C_5-C_5(--C_7--)(-C_5-)$. Using GAP \cite{gap}, we see that in $132$ cases of these $138$ cases, the groups which are obtained are finite or solvable, a contradiction. So, there are just $6$ cases which may lead to the existence of a subgraph isomorphic to the graph $C_5-C_5(--C_7--)(-C_5-)$ in $Z(\alpha,\beta)$. It can be seen that each of such cases gives a contradiction and so, the graph $Z(\alpha,\beta)$ contains no subgraph isomorphic to the graph $C_5-C_5(--C_7--)(-C_5-)$.

We note that each group with two generators $h_2$ and $h_3$ and four relations which is one of the $6$ latter cases has a torsion element, is a cyclic group or is a solvable group.

\subsection{$\mathbf{C_5-C_5(-C_6--)(--C_6-2)}$}
%36) C_5-C_5(-C_6--)(--C_6-2)-----------------------------------------------------------------------------------------------------
%$\mathbf{36) \ C_5-C_5(-C_6--)(--C_6-2)}$ \textbf{subgraph:} 
It can be seen that there are $22$ different cases for the relations of two cycles $C_5$ and two cycles $C_6$ in the subgraph $C_5-C_5(-C_6--)(--C_6-2)$. Using GAP \cite{gap}, we see that in all of such cases, the groups which are obtained are finite and solvable, a contradiction. So, the graph $Z(\alpha,\beta)$ contains no subgraph isomorphic to the graph $C_5-C_5(-C_6--)(--C_6-2)$.

\subsection{$\mathbf{C_4-C_4(-C_7-)(C_4)}$}
%37) C_4-C_4(-C_7-)(C_4)-------------------------------------------------------------------------------------------------------
%$\mathbf{37) \ C_4-C_4(-C_7-)(C_4)}$ \textbf{subgraph:} 
It can be seen that there are $32$ different cases for the relations of a cycle $C_7$ and three cycles $C_4$ in the subgraph $C_4-C_4(-C_7-)(C_4)$. Using GAP \cite{gap}, we see that in all of such cases, the groups which are obtained are finite and solvable, a contradiction. So, the graph $Z(\alpha,\beta)$ contains no subgraph isomorphic to the graph $C_4-C_4(-C_7-)(C_4)$.

\subsection{$\mathbf{C_5--C_5(-C_5--)}$}
%38) C_5--C_5(-C_5--)-------------------------------------------------------------------------------------------------------------
%$\mathbf{38) \ C_5--C_5(-C_5--)}$ \textbf{subgraph:} 
It can be seen that there are $64$ different cases for the relations of three cycles $C_5$ in the subgraph $C_5--C_5(-C_5--)$. Using GAP \cite{gap}, we see that in $58$ cases of these $64$ cases, the groups which are obtained are finite and solvable, a contradiction. So, there are just $6$ cases which may lead to the existence of a subgraph isomorphic to the graph $C_5--C_5(-C_5--)$ in $Z(\alpha,\beta)$. It can be seen that each of such cases gives a contradiction and so, the graph $Z(\alpha,\beta)$ contains no subgraph isomorphic to the graph $C_5--C_5(-C_5--)$.

We note that each group with two generators $h_2$ and $h_3$ and three relations which is one of the $6$ latter cases is a quotient of $B(1,k)$, for some integer $k$ or is a cyclic group.

\subsection{$\mathbf{C_6---C_6(-C_4)}$}
%39) C_6---C_6(-C_4)---------------------------------------------------------------------------------------------------------------
%$\mathbf{39) \ C_6---C_6(-C_4)}$ \textbf{subgraph:} 
It can be seen that there are $420$ different cases for the relations of two cycles $C_6$ and a cycle $C_4$ in the subgraph $C_6---C_6(-C_4)$. Using GAP \cite{gap}, we see that in $398$ cases of these $420$ cases, the groups which are obtained are finite or solvable, a contradiction. So, there are just $22$ cases which may lead to the existence of a subgraph isomorphic to the graph $C_6---C_6(-C_4)$ in $Z(\alpha,\beta)$. It can be seen that each of such cases gives a contradiction and so, the graph $Z(\alpha,\beta)$ contains no subgraph isomorphic to the graph $C_6---C_6(-C_4)$.

We note that each group with two generators $h_2$ and $h_3$ and three relations which is one of the $22$ latter cases is a cyclic group.

\subsection{$\mathbf{C_6--C_6(C_4)}$}
%40) C_6--C_6(C_4)---------------------------------------------------------------------------------------------------------------
%$\mathbf{40) \ C_6--C_6(C_4)}$ \textbf{subgraph:} 
It can be seen that there are $279$ different cases for the relations of two cycles $C_6$ and a cycle $C_4$ in the subgraph $C_6--C_6(C_4)$. Using GAP \cite{gap}, we see that in $268$ cases of these $279$ cases, the groups which are obtained are finite or solvable, a contradiction. So, there are just $11$ cases which may lead to the existence of a subgraph isomorphic to the graph $C_6--C_6(C_4)$ in $Z(\alpha,\beta)$. It can be seen that each of such cases gives a contradiction and so, the graph $Z(\alpha,\beta)$ contains no subgraph isomorphic to the graph $C_6--C_6(C_4)$.

We note that each group with two generators $h_2$ and $h_3$ and three relations which is one of the $11$ latter cases is a quotient of $B(1,k)$, for some integer $k$, or is a cyclic group.

\subsection{$\mathbf{C_4-C_6(-C_4)(-C_4)}$}
%41) C_4-C_6(-C_4)(-C_4)----------------------------------------------------------------------------------------------------------
%$\mathbf{41) \ C_4-C_6(-C_4)(-C_4)}$ \textbf{subgraph:} 
It can be seen that there are $36$ different cases for the relations of a cycle $C_6$ and three cycles $C_4$ in the subgraph $C_4-C_6(-C_4)(-C_4)$. Using GAP \cite{gap}, we see that in all of such cases, the groups which are obtained are solvable, a contradiction. So, the graph $Z(\alpha,\beta)$ contains no subgraph isomorphic to the graph $C_4-C_6(-C_4)(-C_4)$.

\subsection{$\mathbf{C_4-C_6(--C_7--)(-C_5-)}$}
%42) C_4-C_6(--C_7--)(-C_5-)------------------------------------------------------------------------------------------------------
%$\mathbf{42) \ C_4-C_6(--C_7--)(-C_5-)}$ \textbf{subgraph:} 
It can be seen that there are $62$ different cases for the relations of a cycle $C_4$, a cycle $C_6$, a cycle $C_7$ and a cycle $C_5$ in the subgraph $C_4-C_6(--C_7--)(-C_5-)$. Using GAP \cite{gap}, we see that in $58$ cases of these $62$ cases, the groups which are obtained are solvable, a contradiction. So, there are just $4$ cases which may lead to the existence of a subgraph isomorphic to the graph $C_4-C_6(--C_7--)(-C_5-)$ in $Z(\alpha,\beta)$. It can be seen that each of such cases gives a contradiction and so, the graph $Z(\alpha,\beta)$ contains no subgraph isomorphic to the graph $C_4-C_6(--C_7--)(-C_5-)$.

We note that each group with two generators $h_2$ and $h_3$ and four relations which is one of the $4$ latter cases is a cyclic group.

\subsection{$\mathbf{C_5-C_5(-C_6--)(--C_5-)}$}
%43) C_5-C_5(-C_6--)(--C_5-)-----------------------------------------------------------------------------------------------
%$\mathbf{43) \ C_5-C_5(-C_6--)(--C_5-)}$ \textbf{subgraph:} 
It can be seen that there are $14$ different cases for the relations of three cycles $C_5$ and a cycle $C_6$ in the subgraph $C_5-C_5(-C_6--)(--C_5-)$. Using GAP \cite{gap}, we see that in all of such cases, the groups which are obtained are solvable, a contradiction. So, the graph $Z(\alpha,\beta)$ contains no subgraph isomorphic to the graph $C_5-C_5(-C_6--)(--C_5-)$.

\subsection{$\mathbf{C_4-C_6(--C_7--)(C_7-2)}$}
%44) C_4-C_6(--C_7--)(C_7-2)------------------------------------------------------------------------------------------------------
%$\mathbf{44) \ C_4-C_6(--C_7--)(C_7-2)}$ \textbf{subgraph:} 
It can be seen that there are $168$ different cases for the relations of a cycle $C_4$, a cycle $C_6$ and two cycles $C_7$ in the subgraph $C_4-C_6(--C_7--)(C_7-2)$. Using GAP \cite{gap}, we see that in $152$ cases of these $168$ cases, the groups which are obtained are solvable, a contradiction. So, there are just $16$ cases which may lead to the existence of a subgraph isomorphic to the graph $C_4-C_6(--C_7--)(C_7-2)$ in $Z(\alpha,\beta)$. It can be seen that each of such cases gives a contradiction and so, the graph $Z(\alpha,\beta)$ contains no subgraph isomorphic to the graph $C_4-C_6(--C_7--)(C_7-2)$.

We note that each group with two generators $h_2$ and $h_3$ and four relations which is one of the $16$ latter cases is a quotient of $B(1,k)$, for some integer $k$, or is a cyclic group.

%------------------------------------------------------------------------------------------------------------------------------------
In the following theorem, we summarize our results about forbidden subgraphs of a zero-divisor graph of length $3$ over $\mathbb{F}_2$ on any torsion-free group. 
\begin{thm}\label{forbiddens}
Suppose that $\alpha$ and $\beta$ are non-zero elements of a group algebra over $\mathbb{F}_2$ on any torsion-free group $G$ such that $|supp(\alpha)|=3$, $\alpha\beta=0$ and if  $\alpha \beta'=0$ for some non-zero element $\beta'$ of the group algebra, then  $|supp(\beta')|\geq |supp(\beta)|$. Then  $Z(\alpha,\beta)$ is a triangle-free graph which contains none of the other $46$ graphs in Table \ref{tab-forbiddens} as a subgraph.
\end{thm}

%---------------------------------------------------------------------------------------------------------------------------------------------------

\section{\bf The possible number of vertices of a zero-divisor graph \\ of length $3$ over  $\mathbb{F}_2$ on any torsion-free group}\label{S3}

By Remark \ref{r-F2}, the number of vertices of a zero-divisor graph of length $3$ over  $\mathbb{F}_2$ on any torsion-free group must be an even positive integer $n\geq4$. Also by Theorem \ref{thm-graph}, such a latter graph is a connected simple cubic one containing no subgraph isomorphic to a triangle. Furthermore, we found $46$ other forbidden subgraphs of such a graph in Sections \ref{S-C4} and \ref{S2}.

Using Sage Mathematics Software \cite{sage} and its package \textit{Nauty-geng}, all non-isomorphic connected cubic triangle-free graphs with the size of the vertex sets $n$ can be found. In this section by using Sage Mathematics Software, we give some results about checking each of the mentioned forbidden subgraphs in all of non-isomorphic connected cubic triangle-free graphs with the size of vertex sets $n \leq 20$. By using such results, we show that $n$ must be greater than or equal to $20$. Also, some results in  the case $n=20$ is given.

Table \ref{tab-graphs} lists all results about the number of non-isomorphic connected cubic triangle-free graphs with the size of vertex sets $n \leq 20$ which contain each of the forbidden subgraphs. The results in this table, from top to bottom, are presented in such a way that by checking each of the forbidden subgraphs in a row, the number of graphs containing these subgraph are omitted from the total number and the existence of the next forbidden subgraph is checked among the remaining ones.

\begin{scriptsize}
\begin{longtable}{|rl|l|l|l|l|l|l|l|l|l|}
\caption{Existence of the forbidden subgraphs in non-isomorphic connected cubic triangle-free graphs with the size of vertex sets $n \leq 20$}\label{tab-graphs}\\
\hline \multicolumn{1}{|r}{$ $} & \multicolumn{1}{l|}{$ $} & \multicolumn{1}{l|}{$n=4$} & \multicolumn{1}{l|}{$n=6$}
& \multicolumn{1}{l|}{$n=8$} & \multicolumn{1}{l|}{$n=10$} & \multicolumn{1}{l|}{$n=12$} & \multicolumn{1}{l|}{$n=14$}
& \multicolumn{1}{l|}{$n=16$} & \multicolumn{1}{l|}{$n=18$} & \multicolumn{1}{l|}{$n=20$}
\\\hline
\endfirsthead
\multicolumn{11}{c}%
{{\tablename\ \thetable{} -- continued from previous page}} \\
\hline \multicolumn{1}{|r}{$ $} & \multicolumn{1}{l|}{$ $} & \multicolumn{1}{l|}{$n=4$} & \multicolumn{1}{l|}{$n=6$}
& \multicolumn{1}{l|}{$n=8$} & \multicolumn{1}{l|}{$n=10$} & \multicolumn{1}{l|}{$n=12$} & \multicolumn{1}{l|}{$n=14$}
& \multicolumn{1}{l|}{$n=16$} & \multicolumn{1}{l|}{$n=18$} & \multicolumn{1}{l|}{$n=20$}
\\\hline
\endhead
\hline \multicolumn{11}{|r|}{{Continued on next page}} \\\hline
\endfoot
\hline %\hline
\endlastfoot
$ $&$\text{Total}$&$0$&$1$&$2$&$6$&$22$&$110$&$792$&$7805$&$97546$\\\hline
$1)$&$K_{2,3}$&$0$&$1$&$0$&$1$&$4$&$22$&$144$&$1222$&$12991$\\
$2)$&$C_4--C_5$&$0$&$0$&$1$&$2$&$6$&$30$&$223$&$2161$&$25427$\\
$3)$&$C_4--C_6$&$0$&$0$&$1$&$1$&$6$&$31$&$223$&$2228$&$28080$\\
$4)$&$C_4-C_5(-C_5-)$&$0$&$0$&$0$&$0$&$2$&$6$&$40$&$319$&$3396$\\
$5)$&$C_4-C_5(-C_4-)$&$0$&$0$&$0$&$1$&$0$&$3$&$12$&$88$&$1123$\\
$6)$&$C_4-C_5(-C_6--)$&$0$&$0$&$0$&$0$&$0$&$4$&$42$&$389$&$4548$\\
$7)$&$C_4-C_5(-C_6-)$&$0$&$0$&$0$&$0$&$0$&$1$&$20$&$382$&$5661$\\
$8)$&$C_4-C_5(-C_7--)$&$0$&$0$&$0$&$0$&$0$&$1$&$10$&$176$&$3172$\\
$9)$&$C_5--C_5(--C_5)$&$0$&$0$&$0$&$1$&$2$&$3$&$18$&$157$&$1617$\\
$10)$&$C_5--C_5(--C_6)$&$0$&$0$&$0$&$0$&$0$&$4$&$32$&$291$&$4289$\\
$11)$&$C_4-C_6(--C_7--)(C_7-1)$&$0$&$0$&$0$&$0$&$1$&$0$&$9$&$64$&$446$\\
$12)$&$C_4-C_6(--C_7--)(--C_5-)$&$0$&$0$&$0$&$0$&$0$&$0$&$0$&$2$&$51$\\
$13)$&$C_4-C_6(-C_6--)(-C_4-)$&$0$&$0$&$0$&$0$&$1$&$0$&$1$&$5$&$35$\\
$14)$&$C_4-C_6(-C_6--)(--C_5-)$&$0$&$0$&$0$&$0$&$0$&$0$&$1$&$1$&$149$\\
$15)$&$C_4-C_6(-C_6--)(C_6---)$&$0$&$0$&$0$&$0$&$0$&$0$&$1$&$30$&$404$\\
$16)$&$C_4-C_6(-C_6--)(---C_4)$&$0$&$0$&$0$&$0$&$0$&$0$&$1$&$4$&$12$\\
$17)$&$C_5-C_5(--C_6--)$&$0$&$0$&$0$&$0$&$0$&$2$&$0$&$41$&$352$\\
$18)$&$C_5-C_5(-C_6--)(C_6---)$&$0$&$0$&$0$&$0$&$0$&$0$&$7$&$47$&$529$\\
$19)$&$C_5-C_5(-C_6--)(--C_6- 1)$&$0$&$0$&$0$&$0$&$0$&$0$&$1$&$31$&$249$\\
$20)$&$C_5-C_5(-C_6--)(-C_5--)$&$0$&$0$&$0$&$0$&$0$&$0$&$1$&$1$&$69$\\
$21)$&$C_6---C_6(C_6---C_6)$&$0$&$0$&$0$&$0$&$0$&$1$&$1$&$8$&$43$\\
$22)$&$C_6---C_6(C_6)(C_6)(C_6)$&$0$&$0$&$0$&$0$&$0$&$0$&$2$&$6$&$25$\\
$23)$&$C_5(--C_6--)C_5(---C_6)$&$0$&$0$&$0$&$0$&$0$&$0$&$1$&$16$&$374$\\
$24)$&$C_6--C_6(C_6--C_6)$&$0$&$0$&$0$&$0$&$0$&$0$&$0$&$29$&$438$\\
$25)$&$C_6---C_6(C_6--C_6)$&$0$&$0$&$0$&$0$&$0$&$0$&$0$&$20$&$505$\\
$26)$&$C_6---C_6(-C_5-)$&$0$&$0$&$0$&$0$&$0$&$0$&$0$&$27$&$721$\\
$27)$&$C_4-C_6(--C_7--)(---C_6)$&$0$&$0$&$0$&$0$&$0$&$0$&$0$&$8$&$257$\\
$28)$&$C_4-C_6(--C_7--)(C_4)(C_4)$&$0$&$0$&$0$&$0$&$0$&$0$&$0$&$2$&$2$\\
$29)$&$C_6---C_6(-C_5--)$&$0$&$0$&$0$&$0$&$0$&$0$&$0$&$5$&$66$\\
$30)$&$C_6--C_6(--C_5-)(-C_5-)$&$0$&$0$&$0$&$0$&$0$&$0$&$0$&$12$&$293$\\
$31)$&$C_6--C_6(--C_5-)(C_6---)$&$0$&$0$&$0$&$0$&$0$&$0$&$0$&$2$&$267$\\
$32)$&$C_5(--C_6--)C_5(C_6)$&$0$&$0$&$0$&$0$&$0$&$0$&$0$&$1$&$43$\\
$33)$&$C_5(--C_6--)C_5(C_7)$&$0$&$0$&$0$&$0$&$0$&$0$&$0$&$2$&$50$\\
$34)$&$C_5-C_5(--C_7--)(--C_5)$&$0$&$0$&$0$&$0$&$0$&$0$&$0$&$6$&$199$\\
$35)$&$C_5-C_5(--C_7--)(-C_5-)$&$0$&$0$&$0$&$0$&$0$&$0$&$0$&$2$&$69$\\
$36)$&$C_5-C_5(-C_6--)(--C_6-2)$&$0$&$0$&$0$&$0$&$0$&$0$&$0$&$2$&$114$\\
$37)$&$C_4-C_4(-C_7-)(C_4)$&$0$&$0$&$0$&$0$&$0$&$1$&$0$&$6$&$72$\\
$38)$&$C_5--C_5(-C_5--)$&$0$&$0$&$0$&$0$&$0$&$0$&$0$&$1$&$11$\\
$39)$&$C_6---C_6(-C_4)$&$0$&$0$&$0$&$0$&$0$&$0$&$0$&$3$&$94$\\
$40)$&$C_6--C_6(C_4)$&$0$&$0$&$0$&$0$&$0$&$0$&$0$&$2$&$67$\\
$41)$&$C_4-C_6(-C_4)(-C_4)$&$0$&$0$&$0$&$0$&$0$&$0$&$0$&$1$&$30$\\
$42)$&$C_4-C_6(--C_7--)(-C_5-)$&$0$&$0$&$0$&$0$&$0$&$0$&$0$&$1$&$26$\\
$43)$&$C_5-C_5(-C_6--)(--C_5-)$&$0$&$0$&$0$&$0$&$0$&$0$&$0$&$1$&$17$\\
$44)$&$C_4-C_6(--C_7--)(C_7-2)$&$0$&$0$&$0$&$0$&$0$&$0$&$0$&$1$&$41$\\
\hline
$ $&$\text{Isomorphic to }L_n$&$0$&$0$&$0$&$0$&$0$&$0$&$1$&$1$&$1$\\
$ $&$\text{Isomorphic to }M_n$&$0$&$0$&$0$&$0$&$0$&$1$&$1$&$1$&$1$\\
$ $&$\text{Remains}$&$0$&$0$&$0$&$0$&$0$&$0$&$0$&$0$&$1120$
\end{longtable}
\end{scriptsize}

The discussion above and the results of Table \ref{tab-graphs} are summarized in the following theorem. 
\begin{thm}\label{T2}
The vertex set size of a zero-divisor graph of length $3$ over  $\mathbb{F}_2$ on any torsion-free group must be greater than or equal to $20$. Furthermore, there are just $1120$ graphs with vertex set size equal to $20$ which may be isomorphic to such latter graphs.
\end{thm}
\begin{cor}\label{maincor}
Let $\alpha$ and $\beta$ be non-zero elements of the group algebra of any torsion-free group over  $\mathbb{F}_2$. If $|supp(\alpha)|=3$ and $\alpha \beta=0$ then $|supp(\beta)|\geq 20$. 
\end{cor}

\section{\bf Possible zero divisors with supports of size $3$ in $\mathbb{F}[G]$}\label{S1-1}
Throughout this section let $\alpha$  be a non-zero element in the group algebra of a torsion-free group $G$ over a field $\mathbb{F}$ such that $|supp(\alpha)|=3$ and  $\alpha \beta=0$ for some non-zero element $\beta$ of the group algebra. It is known that $|supp(\beta)|\geq3$ (see \cite[Theorem 2.1]{pascal}). In this section, we show that $|supp(\beta)|$ must be at least $10$. Here, one may assume that $\beta$ has minimum possible support size among all  elements $\gamma$ with $\alpha \gamma=0$. Therefore by Remark \ref{r-G}, $1 \in supp(\alpha)$ and $G=\langle supp(\alpha) \rangle$. Let  $supp(\alpha)=\{h_1,h_2,h_3\}$, $supp(\beta)=\{g_1,g_2,\ldots,g_n\}$ and $n=|supp(\beta)|$. If $A=\{1,2,3\}\times \{1,2,\ldots,n\}$, then for all $(i,j)\in A$ there must be an $(i',j')\in A$ such that $i\neq i'$, $j\neq j'$ and  $h_ig_j=h_{i'}g_{j'}$ because $\alpha\beta=0$. 

\begin{thm}[Corollary of \cite{kemp}]\label{kemperman}
Let $G$ be an arbitrary group and let $B$ and $C$ be finite non-empty subsets of $G$. Suppose that each non-identity element $g$ of $G$ has a finite or infinite order greater than or equal to $|B|+|C|-1$. Then $|BC|\geq |B|+|C|-1$.
\end{thm}
\begin{thm}[Corollary 11 of \cite{ham}]\label{hamidoune}
If $C$ is a finite generating subset of a nonabelian torsion-free group $G$ such that $1\in C$ and $|C|\geq 4$, then $|BC|\geq |B|+|C|+1$ for all $B\subset G$ with $|B|\geq 3$.
\end{thm}
Abelian torsion-free groups satisfy the Conjecture \ref{conj-zero} (see \cite[Theorem 26.2]{pass2}). So, $G$ must be a nonabelian torsion-free group. Also, $1 \in supp(\alpha)$ and $G=\langle supp(\alpha) \rangle$. Therefore by Theorem \ref{hamidoune}, $3n\geq |supp(\alpha) supp(\beta)|\geq 4+n$.
\begin{thm}[Proposition 4.12 of \cite{dyk}]\label{dyk-rem}
There exist no $\gamma , \delta \in \mathbb{F}_2[G]$ such that $\gamma \delta=1$, where $|supp(\gamma)|=3$ and $|supp(\delta)|\geq 13$ is an odd integer. 
\end{thm}

\begin{enumerate}
\item
Let $n=3$. By Theorem \ref{kemperman}, $3n\geq|supp(\alpha) supp(\beta)|\geq |supp(\alpha)|+|supp(\beta)|-1$ because $G$ is torsion-free. Let $|supp(\beta)|=3$. Then $9\geq|supp(\alpha) supp(\beta)|\geq 5$. Since $9-5=4$, there is an $(i,j)\in A$ such that $h_ig_j\not=h_{i'}g_{j'}$ for all $(i',j')\in A$ where $i\neq i'$ and $j\neq j'$, a contradiction with $\alpha \beta =0$. So, $|supp(\beta)|$ must be at least $4$.
\item
Let $n=4$. Then by Theorem \ref{hamidoune}, $12\geq |supp(\alpha) supp(\beta)|\geq 8$. Since $12-8=4$, there is an $(i,j)\in A$ such that $h_ig_j\not=h_{i'}g_{j'}$ for all $(i',j')\in A$ where $i\neq i'$ and $j\neq j'$, a contradiction with $\alpha \beta =0$. So, $|supp(\beta)|$ must be at least $5$.
\item
Let $n=5$. Then by Theorem \ref{hamidoune}, $15\geq |supp(\alpha) supp(\beta)|\geq 9$. Since $15-9=6$, there is an $(i,j)\in A$ such that $h_ig_j\not=h_{i'}g_{j'}$ for all $(i',j')\in A$ where $i\neq i'$ and $j\neq j'$, a contradiction with $\alpha \beta =0$. So, $|supp(\beta)|$ must be at least $6$.
\item
Let $n=6$. Then by Theorem \ref{hamidoune}, $18\geq |supp(\alpha) supp(\beta)|\geq 10$. Since $18-10=8$, there is an $(i,j)\in A$ such that $h_ig_j\not=h_{i'}g_{j'}$ for all $(i',j')\in A$ where $i\neq i'$ and $j\neq j'$, a contradiction with $\alpha \beta =0$. So, $|supp(\beta)|$ must be at least $7$.
\item
Let $n=7$. Then by Theorem \ref{hamidoune}, $21\geq |supp(\alpha) supp(\beta)|\geq 11$. Since $21-11=10$, there is an $(i,j)\in A$ such that $h_ig_j\not=h_{i'}g_{j'}$ for all $(i',j')\in A$ where $i\neq i'$ and $j\neq j'$, a contradiction with $\alpha \beta =0$. So, $|supp(\beta)|$ must be at least $8$.
\item
Let $n=8$. Then by Theorem \ref{hamidoune}, $24\geq |supp(\alpha) supp(\beta)|\geq 12$. Let $|supp(\alpha) supp(\beta)|> 12$. Then $|supp(\alpha) supp(\beta)|\geq 13$. Since $24-13=11$, there is an $(i,j)\in A$ such that $h_ig_j\not=h_{i'}g_{j'}$ for all $(i',j')\in A$ where $i\neq i'$ and $j\neq j'$, a contradiction with $\alpha \beta =0$. So, $|supp(\alpha) supp(\beta)|=12$ and because $\alpha \beta =0$, there is a partition $\pi$ of $A$ with all sets containing two elements, such that if $(i, j)$ and $(i',j')$ belong to the same set of $\pi$, then $h_ig_j=h_{i'}g_{j'}$. Let $\alpha'=\sum_{a\in supp(\alpha)}{a}$ and $\beta'=\sum_{b\in supp(\beta)}{b}$. So, $\alpha', \beta' \in \mathbb{F}_2[G]$, $|supp(\alpha')|=3$ and $|supp(\beta')|=8$ and with the above discussion we have $\alpha'\beta'=0$, that is a contradiction (see Corollary \ref{maincor}). Therefore, $|supp(\alpha) supp(\beta)|\not=12$ and so $|supp(\beta)|$ must be at least $9$.
\item
Let $n=9$. Then by Theorem \ref{hamidoune}, $27\geq |supp(\alpha) supp(\beta)|\geq 13$. Let $|supp(\alpha) supp(\beta)|> 13$. Then $|supp(\alpha) supp(\beta)|\geq 14$. Since $27-14=13$, there is an $(i,j)\in A$ such that $h_ig_j\not=h_{i'}g_{j'}$ for all $(i',j')\in A$ where $i\neq i'$ and $j\neq j'$, a contradiction with $\alpha \beta =0$. So, $|supp(\alpha) supp(\beta)|=13$ and because $\alpha \beta =0$, there is a partition $\pi$ of $A$ with one set of size $3$ and all other sets containing two elements, such that if $(i, j)$ and $(i',j')$ belong to the same set of $\pi$, then $h_ig_j=h_{i'}g_{j'}$. With the discussion above, $\left(\sum_{a\in supp(\alpha)}{a}\right)\left(\sum_{b\in supp(\beta)}{b}\right)x^{-1}=1$ where $x=h_ig_j$ for some $(i,j)$ belongs to the set  of size $3$ in $\pi$. Hence, there are $\gamma , \delta \in \mathbb{F}_2[G]$ such that $\gamma \delta=1$, where $\gamma=\sum_{a\in supp(\alpha)}{a}$, $\delta=\sum_{b\in supp(\beta)}{bx^{-1}}$, $|supp(\gamma)|=3$ and $|supp(\delta)|=|supp(\beta)|=9$, that is a contradiction with Theorem \ref{dyk-rem}. Therefore, $|supp(\alpha) supp(\beta)|\not=13$ and so $|supp(\beta)|$ must be at least $10$.
\end{enumerate}
With the discussion above, we have the following theorem.
\begin{thm}\label{main-F}
Let $\alpha$  be a non-zero element in the group algebra of a torsion-free group over an arbitrary field such that $|supp(\alpha)|=3$ and  $\alpha \beta=0$, for some non-zero element $\beta$ of the group algebra. Then $|supp(\beta)|\geq 10$. 
\end{thm}

\begin{prop}\label{supp2}
If $\mathbb{F}[G]$ has no non-zero element $\gamma$ with $|supp(\gamma)|\leq k$ such that $\gamma^2=0$, then there exist no non-zero elements $\gamma_1,\gamma_2 \in \mathbb{F}[G]$ such that $\gamma_1\gamma_2=0$ and $|supp(\gamma_1)||supp(\gamma_2)|\leq k$.
\end{prop}
\begin{proof}
Suppose, for a contradiction, that $\gamma_1,\gamma_2 \in \mathbb{F}[G]\setminus \{0\}$ such that $\gamma_1\gamma_2=0$ and $|supp(\gamma_1)||supp(\gamma_2)|\leq k$. We may assume that $1\in supp(\gamma_1) \cap supp(\gamma_2)$, since $(a^{-1} \gamma_1 ) (\gamma_2 b^{-1})=0$ for any $a\in supp(\gamma_1)$ and $b\in supp(\gamma_2)$. 

Suppose, for a contradiction, that $\gamma_2 x \gamma_1=0$ for all $x\in G$. Then it follows from 
\cite[Lemma 1.3]{pass2} that $\theta(\gamma_2) \theta(\gamma_1)=0$, where $\theta$ is the projection $\theta:\mathbb{F}[G]\rightarrow \mathbb{F}[\Delta]$ given by 
$\beta= \sum_{x\in G} f_x x \mapsto \theta(\beta)=\sum_{x\in \Delta} f_x x$, where $\Delta$ is the subgroup of all
elements of $G$ having a finite number of conjugates in $G$ (see \cite[p. 3]{pass2}). Now it follows from \cite[Lemma 2.2]{pass2} and \cite[Lemma 2.4]{pass2} that $\theta(\gamma_1)=0$ or $\theta(\gamma_2)=0$, which are both contradiction since  $1\in supp(\gamma_1) \cap supp(\gamma_2)$.  
Therefore, there exists an element $x\in G$ such that $\beta=\gamma_2 x \gamma_1\not=0$. Now    
$$\beta^2=(\gamma_2 x \gamma_1)^2=\gamma_2 x \gamma_1 \gamma_2 x \gamma_1=0$$
and $$|supp(\beta)|\leq |supp(\gamma_2)||supp(x \gamma_1)|=|supp(\gamma_2)||supp(\gamma_1)|\leq k,$$
which is a contradiction. This completes the proof. 
\end{proof}

\section{\bf Possible units with supports of size $3$ in $\mathbb{F}[G]$}\label{S-unit}
It is known that a group algebra over any torsion-free group does not contain a unit element whose support is of size at most $2$ (see \cite[Theorem 4.2]{dyk}), but it is not  known a similar result for group algebra elements with the supports of size $3$. 

Throughout this section let $\gamma$  be an element in the group algebra of a torsion-free group $G$ over a field $\mathbb{F}$ such that $|supp(\gamma)|=3$ and  $\gamma \delta=1$ for some element $\delta$ of the group algebra. In this section, we show that $|supp(\delta)|$ must be at least $9$. Here, one may assume that $\delta$ has minimum possible support size among all  elements $\alpha$ with $\gamma \alpha=1$. Therefore by Remark \ref{Unit-r-G}, $1 \in supp(\gamma)$ and $G=\langle supp(\gamma) \rangle$. Let  $supp(\gamma)=\{h_1,h_2,h_3\}$, $supp(\delta)=\{g_1,g_2,\ldots,g_n\}$ and $n=|supp(\delta)|$. 

Suppose that $A=\{1,2,3\}\times \{1,2,\ldots,n\}$. Since $\gamma \delta =1$, there must be at least one $(i,j)\in A$ such that $h_ig_j=1$. By renumbering, we may assume that $(i,j)=(1,1)$. Replacing $\gamma$ by $h_1^{-1}\gamma$ and $\delta$ by $\delta g_1^{-1}$ we may assume that $h_1=g_1=1$. 

There is a partition $\pi$ of $A$ such that $(i, j)$ and $(i',j')$ belong to the same set of $\pi$ if and only if $h_ig_j=h_{i'}g_{j'}$ and because of the relation $\gamma \delta =1$, for all $E\in \pi$ we have 
\begin{equation}\label{e-u}
\displaystyle\sum_{(i,j)\in E}{\gamma_i\delta_j}=\left\{
\begin{array}{lr}
1 &(1,1)\in E\\
0 &(1,1)\notin E\\
\end{array} \right.
\end{equation}
Let $E_1$ be the set in $\pi$ which contains $(1,1)$. 

\subsection{The support of $\delta$ is of size at least $4$}

By Theorem \ref{kemperman}, $3n\geq|supp(\gamma) supp(\delta)|\geq |supp(\gamma)|+|supp(\delta)|-1$ because $G$ is torsion-free. Let $n=3$. Then with the above discussion, $9\geq |supp(\gamma) supp(\delta)|\geq 5$ and so, there are at least $4$ sets different from $E_1$ in $\pi$, namely $E_2,E_3,E_4,E_5$. Since $\gamma \delta =1$, each of such sets must have at least two elements such that $\sum_{(i,j)\in E_k}{\gamma_i\delta_j}=0$ for all $k\in \{2,3,4,5\}$, but since $9-5=4$, $|E_1|=1$ and $|E_k|=2$ for all $k\in \{2,3,4,5\}$. Therefore $E_1=\{(1,1)\}$ and for all $k\in \{2,3,4,5\}$, $E_k=\{(i,j),(i',j')\}$ where $h_ig_j=h_{i'}g_{j'}$ for some $(i,j),(i',j')\in A$ such that $i\not=i'$ and $j\not=j'$. Let $\gamma'=\sum_{a\in supp(\gamma)}{a}$ and $\delta'=\sum_{b\in supp(\delta)}{b}$. So, $\gamma', \delta' \in \mathbb{F}_2[G]$, $|supp(\gamma')|=3$ and $|supp(\delta')|=3$ and with the above discussion we have $\gamma'\delta'=1$, that is a contradiction with Theorem \ref{dyk-rem}. Therefore, $n\not=3$.

\subsection{The support of $\delta$ must be of size greater than or equal to $8$}\label{unitsub-2}
Abelian torsion-free groups satisfy the Conjecture \ref{conj-unit} (see \cite[Theorem 26.2]{pass2}). So, $G$ must be a nonabelian torsion-free group. Therefore  by Theorem \ref{hamidoune}, if $n\geq 4$, then $|supp(\gamma) supp(\delta)|\geq |supp(\gamma)|+|supp(\delta)|+1$. Also, it is easy to see that $|supp(\gamma)| |supp(\delta)|\geq|supp(\gamma) supp(\delta)|$. Hence, $3n\geq |supp(\gamma) supp(\delta)|\geq 4+n$.
\begin{enumerate}
\item
Let $n=4$. Then with the above discussion, $12\geq |supp(\gamma) supp(\delta)|\geq 8$ and so, there are at least $7$ sets different from $E_1$ in $\pi$, namely $E_2,E_3,\ldots,E_8$. Since $\gamma \delta =1$, each of such sets must have at least two elements such that $\sum_{(i,j)\in E_k}{\gamma_i\delta_j}=0$ for all $k\in \{2,3,\ldots,8\}$, but since $12-8=4$, $|E_k|\leq 1$ for some $k\in \{2,3,\ldots,8\}$, a contradiction. Therefore, $n\not=4$.
\item
Let $n=5$. Then with the discussion above $15\geq |supp(\gamma) supp(\delta)|\geq 9$ and so, there are at least $8$ sets different from $E_1$ in $\pi$, namely $E_2,E_3,\ldots,E_9$. Since $\gamma \delta =1$, each of such sets must have at least two elements such that $\sum_{(i,j)\in E_k}{\gamma_i\delta_j}=0$ for all $k\in \{2,3,\ldots,9\}$, but since $15-9=6$, $|E_k|\leq 1$ for some $k\in \{2,3,\ldots,9\}$, a contradiction. Therefore, $n\not=5$.
\item
Let $n=6$. Then with the discussion above $18\geq |supp(\gamma) supp(\delta)|\geq 10$ and so, there are at least $9$ sets different from $E_1$ in $\pi$, namely $E_2,E_3,\ldots,E_{10}$. Since $\gamma \delta =1$, each of such sets must have at least two elements such that $\sum_{(i,j)\in E_k}{\gamma_i\delta_j}=0$ for all $k\in \{2,3,\ldots,10\}$, but since $18-10=8$, $|E_k|\leq 1$ for some $k\in \{2,3,\ldots,10\}$, a contradiction. Therefore, $n\not=6$.
\item
Let $n=7$. Then with the discussion above $21\geq |supp(\gamma) supp(\delta)|\geq 11$ and so, there are at least $10$ sets different from $E_1$ in $\pi$, namely $E_2,E_3,\ldots,E_{11}$. Since $\gamma \delta =1$, each of such sets must have at least two elements such that $\sum_{(i,j)\in E_k}{\gamma_i\delta_j}=0$ for all $k\in \{2,3,\ldots,11\}$, but since $21-11=10$, $|E_1|=1$ and $|E_k|=2$ for all $k\in \{2,3,\ldots,11\}$. Therefore $E_1=\{(1,1)\}$ and for all $k\in \{2,3,\ldots,11\}$, $E_k=\{(i,j),(i',j')\}$ where $h_ig_j=h_{i'}g_{j'}$ for some $(i,j),(i',j')\in A$ such that $i\not=i'$ and $j\not=j'$. Let $\gamma'=\sum_{a\in supp(\gamma)}{a}$ and $\delta'=\sum_{b\in supp(\delta)}{b}$. So, $\gamma', \delta' \in \mathbb{F}_2[G]$, $|supp(\gamma')|=3$ and $|supp(\delta')|=7$ and with the above discussion we have $\gamma'\delta'=1$, that is a contradiction with Theorem \ref{dyk-rem}. Therefore, $n\not=7$.
\item
Let $n=8$. Then with the discussion above $24\geq |supp(\gamma) supp(\delta)|\geq 12$. Let $|supp(\gamma) supp(\delta)|> 12$. Then $|supp(\gamma) supp(\delta)|\geq 13$ and so, there are at least $12$ sets different from $E_1$ in $\pi$, namely $E_2,E_3,\ldots,E_{13}$. Since $\gamma \delta =1$, each of such sets must have at least two elements such that $\sum_{(i,j)\in E_k}{\gamma_i\delta_j}=0$ for all $k\in \{2,3,\ldots,13\}$, but since $24-13=11$, $|E_k|\leq 1$ for some $k\in \{2,3,\ldots,13\}$, a contradiction. So, $|supp(\gamma) supp(\delta)|=12$. Therefore, there are $11$ sets different from $E_1$ in $\pi$, namely $E_2,E_3,\ldots,E_{12}$. Since $\gamma \delta =1$, there are two cases for the number of elements in such sets. 
\begin{enumerate}
\item
$|E_1|=2$ and for all $k\in \{2,3,\ldots,12\}$, $E_k=\{(i,j),(i',j')\}$ where $h_ig_j=h_{i'}g_{j'}$ for some $(i,j),(i',j')\in A$ such that $i\not=i'$ and $j\not=j'$. Let $\gamma'=\sum_{a\in supp(\gamma)}{a}$ and $\delta'=\sum_{b\in supp(\delta)}{b}$. So, $\gamma', \delta' \in \mathbb{F}_2[G]$, $|supp(\gamma')|=3$ and $|supp(\delta')|=8$ and with the above discussion we have $\gamma'\delta'=0$, that is a contradiction (see Corollary \ref{maincor}). 
\item
$E_1=\{(1,1)\}$ and $|E_l|=3$ for exactly one $l\in \{2,3,\ldots,12\}$ and for all $k\in \{2,3,\ldots,12\}\setminus \{l\}$, $E_k=\{(i,j),(i',j')\}$ where $h_ig_j=h_{i'}g_{j'}$ for some $(i,j),(i',j')\in A$ such that $i\not=i'$ and $j\not=j'$.
\end{enumerate}
\end{enumerate}

\begin{thm}\label{main-unit1}
Let $\gamma$ and $\delta$ be elements of the group algebra of any torsion-free group over an arbitrary field. If $|supp(\gamma)|=3$ and $\gamma \delta =1$ then $|supp(\delta)|\geq 8$. 
\end{thm}
\subsection{The support of $\delta$ must be of size greater than or equal to $9$}
In the following, we focus on the unit graph of $\gamma$ and $\delta$ over $\mathbb{F}$, $U(\gamma,\delta)$. By Proposition \ref{simple-unit}, $U(\gamma,\delta)$ is simple. Also by Theorems \ref{U-K3-K3} and \ref{U-K2,3},  $U(\gamma,\delta)$ contains no $C_3-C_3$ or $K_{2,3}$ as a subgraph.

By item $(5)$ of Subsection \ref{unitsub-2}, if $n=8$, then $|supp(\gamma) supp(\delta)|=12$ and there are exactly $11$ sets different from $E_1$ in $\pi$, namely $E_2,E_3,\ldots,E_{12}$. Also, $E_1=\{(1,1)\}$ and $|E_l|=3$ for exactly one $l\in \{2,3,\ldots,12\}$, and for all $k\in \{2,3,\ldots,12\}\setminus \{l\}$, $E_k=\{(i,j),(i',j')\}$ such that $i\not=i'$, $j\not=j'$ and $h_ig_j=h_{i'}g_{j'}$.

Let $E_l=\{(i_1,j_1),(i_2,j_2),(i_3,j_3)\}$. Therefore, $h_{i_1}g_{j_1}=h_{i_2}g_{j_2}=h_{i_3}g_{j_3}$ and so there is a triangle in $U(\gamma,\delta)$ with the vertex set $\{g_{j_1},g_{j_2},g_{j_3}\}$ and there is no other triangle in the latter graph. Let $(2,1),(3,1)\notin E_l$. Then by the way we have chosen $E_k$ for $k\in \{1,2,3,\ldots,12\}$, the degree of $g_{j_1}$, $g_{j_2}$ and $g_{j_3}$ are equal to $4$. So, there must be $6$ other vertices different from $g_{j_1}$, $g_{j_2}$ and $g_{j_3}$ in the vertex set of $U(\gamma,\delta)$ because there is no other triangle in the latter graph. This gives a contradiction because the size of the vertex set of $U(\gamma,\delta)$ is $n=8$. Hence, $E_l=\{(a,1),(i,j),(i',j')\}$ where $a\in\{2,3\}$ and $\{h_a,h_i,h_{i'}\}=supp(\gamma)$. Since $|E_1|=1$, $1=h_1g_1\not=h_mg_n$ for all $(m,n)\in A\setminus E_1$. So, ${\rm deg}(g_1)=3$ and by renumbering, we may assume that $U(\gamma,\delta)$ has the graph $H$ in Figure \ref{f-H} as a subgraph and there is no other vertex in $U(\gamma,\delta)$. In $H$, $g_5\sim g_2$ or $g_5\not\sim g_2$.

\begin{figure}[ht]
\centering
\psscalebox{1.0 1.0} % Change this value to rescale the drawing.
{
\begin{pspicture}(0,-2.305)(5.52,2.305)
\psdots[linecolor=black, dotsize=0.4](2.8,0.495)
\psdots[linecolor=black, dotsize=0.4](2.0,-0.705)
\psdots[linecolor=black, dotsize=0.4](3.6,-0.705)
\psdots[linecolor=black, dotsize=0.4](2.8,1.695)
\psdots[linecolor=black, dotsize=0.4](0.8,-0.705)
\psdots[linecolor=black, dotsize=0.4](1.6,-1.905)
\psdots[linecolor=black, dotsize=0.4](4.0,-1.905)
\psdots[linecolor=black, dotsize=0.4](4.8,-0.705)
\psline[linecolor=black, linewidth=0.04](2.8,0.495)(2.0,-0.705)(3.6,-0.705)(2.8,0.495)(2.8,1.695)
\psline[linecolor=black, linewidth=0.04](2.0,-0.705)(1.6,-1.905)
\psline[linecolor=black, linewidth=0.04](0.8,-0.705)(2.0,-0.705)
\psline[linecolor=black, linewidth=0.04](4.0,-1.905)(3.6,-0.705)(4.8,-0.705)
\rput[bl](3.04,0.495){$g_1$}
\rput[bl](2.8,1.92){$g_2$}
\rput[bl](3.6,-0.5){$g_3$}
\rput[bl](1.6,-0.5){$g_4$}
\rput[bl](0.2,-0.705){$g_5$}
\rput[bl](1.2,-2.305){$g_6$}
\rput[bl](4.1,-2.305){$g_7$}
\rput[bl](5.03,-0.705){$g_8$}
\end{pspicture}
}
\caption{The subgraph $H$ of $U(\gamma,\delta)$ for the case that $n=8$}\label{f-H}
\end{figure}
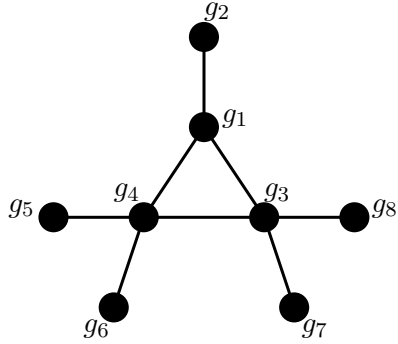

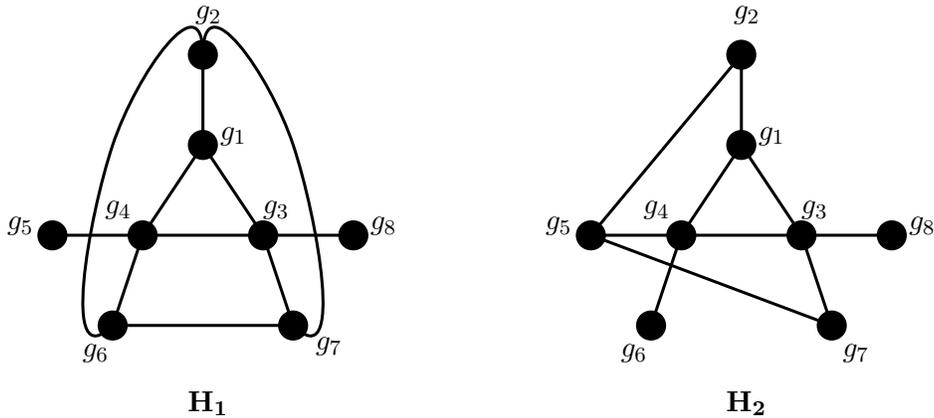
\begin{figure}[h]
\centering
\psscalebox{1.0 1.0} % Change this value to rescale the drawing.
{
\begin{pspicture}(0,-2.705)(5.52,2.705)
\psdots[linecolor=black, dotsize=0.4](2.8,0.895)
\psdots[linecolor=black, dotsize=0.4](2.0,-0.305)
\psdots[linecolor=black, dotsize=0.4](3.6,-0.305)
\psdots[linecolor=black, dotsize=0.4](2.8,2.095)
\psdots[linecolor=black, dotsize=0.4](0.8,-0.305)
\psdots[linecolor=black, dotsize=0.4](1.6,-1.505)
\psdots[linecolor=black, dotsize=0.4](4.0,-1.505)
\psdots[linecolor=black, dotsize=0.4](4.8,-0.305)
\psline[linecolor=black, linewidth=0.04](2.8,0.895)(2.0,-0.305)(3.6,-0.305)(2.8,0.895)(2.8,2.095)
\psline[linecolor=black, linewidth=0.04](2.0,-0.305)(1.6,-1.505)
\psline[linecolor=black, linewidth=0.04](0.8,-0.305)(2.0,-0.305)
\psline[linecolor=black, linewidth=0.04](4.0,-1.505)(3.6,-0.305)(4.8,-0.305)
\rput[bl](3.03,0.895){$g_1$}
\rput[bl](2.7,2.5){$g_2$}
\rput[bl](3.6,-0.095){$g_3$}
\rput[bl](1.5,-0.095){$g_4$}
\rput[bl](0.2,-0.305){$g_5$}
\rput[bl](1.2,-2){$g_6$}
\rput[bl](4.3,-1.905){$g_7$}
\rput[bl](5.03,-0.305){$g_8$}
\psbezier[linecolor=black, linewidth=0.04](4.0,-1.505)(4.7071066,-2.2121067)(4.316228,-0.0536833)(4.0,0.895)(3.6837723,1.8436832)(2.8,3.095)(2.8,2.095)
\psbezier[linecolor=black, linewidth=0.04](1.6,-1.505)(0.8928932,-2.2121067)(1.2837722,-0.0536833)(1.6,0.895)(1.9162278,1.8436832)(2.8,3.095)(2.8,2.095)
\psline[linecolor=black, linewidth=0.04](1.6,-1.505)(4.0,-1.505)
\rput[bl](2.6,-2.705){$\mathbf{H_1}$}
\end{pspicture}
} 
\hspace{2.5cc}
\psscalebox{1.0 1.0} % Change this value to rescale the drawing.
{
\begin{pspicture}(0,-2.705)(5.52,2.705)
\psdots[linecolor=black, dotsize=0.4](2.8,0.895)
\psdots[linecolor=black, dotsize=0.4](2.0,-0.305)
\psdots[linecolor=black, dotsize=0.4](3.6,-0.305)
\psdots[linecolor=black, dotsize=0.4](2.8,2.095)
\psdots[linecolor=black, dotsize=0.4](0.8,-0.305)
\psdots[linecolor=black, dotsize=0.4](1.6,-1.505)
\psdots[linecolor=black, dotsize=0.4](4.0,-1.505)
\psdots[linecolor=black, dotsize=0.4](4.8,-0.305)
\psline[linecolor=black, linewidth=0.04](2.8,0.895)(2.0,-0.305)(3.6,-0.305)(2.8,0.895)(2.8,2.095)
\psline[linecolor=black, linewidth=0.04](2.0,-0.305)(1.6,-1.505)
\psline[linecolor=black, linewidth=0.04](0.8,-0.305)(2.0,-0.305)
\psline[linecolor=black, linewidth=0.04](4.0,-1.505)(3.6,-0.305)(4.8,-0.305)
\rput[bl](3.03,0.895){$g_1$}
\rput[bl](2.7,2.5){$g_2$}
\rput[bl](3.6,-0.095){$g_3$}
\rput[bl](1.5,-0.095){$g_4$}
\rput[bl](0.2,-0.305){$g_5$}
\rput[bl](1.2,-2){$g_6$}
\rput[bl](4.15,-2){$g_7$}
\rput[bl](5.03,-0.305){$g_8$}
\psline[linecolor=black, linewidth=0.04](2.8,2.095)(0.8,-0.305)(4.0,-1.505)
\rput[bl](2.6,-2.705){$\mathbf{H_2}$}
\end{pspicture}
}
\caption{Two possible subgraphs of $U(\gamma,\delta)$ for the case that $n=8$}\label{f-H1-H2}
\end{figure}

Let $g_5\not\sim g_2$. Since $(a,2),(a,5),(a,6),(a,7),(a,8) \notin E_l,E_1$ for all $a\in \{1,2,3\}$ and $|E_k|=2$ for all  $k\in \{2,3,\ldots,12\}\setminus \{l\}$, ${\rm deg}(g_2)={\rm deg}(g_5)={\rm deg}(g_6)={\rm deg}(g_7)={\rm deg}(g_8)=3$. Since $U(\gamma,\delta)$ is a simple graph which contains exactly one triangle, $g_6\not\sim g_i$ for all $i\in \{1,3,4,5,6\}$. So, without loss of generality we may assume that $g_6\sim g_7$ since ${\rm deg}(g_6)=3$.  Suppose, for a contradiction, that $g_6\sim g_8$. Then there is a subgraph isomorphic to $K_{2,3}$ in $U(\gamma,\delta)$ with the vertex set $\{g_3,g_4,g_6,g_7,g_8\}$, a contradiction. So, $g_6\sim g_2$. With a same discussion we have $g_7\sim g_2$ and $U(\gamma,\delta)$ has the graph $H_1$ in Figure \ref{f-H1-H2} as a subgraph. Since ${\rm deg}(g_2)={\rm deg}(g_5)={\rm deg}(g_6)={\rm deg}(g_7)={\rm deg}(g_8)=3$ and ${\rm deg}(g_1)={\rm deg}(g_3)={\rm deg}(g_4)=4$, we must have $g_5\sim g_8$ with a double edge, a contradiction because $U(\gamma,\delta)$ is simple. Therefore, $g_5\sim g_2$.

Since $U(\gamma,\delta)$ is a simple graph which contains exactly one triangle, $g_5\not\sim g_i$ for all $i\in \{1,2,3,4,5,6\}$. So, $g_5\sim g_7$ or $g_5\sim g_8$ because ${\rm deg}(g_5)=3$. As we can see in Figure \ref{f-H}, without loss of generality we may assume that $g_5\sim g_7$ and $U(\gamma,\delta)$ has the graph $H_2$ in Figure \ref{f-H1-H2} as a subgraph. Since $U(\gamma,\delta)$ is a simple graph which contains exactly one triangle, $g_6\not\sim g_i$ for all $i\in \{1,3,4,5,6\}$. So, $g_6\sim g_2$, $g_6\sim g_7$ or $g_6\sim g_8$ because ${\rm deg}(g_6)=3$. If $g_6\sim g_2$ or $g_6\sim g_7$ then there is a subgraph isomorphic to $K_{2,3}$ in $U(\gamma,\delta)$ as we can see in the graph $H_2$ of Figure \ref{f-H1-H2}, a contradiction. Therefore we must have $g_6\sim g_8$ with a double edge, a contradiction because $U(\gamma,\delta)$ is simple. 

Hence with the above discussion, $n\not=8$ and by Theorem \ref{main-unit1}, we have the following result.
\begin{thm}\label{main-unit2}
Let $\gamma$ and $\delta$ be elements of the group algebra of any torsion-free group over an arbitrary field. If $|supp(\gamma)|=3$ and $\gamma \delta =1$ then $|supp(\delta)|\geq 9$. 
\end{thm}

\section*{\bf Acknowledgements}
The authors are grateful to the referee for his/her valuable suggestions and comments. The first author was supported in part by Grant No. 95050219 from School of Mathematics, Institute for Research in Fundamental Sciences (IPM). The first author was additionally financially supported by the Center of Excellence for Mathematics at the University of Isfahan.

%-----------------------------------------------------------------------------------------------------------------------------------------------------

%-----------------------------------------------------------------------------------------------------------------------------------------------------
\vspace{2.5cc}
\noindent Alireza Abdollahi \\
Department of Mathematics,\\
University of Isfahan,\\
Isfahan 81746-73441\\
Iran;\\
and \\
School of Mathematics,\\ Institute for Research in Fundamental Sciences (IPM), \\ P.O. Box 19395-5746, Tehran,\\ Iran \\
E-mail: a.abdollahi@math.ui.ac.ir
\\
\\
\noindent Zahra Taheri \\
Department of Mathematics,\\
University of Isfahan,\\
Isfahan 81746-73441\\
Iran\\
E-mail: zahra.taheri@sci.ui.ac.ir

\end{document}